\documentclass[fleqn,11pt]{elsarticle}
\usepackage{graphicx} 
\usepackage[utf8]{inputenc}
\usepackage{natbib}
\usepackage{graphicx}
\usepackage{mathdots}
\usepackage{hyperref}
\usepackage{float}
\usepackage{amsmath,amsthm,amssymb,mathrsfs}
\usepackage{geometry} 
\usepackage{centernot} 
\usepackage{stmaryrd}
\usepackage{graphicx}
\usepackage{tikz} 
\usetikzlibrary{shapes,decorations}
\usepackage{caption} 
\usepackage{subcaption} 
\usepackage{mdframed} 
\usepackage{upquote}
\usepackage{multirow}
\usepackage{booktabs}
\usepackage{xspace}
\usepackage{hyperref}
\usepackage{bm}
\usepackage[normalem]{ulem} 

\graphicspath{{FIGS/}}

\biboptions{sort&compress} 

\newtheorem{lem}{Lemma}[section]
\newtheorem{alg}{Algorithm}[section]
\newtheorem{thm}{Theorem}[section]
\newtheorem{rmk}{Remark}[section]

\newcommand{\bena}{\begin{eqnarray}\begin{array}{l}}
\newcommand{\eena}{\end{array}\end{eqnarray}}
\newcommand{\benas}{\begin{eqnarray*}\begin{array}{l}}
\newcommand{\eenas}{\end{array}\end{eqnarray*}}
\newcommand{\benall}{\begin{eqnarray}\begin{array}{ll}}
\newcommand{\eenall}{\end{array}\end{eqnarray}}
\newcommand{\besn}{\begin{subnumcases}}
\newcommand{\eesn}{\end{subnumcases}}
\newcommand{\ben}{\begin{eqnarray}}
\newcommand{\een}{\end{eqnarray}}
\newcommand{\bea}{\begin{array}}
\newcommand{\eea}{\end{array}}
\newcommand{\bes}{\begin{subequations}}
\newcommand{\ees}{\end{subequations}}
\newcommand{\bec}{\begin{cases}}
\newcommand{\eec}{\end{cases}}
\newcommand{\bef}{\begin{figure}[H]\centering}
\newcommand{\eef}{\end{figure}}
\newcommand{\bet}{\begin{tikzpicture}}
\newcommand{\eet}{\end{tikzpicture}}
\newcommand{\beq}{\begin{equation}}
\newcommand{\eeq}{\end{equation}}
\newcommand{\bep}{\begin{proof}}
\newcommand{\eep}{\end{proof}}
\newcommand{\bei}{\begin{itemize}}
\newcommand{\eei}{\end{itemize}}
\newcommand{\beu}{\begin{enumerate}}
\newcommand{\eeu}{\end{enumerate}}
\def\beg#1\eeg{\begin{align}#1\end{align}}
\def\begd#1\eegd{\bena \begin{aligned}#1\end{aligned}\eena}
\def\besp#1\eesp{\begin{split}#1\end{split}}
\def\besl#1\eesl{\begin{subequations}\begin{align}#1\end{align}\end{subequations}} 

\def\inc(#1){\includegraphics[width=0.5\linewidth]{#1}}

\def\bG{\ensuremath{{\bf G}}}

\def\bx{\ensuremath{{\bf x}}}

\begin{document}
\begin{frontmatter}
	\title{Energy-Stable Swarm-Based Inertial Algorithms for Optimization}
	\author{Xuelong Gu$^{a}$}
	\author{Qi Wang$^{a,*}$}

	\address[1]{Department of Mathematics, University of South Carolina, Columbia, SC, 29208, USA\\ \vspace{-1cm}}

	\begin{abstract}
		We formulate the swarming optimization problem as a weakly coupled, dissipative dynamical system governed by a controlled energy dissipation rate and initial velocities that adhere to the nonequilibrium Onsager principle. In this framework, agents' inertia, positions, and masses are dynamically coupled. To numerically solve the system, we develop a class of efficient, energy-stable algorithms that either preserve or enhance energy dissipation at the discrete level. At equilibrium, the system tends to converge toward one of the lowest local minima explored by the agents, thereby improving the likelihood of identifying the global minimum. Numerical experiments confirm the effectiveness of the proposed approach, demonstrating significant advantages over traditional swarm-based gradient descent methods, especially when operating with a limited number of agents.
	\end{abstract}


\end{frontmatter}

\begin{figure}[b]
	\small \baselineskip=10pt
	\rule[2mm]{1.8cm}{0.2mm} \par
	$^{*}$Corresponding author.\\
	E-mail address: QWANG@math.sc.edu (Q. Wang).
\end{figure}

\section{Introduction}

The global optimization over non-convex landscapes associated with non-convex objective functions continues to be a critical and challenging problem in computational science and engineering, with significant implications for disciplines ranging from materials science to artificial intelligence and machine learning \cite{NocedalWright2006, Sra2011, Wales2004}. Traditional optimization methods, particularly those based on gradient descent, have been widely employed in practice, especially in machine learning, due to their simplicity and well-understood convergence properties. However, these methods are inherently local and often trapped in suboptimal minima when confronted with complex, nonconvex objective functions. Over the past few decades, numerous intelligent optimization methods, such as particle swarm optimization \cite{Grassi2023,Kennedy1995}, ant colony optimization \cite{Dorigo1996,Yang2010}, consensus-based methods \cite{Carrillo2018,Pinnau2017}, and others \cite{Borghi2023,Chen2020,vanLaarhoven1987}, have been developed to address these limitations. Despite their success in promoting global exploration of objective functions' landscape, these methods often face difficulties in balancing the trade-off between rapid local convergence and extensive global search, particularly when the underlying problem exhibits intricate landscapes in the objective functions.

Recent developments in swarm-based gradient descent (SBGD) methods \cite{Lu2024SwarmGD, SwarmGDRandom, SwarmGDAnnel} have sought to ameliorate these challenges by endowing individual agents not merely with positional data but also with a dynamic weight or ``mass'' that encapsulates their relative significance within the swarm. Within these paradigms, agents communicate and adaptively modulate their step sizes: those bearing greater mass, deemed ``heavier'', typically adopt reduced time steps, thereby converging more swiftly to proximate local minimizers, whereas ``lighter'' agents maintain "sufficient momentum" to traverse more expansive regions of the search space. Nonetheless, SBGD approaches continue to encounter limitations, particularly regarding their capacity to finely modulate agent inertia and dynamically redistribute mass in a manner that consistently augments global search efficacy without compromising convergence and stability.

The inertial algorithm for global optimization leverages momentum-based dynamics to enhance the efficiency of optimization processes, particularly in high-dimensional and non-convex landscapes. Unlike traditional gradient-based methods that may stagnate in local minima, inertial approaches incorporate acceleration terms that help escape shallow traps and facilitate convergence toward the global minimum. These algorithms are often inspired by physical principles, such as nonequilibrium thermodynamical principles (i.e., the Onsager principle \cite{Wang-2020}) for dissipative dynamical systems, where an agent's motion is governed by inertia, damping, and external forces derived from an objective function \cite{Polyak1964}. One prominent example is the heavy-ball method, which introduces a velocity-dependent term to smooth the optimization trajectory and prevent oscillations \cite{Polyak1964}. More advanced formulations, including Nesterov’s accelerated gradient method \cite{Nesterov1983} and second-order inertial systems \cite{Attouch2000}, strategically adjust the dissipation rate to balance exploration and convergence. These methods have been further extended in various applications, such as machine learning \cite{Sutskever2013}, physics-based optimization \cite{Wibisono2016}, and engineering design \cite{Deb2012}. By leveraging inertia, these algorithms achieve faster convergence rates and improved robustness, making them particularly effective for global optimization problems in diverse scientific and engineering domains.

Motivated by these challenges and advances, we present a novel swarm-based inertial (SBI) algorithm that seamlessly integrates agent communication with the foundational principles of nonequilibrium thermodynamics. Based on the generalized Onsager principle \cite{Wang-2020}, we reformulate the global optimization problem as a weakly coupled dissipative system among the dynamics of agents, where each agent is endowed with a total mechanical energy comprising contributions from both the kinetic and potential energy, specifically defined by
\begin{equation}\label{eq:intro-energy}
	E_i(\mathbf{x}) = \frac{m_i(\mathbf{x}) + \epsilon}{2}  \|\dot{\mathbf{x}}_i\|^2 + w_i F(\mathbf{x}_i), \quad i = 1, \cdots, N,
\end{equation}
where $F$ denotes the objective function perceived as the potential energy here, $\epsilon > 0$ is sufficiently small to ensure a positive lower bound on the agent's mass $m_i(\mathbf{x})$, and $w_i$ are weighting factors to balance kinetic energy and potential contributions. We denote  $\mathbf{v}_i = \dot{\mathbf{x}}_i$. By differentiating the aforementioned energy expression with respect to time and imposing an energy dissipation rate through a friction operator/parameter $R>0$, we obtain the following dynamical system for mass dynamics:
\begin{equation}\label{eq:intro-dynamic}
	\left\lbrace
	\begin{aligned}
		 & \dot{\mathbf{x}}_i = \mathbf{v}_i,                                                                                                              \\
		 & \dot{\mathbf{v}}_i = -\left(R + \tfrac{\dot{m}_i}{2 (m_i + \epsilon)}\right) \mathbf{v}_i - \tfrac{w_i}{m_i + \epsilon} \nabla F(\mathbf{x}_i).
	\end{aligned}
	\right.
\end{equation}
This reformulation ensures that ``heavier'' agents, characterized by larger effective mass, experience enhanced damping and thus tend to rapidly converge to local minima, whereas ``lighter'' agents maintain sufficient momentum to escape shallow basins and explore wider regions in the search space. To complete the formulation, we augment the system with equations governing mass dynamics to specify the inter-agent communication. Lagrange multipliers are employed to enforce mass conservation, thereby enabling the selective reallocation of mass from underperforming agents to those demonstrating promising descent trajectories. This “survival-of-the-fittest” paradigm is pivotal in establishing a robust equilibrium between local exploitation and global exploration.

A notable characteristic of system \eqref{eq:intro-energy} is its intrinsic capacity to dissipate total energy \eqref{eq:intro-dynamic} and reduce the mass of the less optimal agents in time. Consequently, it is natural to construct numerical algorithms that preserve this property at the discrete level. These algorithms are commonly referred to as energy-stable schemes. Recent advances in numerical analysis have made substantial progress in the development of such algorithms for dissipative systems; see Refs. \cite{Zhao&W2016,Gong-2018,Gong-2021,Hong-2023,Shen2019,Eyre1998,Du2019,Du2021,ieq2,ieq3}. In this study, we introduce two energy-stable schemes \cite{Eyre1998,Zhao&W2016}. The first employs an explicit–implicit discretization to construct a numerical scheme for \eqref{eq:intro-energy}, which has been rigorously demonstrated to preserve both mass bounds and energy stability under a specified constraint on the time step. Subsequently, utilizing stabilization techniques, we propose an unconditionally energy-stable scheme that preserves the energy stability of system \eqref{eq:intro-energy} and mass bounds irrespective of the time step size.

For any objective function in a minimization problem, we formulate the problem into a minimization problem for the total energy. Then, we implement the two energy-stable algorithms to search for the equilibrium of the weakly decoupled dynamical system, in hoping that it will yield a minimum close to the global minimum of the original objective function. We then compare the numerical results with those obtained using the swarm-based gradient descent method and its invariants to showcase the superior performance of the new algorithms in most cases, especially when the number of agents is small.

The remainder of this paper is organized as follows. In \S 2, we briefly review the SBGD method. In \S 3, we detail the formulation of the SBI system, and rigorously prove that the linearly stable state of the proposed dynamical system corresponds to a minimum of the objective function.  In \S 4 we develop a couple of energy-stable schemes for the SBI system, including the implicit-explicit SBI (SBI-IMEX) algorithm and the stabilized implicit-explicit SBI (SBI-SIMEX) algorithm. We show rigorously that the SBI-IMEX scheme is conditionally energy-stable and that the SBI-SIMEX algorithm is unconditionally energy-stable. Finally, we  provide extensive numerical experiments to benchmark the performance of the proposed SBI-SIMEX algorithm against SBGD methods in \S5. \S6 summarizes our results.

\section{Swarm-Based gradient descent method}

We succinctly review the swarm-based gradient descent (SBGD) method introduced in \cite{Lu2024SwarmGD}. We consider the following optimization problem
\begin{equation}
	\min_{x\in\Omega} F(\mathbf{x}), \label{eq:object-abstract}
\end{equation}
where $\Omega \subset \mathbb{R}^d$ and $F(\bullet): \mathbb{R}^d \to \mathbb{R}$ is a differentiable and likely nonconvex objective function. The classical gradient decent (GD) method for solving \eqref{eq:object-abstract} is given by
\begin{equation}
	\mathbf{x}^{n+1} = \mathbf{x}^n - h \nabla F(\mathbf{x}^n),
\end{equation}
where $h$ denotes the step size. It is well known that the classical GD protocol often becomes ensnared within the basins of attraction of local minima, thereby limiting its effectiveness for global optimization.

To alleviate this limitation, Lu et al. proposed the SBGD method in \cite{Lu2024SwarmGD}. The main idea is to expand the problem into a multi-agent optimization problem by initializing $N$ agents $\mathbf{x}_i \in \mathbb{R}^d, \ i = 1, \cdots, N$, each endowed with an associated mass $m_i$, such that $\sum_{i=1}^N m_i = 1$. At each time step, the mass of each agent is redistributed: agents corresponding to larger function values relinquish mass, thereby enabling the agent with the minimal function value to accrue additional mass. This dynamic redistribution is governed by the following dynamic equations of $m_i(t)$ and mass conservation:
\begin{equation}\label{eq:mass-SBGD}
	\left\lbrace
	\begin{aligned}
		\frac{d}{dt} m_i(t) & = -\phi_p(\eta_i(t)) m_i(t),            &  & i \neq i(t)        ,                          \\
		m_i(t)              & = 1 - \sum\limits_{j \neq i(t)} m_j(t), &  & i = i(t) = {\rm argmin}_i F(\mathbf{x}_i(t)).
	\end{aligned}
	\right.
\end{equation}
Here,
\ben
\left\{
\bea{l}
\phi_p(x) = x^p, \quad \eta_i(t) = \frac{F(\mathbf{x}_i) - F(\mathbf{x}_{i_-})}{F(\mathbf{x}_{i_+}) - F(\mathbf{x}_{i_-})},\\
\\
i_- (t) = {\rm argmin}_{1 \leq i \leq N} F(\mathbf{x}_i(t)), \quad i_+(t) = {\rm argmax}_{1 \leq i \leq N} F(\mathbf{x}_i(t)).
\eea\right.
\een
Once $m_i^{n+1} (i=1, \cdots, N)$ are computed via an appropriate numerical scheme from \eqref{eq:mass-SBGD}, the following gradient descent step is employed to update the positions of the agents.
\begin{equation}\label{eq:dg-SBGD}
	\mathbf{x}_i^{n+1} = \mathbf{x}_i^n - h(\mathbf{x}_i^n, \lambda \psi_q (\widetilde{m}_i^{n+1})) \nabla F(\mathbf{x}_i^n), \quad \widetilde{m}_i^{n+1} = \frac{m_i^{n+1}}{\max\limits_i m_i^{n+1}}.
\end{equation}
Here, $\psi_q (x) = x^q$, and the step size, $h$, is selected by a backtracking algorithm that it is as large as possible while satisfying
\begin{equation}\label{eq:dg-step-choice}
	F(\mathbf{x}_i^n - h\nabla F(\mathbf{x}_i^n)) \leq F(\mathbf{x}_i^n) - \lambda \psi_q (\widetilde{m}_i^{n+1}) h |\nabla F(\mathbf{x}_i^n)|^2.
\end{equation}
Criterion \eqref{eq:dg-step-choice} is essential for the SBGD method. First, it ensures that each agent is assigned a time step that yields a minimum decrease in the objective function at every iteration. Second, it facilitates communication among agents. Specifically, agents with lower mass receive larger time steps, enabling them to escape local minima and explore broader regions in searching for better solutions. In contrast, agents with higher mass are given smaller time steps to promote rapid convergence toward a local minimum. For any given pair $(p, q)$, the SBGD method based on \eqref{eq:mass-SBGD} and \eqref{eq:dg-SBGD} is referred to as the ${\rm SBGD}_{pq}$ method.

\section{Swarm-based inertial method}
\subsection{Inertial dynamics}
We illustrate our novel swarm-based inertial (SBI) method from the perspective of nonequilibrium thermodynamics. Consider a system of $N$ agents with positions $\mathbf{x}_i(t)$ for $i = 1, \cdots, N$, each of which is endowed with mass $m_i(\mathbf{x}(t))$ and evolves in the terrain shaped by the potential (objective) function $F(\bx_i(t))$. The velocity of each agent is denoted as $\dot{\mathbf{x}}_i(t) = \frac{d}{dt} \mathbf{x}_i(t)$. Our goal is to minimize the following total mechanical energy for each agent
\begin{equation*}
	\min\limits_{\mathbf{x}_i \in \mathbb{R}^d} E_i(\mathbf{x}) = \frac{m_i(\mathbf{x}) + \epsilon}{2}  \|\dot{\mathbf{x}}_i\|^2 + w_i F(\mathbf{x}_i), \quad i = 1, \cdots, N.
\end{equation*}
Here, $\mathbf{x} = (\mathbf{x}_1, \cdots, \mathbf{x}_N)$ represents the collective state of the agents, $w_i (i=1, \cdots, N)$ are weighting factors employed to balance the inertia and potential contributions to the total energy of each agent, and $\epsilon>0$ is a user-defined parameter to safeguard the lower bound of mass, chosen to be sufficiently small here. The introduction of inertia provides an additional mechanism to facilitate oscillatory movement of the agents.

Our primary objective is then to derive a dynamic system in which the velocity of ``heavier'' agents decreases rapidly, ensuring that the agents converge toward the local minimum within the attractive region. Conversely, the velocity of ``lighter'' agents undergoes a more gradual deceleration, allowing them to maintain sufficient inertia to escape local minima and explore more regions.

The time derivative of each agent's mechanical energy is computed as follows:
\begin{equation*}
	\begin{aligned}
		\frac{d E_i(\mathbf{x})}{dt} & = \frac{1}{2} \dot{m}_i \|\dot{\mathbf{x}}_i\|^2 + (m_i + \epsilon) (\ddot{\mathbf{x}}_i, \dot{\mathbf{x}}_i) + w_i (\nabla F(\mathbf{x}_i), \dot{\mathbf{x}}_i) \\
		                             & = \left( (m_i + \epsilon) \ddot{\mathbf{x}}_i + \frac{1}{2} \dot{m}_i \dot{\mathbf{x}}_i + w_i \nabla F(\mathbf{x}_i), \dot{\mathbf{x}}_i \right).
	\end{aligned}
\end{equation*}
Invoking the generalized Onsager principle \cite{Wang-2020}, we introduce a friction parameter $R > 0$ and define the following relationship between generalized force and flux
\begin{equation}\label{eq:dynamic-xv-second-order}
	(m_i + \epsilon) \ddot{\mathbf{x}}_i + \frac{1}{2} \dot{m}_i \dot{\mathbf{x}}_i + w_i \nabla F(\mathbf{x}_i) = - R(m_i + \epsilon) \dot{\mathbf{x}}_i.
\end{equation}
The energy dissipation rate is rewritten into
\ben
\frac{dE_i}{dt}=-(R(m_1+\epsilon) \dot{\bx},\dot{\bx}).
\een

Introducing an intermediate variable $\mathbf{v}_i = \dot{\mathbf{x}}_i$, we recast \eqref{eq:dynamic-xv-second-order} into the following first-order system
\begin{equation}\label{eq:dynamic-xv-first-order}
	\left\lbrace
	\begin{aligned}
		 & \dot{\mathbf{x}}_i = \mathbf{v}_i,                                                                                                    \\
		 & (m_i+\epsilon) \dot{\mathbf{v}}_i = -\left(R(m_i + \epsilon) + \frac{1}{2}\dot{m}_i\right) \mathbf{v}_i - w_i \nabla F(\mathbf{x}_i).
	\end{aligned}
	\right.
\end{equation}
To complete system \eqref{eq:dynamic-xv-first-order}, it is necessary to introduce the dynamics of $\mathbf{m} = (m_1, \cdots, m_N)^\top$.   To accomplish this, we adopt the dynamical equation of \eqref{eq:mass-SBGD} for all variables and enforce mass conservation through a Lagrange multiplier $\lambda$,
\begin{equation}\label{eq:mass-lag01}
	\left\lbrace
	\begin{aligned}
		 & \frac{d}{dt} m_i(t) = -\phi_p (\eta_i(t))m_i(t) - \lambda(t) \alpha_i, \ \alpha_i \geq 0, \ \sum_{i=1}^N \alpha_i=1, \\
		 & \sum\limits_{i=1}^N m_i = 1,
	\end{aligned}
	\right.
\end{equation}
where $\alpha_i, i=1,\cdots, N$ are prescribed weights for the mass dynamical equations.

It is easy to see that if the initial total mass satisfies the consistency condition $\sum_{i=1}^N m_i = 1$, then the second equation in \eqref{eq:mass-lag01} becomes equivalent to
\begin{equation}\label{eq:mass-lag02}
	\frac{d}{dt} \sum\limits_{i=1}^N m_i = 0.
\end{equation}
The Lagrange multiplier \(\lambda\) can therefore be determined explicitly by summing the first equation of \eqref{eq:mass-lag01} over $i=1,\ldots,N$ and combining the resulting expression with \eqref{eq:mass-lag02}:
\ben
\lambda(t)=\sum_{j=1}^N \phi_p(\eta_j(t))m_j(t)
\een
Consequently, the dynamics of the mass can be expressed as follows:
\begin{equation}
	\dot{m}_i = - \phi_p(\eta_i(t))m_i(t) + \alpha_i \sum\limits_{j=1}^{N} \phi_p(\eta_j(t))m_j(t).
\end{equation}
We summarize the final governing system as follows:
\begin{equation}\label{eq:dynamic-final}
	\left\lbrace
	\begin{aligned}
		\dot{\mathbf{x}}_i & = \mathbf{v}_i,                                                                                                           \\
		\dot{\mathbf{v}}_i & = -\left(R  + \tfrac{\dot{m}_i}{2(m_i + \epsilon)} \right)\mathbf{v}_i - \frac{w_i}{m_i+\epsilon} \nabla F(\mathbf{x}_i), \\
		\dot{m}_i          & = -\phi_p(\eta_i(t)) m_i(t) + \alpha_i \sum\limits_{j=1}^N  \phi_p(\eta_j(t))m_j(t),                                      \\
	\end{aligned}
	\right.
\end{equation}
where
\begin{equation}
	\phi_p(\eta) = \eta^p, \ \eta_i(t) = \frac{F(\mathbf{x}_i) - F(\mathbf{x}_{i_-}) + \epsilon}{F(\mathbf{x}_{i_+}) - F(\mathbf{x}_{i_-}) + \epsilon}.
\end{equation}
In this paper, we choose
\begin{equation*}
	\alpha_i(t) =
	\left\lbrace
	\begin{aligned}
		 & 1 &  & i = i_-(t)  \\
		 & 0 &  & \text{else}
	\end{aligned}
	\right.
\end{equation*}
In general, the friction operator, $R$, can be chosen to be distinct for each agent based on the need of users. In this case, these can be adjustable parameters of the model. For simplicity, we adopt a unified value in this study.

To be succinct, we rewrite \eqref{eq:dynamic-final} into the following compact form
\begin{equation}\label{eq:dynamic-final-compact}
	\left\lbrace
	\begin{aligned}
		\dot {\mathbf{x}} & = \mathbf{v},                                                                                                                                                                                                           \\
		\dot{\mathbf{v}}  & = - \mathbb{I}_d \otimes {\rm diag}\left(R  + \tfrac{1}{2}\dot{\mathbf{g}}_\epsilon(\mathbf{m})\right) \mathbf{v} - \mathbb{I}_d \otimes {\rm diag}\left(\mathbf{w}_\epsilon(\mathbf{m})\right) \mathbf{G}(\mathbf{x}), \\
		\dot{\mathbf{m}}  & = - \left(\mathbb{I}_N -  \bm{\alpha} \mathbf{1}^\top\right) {\rm diag }(\Phi_p ) \mathbf{m},
	\end{aligned}
	\right.
\end{equation}
where $\bm{\alpha}=(\alpha_1, \cdots,\alpha_N)^\top, \ (\alpha)_i=\alpha_i,  \ (\mathbf{g}_{\epsilon}(\mathbf{m}))_i=\ln(m_i+\epsilon), \ (\mathbf{w}_\epsilon)_i=\frac{w_i}{m_i+\epsilon}$, and $(\bG(\mathbf{x}))_i=\nabla F(\mathbf{x}_i)$, $(\Phi_p)_i = \phi_p(\eta_i(t))$.

We will demonstrate that the equilibrium states of \eqref{eq:dynamic-final-compact} correspond to local minima of $F(\bx)$.

\begin{lem}\label{lem:stable-state}
	Let $\mathbb{B} \in \mathbb{R}^{n \times n}$ be a symmetric and negative definite matrix, $\mathbb{A} \in \mathbb{R}^{n \times n}$ be a symmetric matrix, $\mathbb{A}\mathbb{B} = \mathbb{B} \mathbb{A}$. Then, matrix $\mathbb{S} \in \mathbb{R}^{2n\times 2n}$ given by
	\begin{equation*}
		\mathbb{S} = \begin{pmatrix}
			\mathbb{O} & \mathbb{I} \\ \mathbb{A} & \mathbb{B}
		\end{pmatrix}
	\end{equation*}
	possesses eigenvalues with strictly negative real parts if and only if $\mathbb{A}$ is negative definite.
\end{lem}
\begin{proof}
	Since $\mathbb{A} \mathbb{B} = \mathbb{B} \mathbb{A}$. There exists an invertible matrix $\mathbb{P}$ such that
	\begin{equation*}
		\mathbb{P}^{-1} \mathbb{A} \mathbb{P} = \mathbb{D}_A =  {\rm diag}\{\lambda^A_1, \cdots, \lambda^A_n \}, \ \mathbb{P}^{-1} \mathbb{B} \mathbb{P} = \mathbb{D}_B = {\rm diag} \{\lambda^B_1, \cdots, \lambda^B_n \}.
	\end{equation*}
	Consequently, we have
	\begin{equation*}
		\mathbb{S} = \begin{pmatrix} \mathbb{I} & \mathbb{O} \\ \mathbb{O} & \mathbb{P} \end{pmatrix}
		\begin{pmatrix} \mathbb{O} & \mathbb{I} \\ \mathbb{D}_A & \mathbb{D}_B \end{pmatrix}
		\begin{pmatrix} \mathbb{I} & \mathbb{O} \\ \mathbb{O} & \mathbb{P}^{-1} \end{pmatrix}
	\end{equation*}
	Therefore, we only need to consider the eigen-structure of the following subsystem
	\begin{equation*}
		\begin{pmatrix}
			0 & 1 \\ \lambda^A_i & \lambda^B_i
		\end{pmatrix}, \quad 1 \leq i \leq N.
	\end{equation*}
	It is easily confirmed by straightforward calculations that the above system possesses eigenvalues with negative real parts if and only if $\lambda^A_i < 0, \ 1 \leq i \leq N$, which implies that $\mathbb{A}$ is negative definite. The proof is thus completed.
\end{proof}

\begin{thm}
	Let $F(\mathbf{x}) \in C^2(\Omega)$, and denote $\mathbf{x}^\star = (\mathbf{x}_1^\star, \cdots, \mathbf{x}_N^\star)$, with $\mathbf{x}_i^\star \in \mathbb{R}^d$; analogous definitions apply to $\mathbf{v}^\star$, and $\mathbf{m}^\star \in \mathbb{R}^N$. We assume that the Hessian matrices $\nabla^2 F(\mathbf{x}_i^\star), \ 1 \leq i \leq N$ are non-degenerate and that $F(\mathbf{x}_i^\star) \neq F(\mathbf{x}_j^\star) \ \forall i \neq j$. Then, the triplet $(\mathbf{x}^\star, \mathbf{v}^\star, \mathbf{m}^\star)$ constitutes a linearly stable state of system \eqref{eq:dynamic-final} if and only if $\mathbf{x}_i^\star$ are distinct non-degenerate local minima of $F(\mathbf{x})$, with $\mathbf{v}_i = \mathbf{0}, \ \mathbf{m} = \mathbf{e}_{i_-^\star}$, where $i_-^\star = {\rm argmin}_{1\leq i \leq N} F(\mathbf{x}_i^\star)$, and $( \mathbf{e}_{i_-^\star} )_i = \delta_{i,i_-^\star}$, with $\delta_{i,j}$ representing the Kronecker delta.
\end{thm}
\begin{proof}
	To begin, we compute the Jacobian matrix associated with system \eqref{eq:dynamic-final-compact} thereby obtaining
	\begin{equation*}
		\mathbb{J} = \frac{\partial(\dot{\mathbf{x}}_1, \cdots, \dot{\mathbf{x}}_N, \dot{\mathbf{v}}_1, \cdots, \dot{\mathbf{v}}_N, \dot{\mathbf{m}})}{\partial(\mathbf{x}_1, \cdots, \mathbf{x}_N, \mathbf{v}_1, \cdots, \mathbf{v}_N, \mathbf{m})} =
		\begin{pmatrix}
			\mathbb{J}^{\mathbf{xx}}_{11} & \cdots & \mathbb{J}^{\mathbf{xx}}_{1N} & \mathbb{J}^{\mathbf{xv}}_{11} & \cdots & \mathbb{J}^{\mathbf{xv}}_{1N} & \mathbb{J}^{\mathbf{xm}}_1 \\
			\vdots                        &        & \vdots                        & \vdots                        &        & \vdots                        & \vdots                     \\
			\mathbb{J}^{\mathbf{xx}}_{N1} & \cdots & \mathbb{J}^{\mathbf{xx}}_{NN} & \mathbb{J}^{\mathbf{xv}}_{N1} & \cdots & \mathbb{J}^{\mathbf{xv}}_{NN} & \mathbb{J}^{\mathbf{xm}}_N \\
			\mathbb{J}^{\mathbf{vx}}_{11} & \cdots & \mathbb{J}^{\mathbf{vx}}_{1N} & \mathbb{J}^{\mathbf{vv}}_{11} & \cdots & \mathbb{J}^{\mathbf{vv}}_{1N} & \mathbb{J}^{\mathbf{vm}}_1 \\
			\vdots                        &        & \vdots                        & \vdots                        &        & \vdots                        & \vdots                     \\
			\mathbb{J}^{\mathbf{vx}}_{N1} & \cdots & \mathbb{J}^{\mathbf{vx}}_{NN} & \mathbb{J}^{\mathbf{vv}}_{N1} & \cdots & \mathbb{J}^{\mathbf{vv}}_{NN} & \mathbb{J}^{\mathbf{vm}}_N \\
			\mathbb{J}^{\mathbf{mx}}_1    & \cdots & \mathbb{J}^{\mathbf{mx}}_N    & \mathbb{J}^{\mathbf{vm}}_1    & \cdots & \mathbb{J}^{\mathbf{vm}}_N    & \mathbb{J}^{\mathbf{mm}}
		\end{pmatrix},
	\end{equation*}
	where
	\begin{equation*}
		\begin{aligned}
			 & \mathbb{J}^{\mathbf{xx}}_{ij} = \mathbb{O}_d, \ \mathbb{J}^{\mathbf{xv}}_{ij} = \delta_{ij} \mathbb{I}_d, \  \mathbb{J}^{\mathbf{xm}}_i = \mathbb{O}_{d \times N},                                             \\
			 & \mathbb{J}^{\mathbf{vx}}_{ij} =  - \tfrac{1}{2 (m_i + \varepsilon)} \tfrac{\partial \dot{m}_i}{\partial \mathbf{x}_j} \otimes \mathbf{v}_i - \tfrac{w_i}{m_i + \epsilon} \delta_{ij} \nabla^2 F(\mathbf{x}_i), \\
			 & \mathbb{J}^{\mathbf{vv}}_{ij} = - \left(R + \tfrac{\dot{m}_i}{2 (m_i + \epsilon)} \right) \mathbb{I}_d \delta_{ij},                                                                                            \\
			 & \mathbb{J}^{\mathbf{vm}}_i = \sigma(\mathbf{m}) \mathbf{v}_i,                                                                                                                                                  \\
			 & \mathbb{J}^{\mathbf{mx}}_i = \chi(\nabla F(\mathbf{x}_i)) \mathbf{m},                                                                                                                                          \\
			 & \mathbb{J}^{\mathbf{mv}}_i = \mathbb{O}_{N \times d},                                                                                                                                                          \\
			 & \mathbb{J}^{\mathbf{mm}} = -(\mathbb{I}_N - \bm{\alpha} \mathbf{1}^\top) {\rm diag}(\Phi_p),
		\end{aligned}
	\end{equation*}
	Here, $\sigma(\mathbf{m})$ is a linear operator, and $\chi(0) = 0$.

	``$\Leftarrow$'': Assume that the vectors $\mathbf{x}_i^\star$ are distinct, non-degenerate local minima of $F(\mathbf{x})$, and that $\mathbf{v}_i = \mathbf{0}$ with $\mathbf{m} = \mathbf{e}_{i_-^\star}$. In this setting, we have $\nabla F(\mathbf{x}^\star) = 0$, and the Hessian $\nabla^2 F(\mathbf{x}^\star)$ is positive definite. Consequently, the triplet $(\mathbf{x}^\star, \mathbf{v}^\star, \mathbf{m}^\star)$ constitutes a steady state of \eqref{eq:dynamic-final}. To establish its linear stability, we investigate the eigenvalues of the Jacobian $\mathbb{J}^\star = \mathbb{J}(\mathbf{x}^\star, \mathbf{v}^\star, \mathbf{m}^\star)$. A straightforward computation yields
	\begin{equation*}
		\mathbb{J}^\star =
		\begin{pmatrix}
			\mathbb{O}_{dN}                                                                                       & \mathbb{I}_{dN}          & \mathbb{O}_{dN \times N}                                                          \\
			-\left[\mathbb{I}_d \otimes {\rm diag}(\mathbf{w}^\star_\epsilon)\right] \mathbb{H}(\mathbf{x}^\star) & -R \mathbb{I}_{dN}       & \mathbb{O}_{dN \times N}                                                          \\
			\mathbb{O}_{N \times dN}                                                                              & \mathbb{O}_{N \times dN} & - (\mathbb{I}_N - \mathbf{e}_{i^\star_-}\mathbf{1}^\top) {\rm diag}(\Phi^\star_p)
		\end{pmatrix},
	\end{equation*}
	and we observe that block $- (\mathbb{I}_N - \mathbf{e}_{i^\star_-}\mathbf{1}^\top) {\rm diag}(\Phi^\star_p)$  is upper triangular with diagonal entries $-\phi_p(\eta_1^\star)$, $-\phi_p(\eta_{i^\star_- - 1})$, $0$,  $-\phi_p(\eta_{i^\star_- + 1})$, $-\phi_p(\eta^\star_N)$. It is noteworthy that the zero eigenvalue arises from the mass constraint; indeed, one must demonstrate that $\mathbb{J}^\star$ has strictly negative eigenvalues when restricted to the manifold $\mathbb{R}^{Nd} \times \mathbb{R}^{Nd} \times \mathcal{M}$, with $\mathcal{M} = \{ \mathbf{m} \in \mathbb{R}^N: \mathbf{1}^\top \mathbf{m} = 1 \}$. More rigorously, we define a coordinate transformation
	\begin{equation*}
		\psi:\mathbb{R}^{N-1} \to \mathcal{M}, \ \psi(\mathbf{z}) = \left(1 - \sum\limits_{i=1}^{N-1} \mathbf{z}_i, \mathbf{z}_1, \cdots, \mathbf{z}_{N-1}\right),
	\end{equation*}
	and, by employing the above coordinate transformation, the original system is reformulated in terms of the variables $(\mathbf{x}, \mathbf{v}, \mathbf{z})$ as follows:
	\begin{equation*}
		\left\lbrace
		\begin{aligned}
			\dot {\mathbf{x}} & = \mathbf{v},                                                                                                                                                                                                                                      \\
			\dot{\mathbf{v}}  & = - \mathbb{I}_d \otimes {\rm diag}\left(R + \tfrac{1}{2} \dot{\widetilde{\mathbf{g}}}_\epsilon (\mathbf{z}) \right) \mathbf{v} - \mathbb{I}_d \otimes {\rm diag}\left(\widetilde{\mathbf{w}}_\epsilon (\mathbf{z})\right) \mathbf{G}(\mathbf{x}), \\
			\dot{\mathbf{z}}  & = - [D\psi(\mathbf{z})]^{\dagger} \left(\mathbb{I}_N -  \bm{\alpha} \mathbf{1}^\top\right) {\rm diag }(\Phi_p ) \psi(\mathbf{z}).
		\end{aligned}
		\right.
	\end{equation*}
	Here, the function $\widetilde{\mathbf{g}}_\epsilon: \mathbb{R}^{N-1} \to \mathbb{R}^N$ is defined via the composition $\widetilde{\mathbf{g}}_\epsilon = \mathbf{g}_\epsilon \circ \psi$, and $D\psi(\mathbf{z}) : \mathbb{R}^N \to \mathbb{R}^{N-1}$ denotes the non-degenerate Jacobian of $\psi$ at $\mathbf{z}$. Consequently, the restriction of the Jacobian $\mathbb{J}^\star$, denoted by $\widetilde{\mathbb{J}}^\star$, evaluated at $(\mathbf{x}^\star, \mathbf{v}^\star, \mathbf{z}^\star)$ may be written as
	\begin{equation*}
		\widetilde{\mathbb{J}}^\star =
		\begin{pmatrix}
			\mathbb{O}_{dN}                                                                                       & \mathbb{I}_{dN}            & \mathbb{O}_{dN \times N-1}                                                                                                                                                                    \\
			-\left[\mathbb{I}_d \otimes {\rm diag}(\mathbf{w}^\star_\epsilon)\right] \mathbb{H}(\mathbf{x}^\star) & -R \mathbb{I}_{dN}         & \mathbb{O}_{dN \times N-1}                                                                                                                                                                    \\
			\mathbb{O}_{N-1 \times dN}                                                                            & \mathbb{O}_{N-1 \times dN} & - \frac{\partial}{\partial \mathbf{z}}\left([D\psi(\mathbf{z}^\star)]^{\dagger} (\mathbb{I}_N - \mathbf{e}_{i^\star_-}\mathbf{1}^\top) {\rm diag}(\Phi^\star_p) \psi(\mathbf{z}^\star)\right)
		\end{pmatrix}.
	\end{equation*}
	It can be readily verified that all eigenvalues of $-\frac{\partial}{\partial \mathbf{z}}\left( [D\psi(\mathbf{z}^\star)]^{\dagger} (\mathbb{I}_N - \mathbf{e}_{i^\star_-}\mathbf{1}^\top) {\rm diag}(\Phi^\star_p) \psi(\mathbf{z}^\star)\right)$ are negative. Therefore, our attention shifts to the block
	\begin{equation*}
		\mathbb{S}^\star =
		\begin{pmatrix}
			\mathbb{O}_{dN}                                                                                       & \mathbb{I}_{dN}    \\
			-\left[\mathbb{I}_d \otimes {\rm diag}(\mathbf{w}^\star_\epsilon)\right] \mathbb{H}(\mathbf{x}^\star) & -R \mathbb{I}_{dN} \\
		\end{pmatrix}.
	\end{equation*}
	By invoking Lemma \ref{lem:stable-state}, one deduces that all eigenvalues of $\mathbb{S}^\star$ possess negative real parts.

	``$\Rightarrow$'': Conversely, suppose that $(\mathbf{x}^\star, \mathbf{v}^\star, \mathbf{m}^\star)$ represents a linearly stable state of \eqref{eq:dynamic-final}. A combination of the first and second equations in \eqref{eq:dynamic-final} implies that $\mathbf{v}^\star = 0$ and $\nabla F(\mathbf{x}_i^\star) = 0$. Furthermore, by invoking the final equation of \eqref{eq:dynamic-final} in conjunction with the definition of $\alpha_i$, it follows that for $i \neq i_{-}^\star$
	\begin{equation*}
		- \phi_p(\eta^\star_{i}(t)) m^\star_i(t) = 0,
	\end{equation*}
	and hence $m^\star_i(t) = 0 \ \forall i \neq i_i^\star$. By mass conservation, one concludes that $m_{i_-^\star} = 1$, which implies $\mathbf{m}^\star = \mathbf{e}_{i_-^\star}$. Finally, applying Lemma \ref{lem:stable-state} reveals that $H(\mathbf{x}_i^\star)$ is positive definite, thereby confirming that each $\mathbf{x}_i^\star$ is indeed a local minimum of $F(\mathbf{x})$. This completes the proof.
\end{proof}

\begin{rmk}
	We remark that in practice the mass constraint may be omitted; that is, one may simply adopt the following mass evolution equation
	\begin{equation*}
		\dot{m}_i = - \phi_p (\eta_i(t)) m_i(t), \ 1 \leq i \leq N.
	\end{equation*}
	Due to the definition of $\phi_p(\bullet)$, it is evident that the mass associated with the ``heavy'' agent who possesses a smaller value of $F(\bx)$ decreases more gradually than that of their ``light'' counterparts who has a larger value of $F(\bx)$. When the algorithm is stopped by a specified criterion, the agent with the largest mass is interpreted as the global minimum identified by the system. A potential concern for this practice is that the masses of all agents might decay too rapidly toward zero. This issue can be mitigated by introducing merging, renormalization and removal strategies for agents, as detailed in subsequent sections.
\end{rmk}

\subsection{Energy stable schemes for swarming dynamics with inertia}

In this section, we introduce several numerical schemes to solve system \eqref{eq:dynamic-final} (or equivalently, \eqref{eq:dynamic-final-compact}) in the context of structure-preserving approximations to highlight the preservation of energy dissipative properties. These methods are designed to effectively dissipate the mechanical energy of every agent at the discrete level. To facilitate the analysis, we assume that $F \in C^2(\Omega)$ and that $\nabla F(\mathbf{x})$ is Lipschitz continuous, with the Lipschitz constant given by
\begin{equation*}
	L := \max\limits_{\mathbf{x} \in \Omega} \|D^2F(\mathbf{x})\| < \infty.
\end{equation*}

The first algorithm  is the following IMEX method for \eqref{eq:dynamic-final-compact}.
\begin{alg}[SBI-IMEX]
	\begin{equation}\label{eq:IMEX-Euler}
		\left\lbrace
		\begin{aligned}
			h^{-1} (\mathbf{x}_i^{n+1} - \mathbf{x}_i^n) & = \mathbf{v}_i^{n+1},                                                                                                                                          \\
			h^{-1} (\mathbf{v}_i^{n+1} - \mathbf{v}_i^n) & = - \left(R + \tfrac{h^{-1}}{2}\tfrac{m_i^{n+1} - m_i^n}{m_i^n + \epsilon}\right) \mathbf{v}_i^{n+1} - \tfrac{w_i}{m_i^n + \epsilon} \nabla F(\mathbf{x}_i^n), \\
			h^{-1} (m_i^{n+1} - m_i^n)                   & = -\phi_p(\eta_i^n)m_i^n +  \alpha_i \sum\limits_{j=1}^N \phi_p(\eta_j^n) m_j^n.
		\end{aligned}
		\right.
	\end{equation}
	where $h$ is the step size to be selected. The stability of \eqref{eq:IMEX-Euler} is determined by two factors: one is whether the mass is bound-preserving, i.e., $m_i^0 \in [0, 1] \Rightarrow m_i^n \in [0, 1] \ \forall n$ and $1 \leq i \leq N-1$, and another is whether it decreases the mechanical energy of each agent. These two factors are guaranteed for the IMEX scheme by the following theorem.
\end{alg}

\begin{thm}\label{thm:IMEX-Euler}
	Suppose that the time step in \eqref{eq:IMEX-Euler} satisfies $h \leq \min \left\{\min\limits_{1 \leq i \leq N}\tfrac{2 R}{w_i L}, 1\right\}$, and that the initial mass lies in $[0, 1]$ with $\sum\limits_{i=1}^N m_i^0 = 1$. The IMEX method then preserves the bound of $m_i$ and the total mass, and dissipates the mechanical energy of each agent at every iteration as follows:
	\begin{equation*}
		E_i^{n+1} - E_i^n \leq - \frac{m_i^n + \epsilon}{2} \|\mathbf{v}_i^{n+1} - \mathbf{v}_i^n\|^2 - h \left(R(m^n_i+\epsilon) - \frac{1}{2} h w_i L\right)\|\mathbf{v}_i^{n+1}\|^2,
	\end{equation*}
	where $E_i^n$ is the discrete energy, defined by:
	\begin{equation}\label{eq:discrete-energy}
		E_i^n = \frac{m_i^n + \epsilon}{2} \|\mathbf{v}_i^n\|^2 + w_i F(\mathbf{x}_i^n).
	\end{equation}
\end{thm}
\begin{proof}
	We begin by proving the bound-preserving property of the mass of each agent using the mathematical induction. The property holds for the case $n = 0$ according to the choice of the initial condition. Assuming that for every $1 \leq i \leq N$, we have $m_i^n \in [0, 1]$, we now demonstrate that $m_i^{n+1} \in [0,1], \ 1 \leq i \leq N$. Notice that we have $0 \leq \phi_p(\eta_i^n) \leq 1$ and $\sum_{i=1}^N m_i^n = 1$ according to their definitions. Consequently,
	\begin{equation*}
		m_i^{n+1} = (1 - h \phi_p(\eta_i^n)) m_i^n + h \alpha_i \sum\limits_{j=1}^N \phi_p(\eta_j^n) m_j^n \geq 0.
	\end{equation*}
	\begin{equation*}
		\begin{aligned}
			m_i^{n+1} & = \left(1 -  h (1-\alpha_i) \phi_p(\eta_i^n) \right) m_i^n + h \alpha_i \sum_{\substack{j=1 \\ j\neq i}}^N \phi_p(\eta_j^n) m_j^n \\
			          & \leq m_i^n + h\alpha_i \sum_{\substack{j=1                                                  \\ j\neq i}}^N m_j^n = m_i^n + h\alpha_i(1 - m_i^n) \\
			          & = \left(1 - h\alpha_i\right)m_i^n + h\alpha_i \leq 1.
		\end{aligned}
	\end{equation*}
	The total mass conservation is shown by summing all the equations of the masses.

	To prove the dissipation property, we take the inner product of the first equation in \eqref{eq:IMEX-Euler} with $h\nabla F(\mathbf{x}_i^n)$ and the inner product of the second equation with $h(m_i^n + \epsilon) \mathbf{v}_i^{n+1}$. Consequently, we obtain
	\begin{equation}\label{eq:proof-ep-euler1}
		(\nabla F(\mathbf{x}_i^n), \mathbf{x}_i^{n+1} - \mathbf{x}_i^n) = h (\nabla F(\mathbf{x}_i^n), \mathbf{v}_i^{n+1}),
	\end{equation}
	\begin{equation}\label{eq:proof-ep-euler2}
		\begin{aligned}
			(m_i^n+\epsilon) (\mathbf{v}_i^{n+1}, \mathbf{v}_i^{n+1} - \mathbf{v}_i^n) & = -h \left(R(m^n_i+\epsilon) + \tfrac{h^{-1}}{2}(m_i^{n+1} - m_i^n)\right) \|\mathbf{v}_i^{n+1}\|^2 \\
			                                                                           & \quad - h w_i (\nabla F(\mathbf{x}_i^n), \mathbf{v}_i^{n+1}).
		\end{aligned}
	\end{equation}
	Combining \eqref{eq:proof-ep-euler1} and \eqref{eq:proof-ep-euler2} with the following identity
	\begin{equation}
		(\mathbf{v}_i^{n+1}, \mathbf{v}_i^{n+1} - \mathbf{v}_i^n) = \frac{1}{2}\|\mathbf{v}_i^{n+1}\|^2 - \frac{1}{2} \|\mathbf{v}_i^n\|^2 + \frac{1}{2}\|\mathbf{v}_i^{n+1} - \mathbf{v}_i^n \|^2,
	\end{equation}
	we have
	\begin{equation}\label{eq:add_01}
		\begin{aligned}
			 & \tfrac{m_i^n + \epsilon}{2} (\|\mathbf{v}^{n+1}_i\|^2 - \|\mathbf{v}_i^n\|^2) + \tfrac{m_i^{n+1} - m_i^n}{2} \|\mathbf{v}_i^{n+1}\|^2 +  \tfrac{m_i^n + \epsilon}{2} \|\mathbf{v}_i^{n+1} - \mathbf{v}_i^n\|^2 \\
			 & \quad = - h R(m_i^n + \epsilon) \|\mathbf{v}_i^{n+1}\|^2 - w_i (\nabla F(\mathbf{x}_i^n), \mathbf{x}_i^{n+1} - \mathbf{x}_i^n).
		\end{aligned}
	\end{equation}
	The Taylor's formula gives us
	\begin{equation} \label{eq:add_02}
		\begin{aligned}
			F(\mathbf{x}_i^{n+1}) - F(\mathbf{x}_i^n) & \leq (\nabla F(\mathbf{x}_i^n), \mathbf{x}_i^{n+1} - \mathbf{x}_i^n) + \tfrac{L}{2} \|\mathbf{x}_i^{n+1} - \mathbf{x}_i^n\|^2 \\
			                                          & = (\nabla F(\mathbf{x}_i^n), \mathbf{x}_i^{n+1} - \mathbf{x}_i^n) + \tfrac{h^2 L}{2} \|\mathbf{v}_i^{n+1}\|^2.
		\end{aligned}
	\end{equation}
	Multiply both sides of \eqref{eq:add_02} by $w_i$ and adding the resulting expression to \eqref{eq:add_01}, noticing the identity
	\begin{equation*}
		\tfrac{m_i^n + \epsilon}{2} (\|\mathbf{v}^{n+1}_i\|^2 - \|\mathbf{v}_i^n\|^2) + \tfrac{m_i^{n+1} - m_i^n}{2} \|\mathbf{v}_i^{n+1}\|^2 = \tfrac{m_i^{n+1} + \epsilon}{2} \|\mathbf{v}^{n+1}_i\|^2 - \tfrac{m_i^{n} + \epsilon}{2} \|\mathbf{v}^{n}_i\|^2,
	\end{equation*}
	we obtain
	\begin{equation*}
		\begin{aligned}
			 & \tfrac{m_i^{n+1} + \epsilon}{2} \|\mathbf{v}^{n+1}_i\|^2 + w_i F(\mathbf{x}_i^{n+1}) - ( \tfrac{m_i^{n} + \epsilon}{2} \|\mathbf{v}^{n}_i\|^2 + w_i F(\mathbf{x}_i^{n}))           \\
			 & \quad \leq  - h R(m_i^n + \epsilon) \|\mathbf{v}_i^{n+1}\|^2 + \tfrac{h^2 w_i L}{2} \|\mathbf{v}_i^{n+1}\|^2 - \tfrac{m_i^n+\epsilon}{2}\|\mathbf{v}_i^{n+1} - \mathbf{v}_i^n\|^2.
		\end{aligned}
	\end{equation*}
	The proof is thus completed.
\end{proof}
This theorem demonstrates that the discrete energy dissipation rate is enhanced by an additional term associate with the inertia $- \frac{m_i^n + \epsilon}{2} \|\mathbf{v}_i^{n+1} - \mathbf{v}_i^n\|^2,$ highlighting the role numerical inertia plays in this numerical algorithm.

A primary limitation of the SBI-IMEX scheme is its stringent restriction on the time step size. To alleviate this constraint, we introduce the stabilized SBI-IMEX (SBI-SIMEX) scheme.
\begin{alg}[SBI-SIMEX]
	\begin{equation}\label{eq:SIMEX}
		\left\lbrace
		\begin{aligned}
			h^{-1} (\mathbf{x}_i^{n+1} - \mathbf{x}_i^n) & = \mathbf{v}_i^{n+1}                                                                                                                                            \\
			h^{-1} (\mathbf{v}_i^{n+1} - \mathbf{v}_i^n) & = - \left(R + \tfrac{h^{-1}}{2}\tfrac{m_i^{n+1} - m_i^n}{m_i^n + \epsilon} \right) \mathbf{v}_i^{n+1} - \tfrac{w_i \kappa}{m_i^n + \epsilon} \mathbf{x}_i^{n+1} \\
			                                             & \quad + \tfrac{w_i}{m_i^n + \epsilon} \left(\kappa \mathbf{x}_i^n -  \nabla F(\mathbf{x}_i^n)\right)                                                            \\
			h^{-1} (m_i^{n+1} - m_i^n)                   & = -\phi_p(\eta_i^n) m_i^{n} + \alpha_i \sum\limits_{j=1}^N \phi_p(\eta_j^n)m_j^n.
		\end{aligned}
		\right.
	\end{equation}
\end{alg}
\begin{thm}
	Suppose that the time step in \eqref{eq:SIMEX} satisfies $h \leq 1$, the stabilization parameters $\kappa \geq L$, and  the initial mass lies in $[0, 1]$ with $\sum\limits_{i=1}^N m_i^0 = 1$. The SBI-SIMEX method preserves the mass conservation property and dissipates the mechanical energy of each agent at every iteration as follows:
	\begin{equation*}
		E_i^{n+1} - E_i^n \leq - \frac{m_i^n + \epsilon}{2} \|\mathbf{v}_i^{n+1} - \mathbf{v}_i^n\|^2 - h R_i\left(\mathbf{m}^n\right) \|\mathbf{v}_i^{n+1}\|^2,
	\end{equation*}
	where, \(E_i^n\) denotes the discrete energy, defined by \eqref{eq:discrete-energy}.
\end{thm}
\begin{proof}
	The proof of the bound-preserving property of the mass is similar to that provided in the proof of Theorem \ref{thm:IMEX-Euler}. We focus on the discrete energy dissipation law.

	Taking the inner product of both sides of the first equation with $\kappa \mathbf{x}_i^{n+1} - \kappa \mathbf{x}_i^n + \nabla F(\mathbf{x}_i^n)$ and applying the following inequality:
	\begin{equation*}
		F(\mathbf{x}^{n+1}) - F(\mathbf{x}^n) \leq (\kappa \mathbf{x}^{n+1} - \kappa \mathbf{x}^n + \nabla F(\mathbf{x}^n), \mathbf{x}^{n+1} - \mathbf{x}^n),
	\end{equation*}
	we obtain
	\begin{equation*}
		F(\mathbf{x}_i^{n+1}) - F(\mathbf{x}_i^n) \leq h (\kappa \mathbf{x}_i^{n+1}- \kappa \mathbf{x}_i^n + \nabla F(\mathbf{x}_i^n), \mathbf{v}_i^{n+1}).
	\end{equation*}
	Taking the inner product of both sides of the second equation with $h(m_i^n+\epsilon) \mathbf{v}_i^{n+1}$, and then combining the resulting expression with the inequality derived above, yields the desired result after straightforward calculations.
\end{proof}
The stabilized algorithm reduces the discrete numerical energy dissipation rate further, indicating more rapid convergence in searching for minima.

In practice, iterating over the entire ensemble of agents until convergence is typically inefficient, as the computational cost of the SBI method escalates with the increase of the number of agents. To address this computational issue, it is essential to merge those agents that are in close proximity and to remove those ensnared in local minima. We illustrate the complete SBI method, incorporating both merging and removal strategies as follows:
\begin{table}[H]
	\centering
	\caption{Implementation of Swarm-based Inertial Method}
	\label{tab:algo}
	\fontsize{10pt}{12pt}\selectfont
	\renewcommand{\arraystretch}{1.2}
	\begin{tabular*}{0.9\textwidth}{@{\extracolsep{\fill}}l} \toprule[2pt]
		{\bf Input:} \quad Initial positions $\mathbf{x}_i^0$, velocities $\mathbf{v}_i^0$, masses $m_i^0$; \\
		\quad Maximum iterations $N_{iter}$, number of agents $N$; \\
		\quad Tolerance parameters: $tol_m$, $tol_{merge}$, $tol_{res}$ \\
		{\bf Output:} \quad Minimum value $\mathbf{x}_{min}$ \\ \midrule[1pt]
		{\bf for} $k=1,\cdots,N_{iter}$ \\
		\quad {\bf if} $N>1$ \\
		\quad\quad {\bf for} $i=1,\cdots,N$ \\
		\quad\quad\quad Update $m_i^{n+1}$, $\mathbf{x}_i^{n+1}$, $\mathbf{v}_i^{n+1}$ by solving (\ref{eq:dynamic-final}) \\
		\quad\quad {\bf end for} \\
		\quad\quad Move $\mathbf{x}_i^{n+1}$ if $m_i < \tfrac{1}{N}\,tol_m$\\
		\quad\quad Merge two agents by setting $(\mathbf{x}_i, \mathbf{v}_i, m_i) =(\tfrac{\mathbf{x}_i + \mathbf{x}_j}{2}, \tfrac{\mathbf{v}_i + \mathbf{v}_j}{2}, m_i + m_j)$, if $\|\mathbf{x}_i-\mathbf{x}_j\| \leq tol_{merge}$ \\
		\quad {\bf else} \quad (Only one particle) \\
		\quad\quad {\bf while} $\|\mathbf{x}^{n+1}-\mathbf{x}^n\|\ge tol_{res}$ \\
		\quad\quad\quad $\mathbf{x}^{n+1}=\mathbf{x}^n-h\nabla F(\mathbf{x}^n)$ \\
		\quad\quad {\bf end while} \\
		\quad {\bf end if} \\
		{\bf end for} \\ \bottomrule[2pt]
	\end{tabular*}
\end{table}
This practical strategy for merging and removal can significantly speed up the search without missing the target.

\section{Numerical Results}

We apply the algorithms to several test problems to 1) show that they are energy stable, i.e., dissipating total mechanical energy; and 2) compare them with the swarm-based gradient descent method (SBGD) to show their efficacy. In all experiments, unless otherwise specified, the tolerances in the Algorithm \ref{tab:algo} are set as follows: $tol_m = 10^{-4}, tol_{merge} = 10^{-3}, tol_{res} = 10^{-5}$.

\subsection{Test of energy dissipation and comparisons with SBGD}
In the first numerical experiment, we illustrate the efficacy of the proposed scheme using the following objective function:
\ben
F(x) = e^{\sin{(2x^2)}} + \frac{1}{10} (x - \frac{\pi}{2}). \label{eq:objective-01}
\een
As depicted in Figure \ref{fig:plot-ex1-objective}, this function exhibits multiple local minima, with the global minimum located at $(x^\star \approx 1.5355)$.

\begin{figure}[H]
	\begin{center}
		\includegraphics[width=0.3\textwidth]{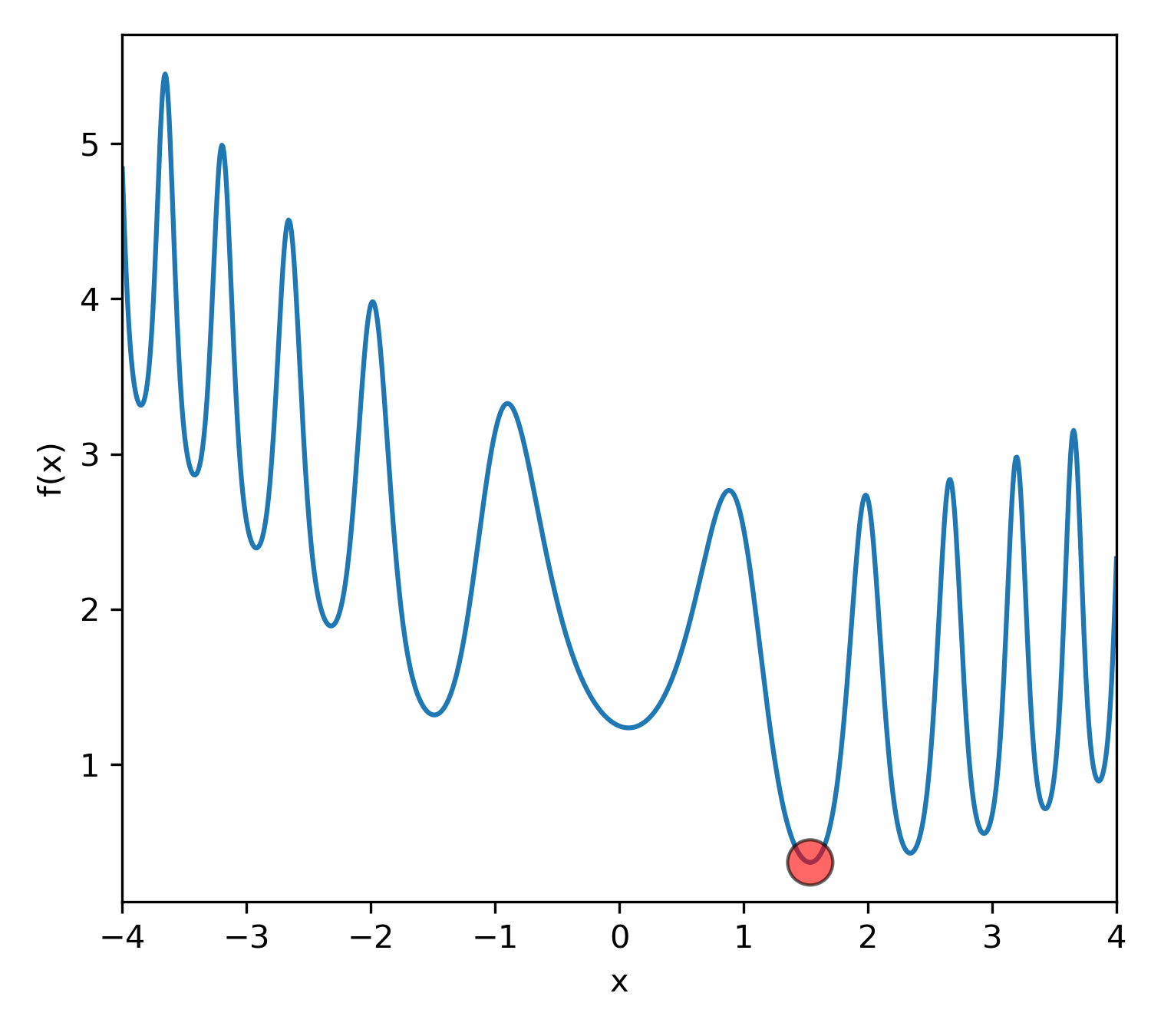}
	\end{center}
	\caption{Plot of the objective function in \eqref{eq:objective-01}}\label{fig:plot-ex1-objective}
\end{figure}

In the subsequent tests, we  vary the number of agents, uniformly distributed within the interval $[-3, -1]$, and assess the success rate of each method in locating the global minimum. For the SBI-SIMEX method, we set $w_i = 10^{-4}$, $R = 1$, and employ a stabilized parameter $\kappa = 10$ with a time step size of $h = 0.5$. Furthermore, the initial velocities of the particles are uniformly sampled from the interval $[1, 5]$.

We first assess the dissipative properties of the proposed SBI-SIMEX scheme by initializing five particles. To provide a clear depiction of the energy evolution of the global system and individual particles, we refrain from merging or removing particles throughout the iterations in this test. The first and second subplots of Figure \ref{fig:plot-ex1-energy} illustrates the energy evolution of individual agents and the overall system. Notice that both the energy of each individual agent and the total system exhibit a monotonically decreasing trend over time towards zero. The last subplot of Figure \ref{fig:plot-ex1-energy} is the time history of the values of the objective function for each agent, which implies that the first and the fourth agent successfully find the global minimum at the steady state.

We evaluate the performance of the SBI‐IMEX and SBI‐SIMEX methods under both mass‐constrained and unconstrained conditions. Table \ref{tab:ex1_comparison_00} provides a comprehensive comparative analysis of the different strategies. It is evident that the success rates of the proposed methods are remarkably similar.
\begin{table}[H]
	\caption{Success rates of different methods for global optimization based on 1000 runs using uniformly generated initial data in $[-3, -1]$ and initial velocity in $[1, 5]$. The stabilization parameter in SBI-SIMEX is taken as $\kappa = 10$.} \label{tab:ex1_comparison_00}
	\begin{center}
		\begin{tabular}[c]{c c c c c c}
			\hline
			N                                     & 5        & 10       & 15       & 20       & 30      \\
			\hline
			SBI-SIMEX                             & $78.8\%$ & $96.5\%$ & $99.1\%$ & $99.8\%$ & $100\%$ \\
			SBI-SIMEX (without mass conservation) & $76.4\%$ & $95.1\%$ & $99.2\%$ & $99.9\%$ & $100\%$ \\
			SBI-IMEX                              & $82.0\%$ & $95.8\%$ & $99.5\%$ & $99.8\%$ & $100\%$ \\
			SBI-IMEX (without mass conservation)  & $77.0\%$ & $94.7\%$ & $99.0\%$ & $99.9\%$ & $100\%$ \\
			\hline
		\end{tabular}
	\end{center}
\end{table}

\begin{figure}[H]
	\begin{center}
		\includegraphics[width=0.3\textwidth]{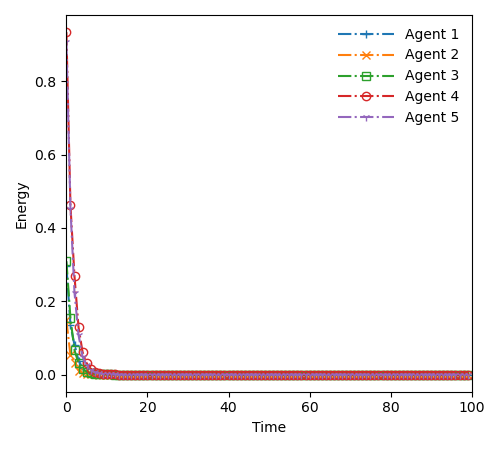}
		\includegraphics[width=0.3\textwidth]{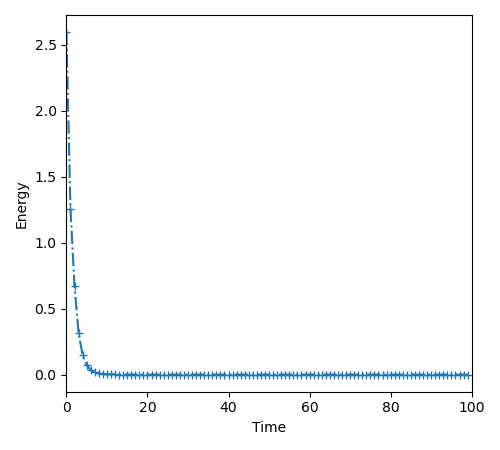}
		\includegraphics[width=0.3\textwidth]{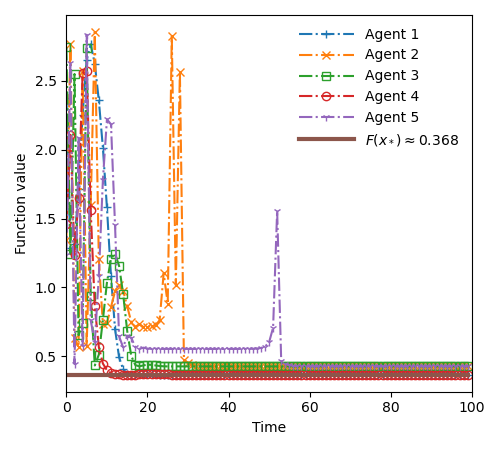}
	\end{center}
	\caption{Temporal evolution of the energy for each agent and the overall system and values of objective function for each agent, respectively.}\label{fig:plot-ex1-energy}
\end{figure}

Next, we compare our proposed approach with the swarm-based gradient descent method (SBGD) introduced in \cite{Lu2024SwarmGD}. Table \ref{tab:ex1_comparison} shows a comparative analysis of SBGD and SBI-SIMEX method, summarizing the results from 1000 independent simulation runs. The results indicate that as the number of agents increases, all methods achieve a high success rate in identifying the global minimum. However, when the number of agents is limited, the SBI-SIMEX scheme demonstrates superior performance.
\begin{table}[H]
	\caption{Success rates of different methods for global optimization based on 1000 runs using uniformly generated initial data in $[-3, -1]$. The results of SBGD are obtained from \cite{Lu2024SwarmGD}.}\label{tab:ex1_comparison}
	\begin{center}
		\begin{tabular}[c]{c c c c c c}
			\hline
			N           & 5        & 10       & 15       & 20       & 30      \\
			\hline
			SBN-SIMEX   & $78.8\%$ & $96.5\%$ & $99.1\%$ & $99.8\%$ & $100\%$ \\
			SBGD$_{11}$ & $36.5\%$ & $83.1\%$ & $97.2\%$ & $99.5\%$ & $100\%$ \\
			SBGD$_{21}$ & $42.4\%$ & $91.4\%$ & $99.0\%$ & $99.8\%$ & $100\%$ \\
			\hline
		\end{tabular}
	\end{center}
\end{table}

\subsection{Effects of the initial velocity and parameters}

One of the keys to the SBI method's success in identifying the ``global'' minimum lies in the selection of the initial velocities and parameters $w_i, i=1, \cdots, N,$ and $R$. The initial velocity is critical in determining whether the agents can overcome local energy barriers to jump to an adjacent valley. When the magnitude of the initial velocity is insufficient, agents may not possess the inertia required to traverse the distance to the global minimum from their starting positions. Conversely, if the initial velocity is excessively high, agents risk overshooting and consequently fail to thoroughly explore regions proximate to local minima. The parameter $w_i$ is instrumental in balancing the contributions of inertia and potential energy, while $R$ modulates the rate of energy decay. In practice, larger values of $w_i$ and $R$ are advantageous when the initial velocity is high, as they help prevent agents from straying excessively far from the optimum. Conversely, smaller values of $w_i$ and $R$ can be beneficial when the initial velocity is low, as they enhance the influence of inertia. These effects will be demonstrated through several illustrative examples.
\begin{figure}[H]
	\begin{center}
		\includegraphics[width=0.15\textwidth]{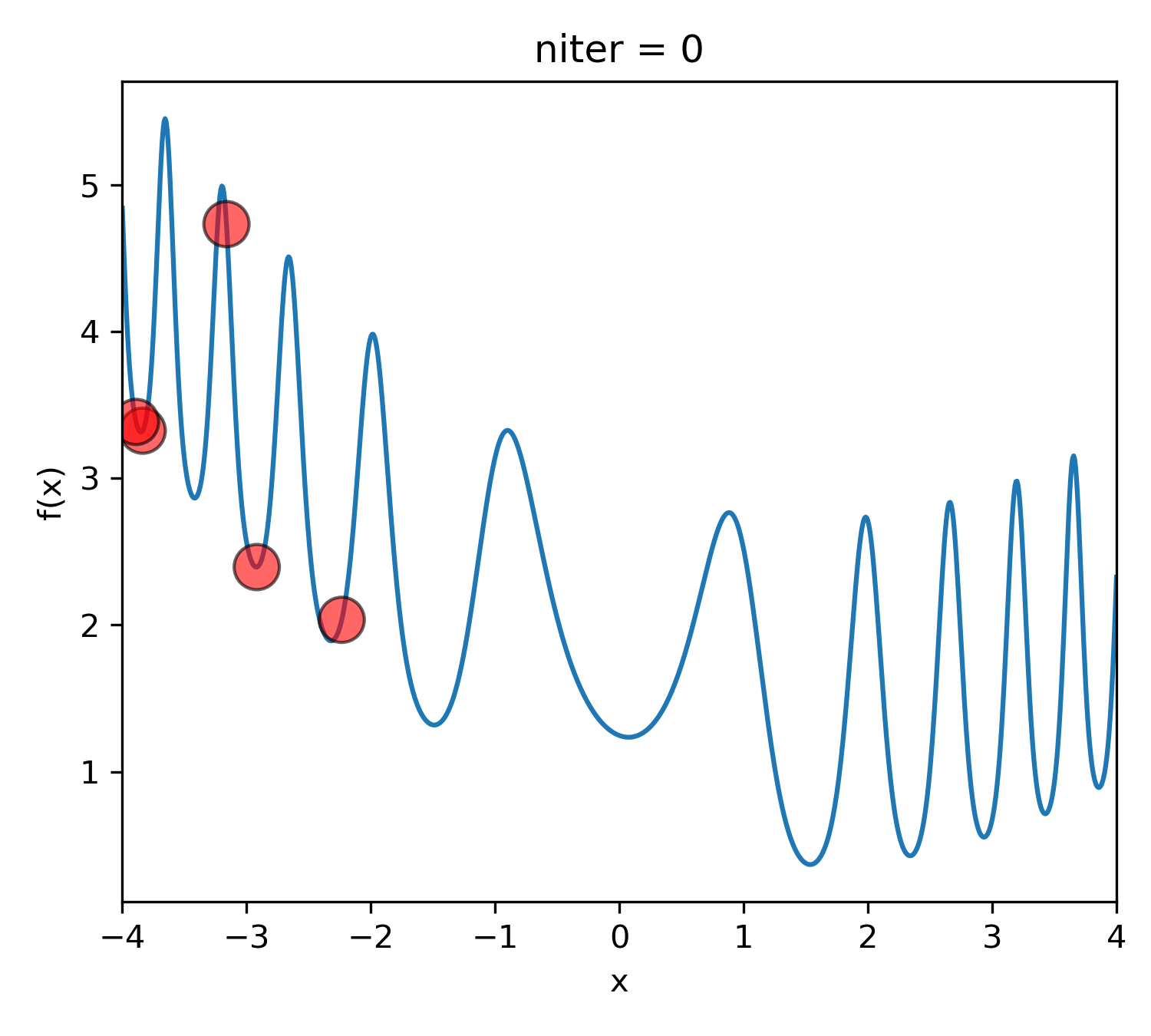}
		\includegraphics[width=0.15\textwidth]{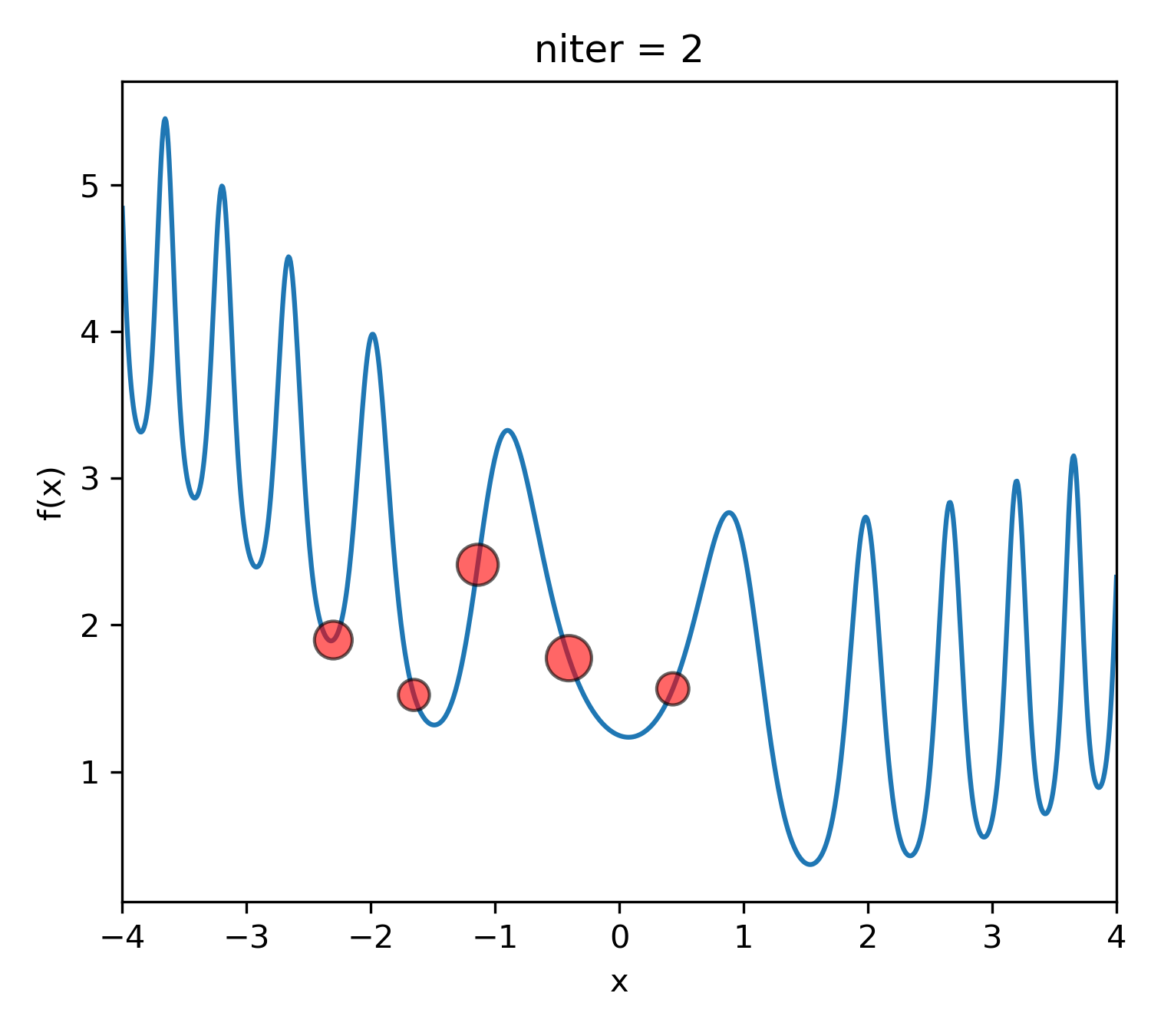}
		\includegraphics[width=0.15\textwidth]{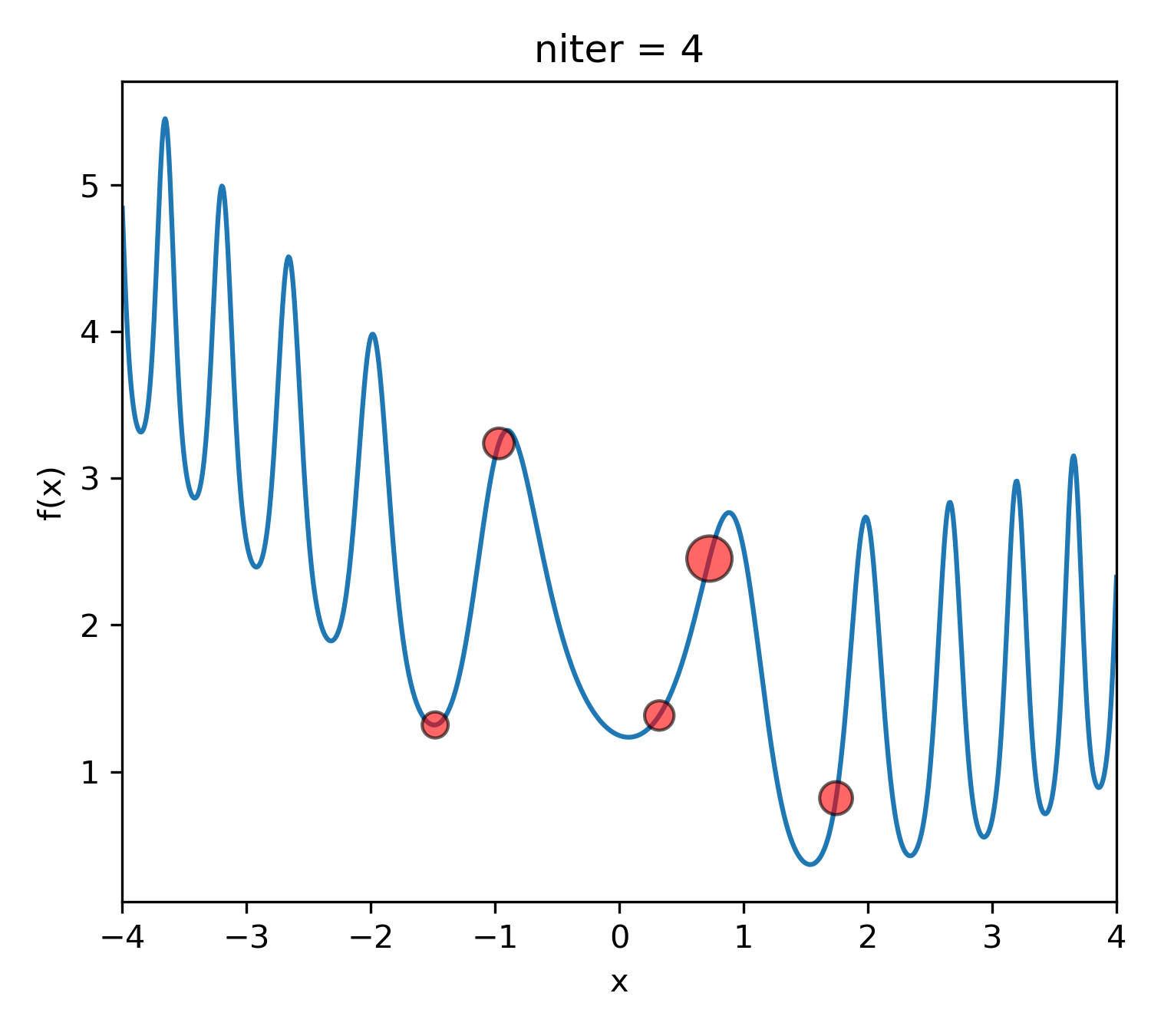}
		\includegraphics[width=0.15\textwidth]{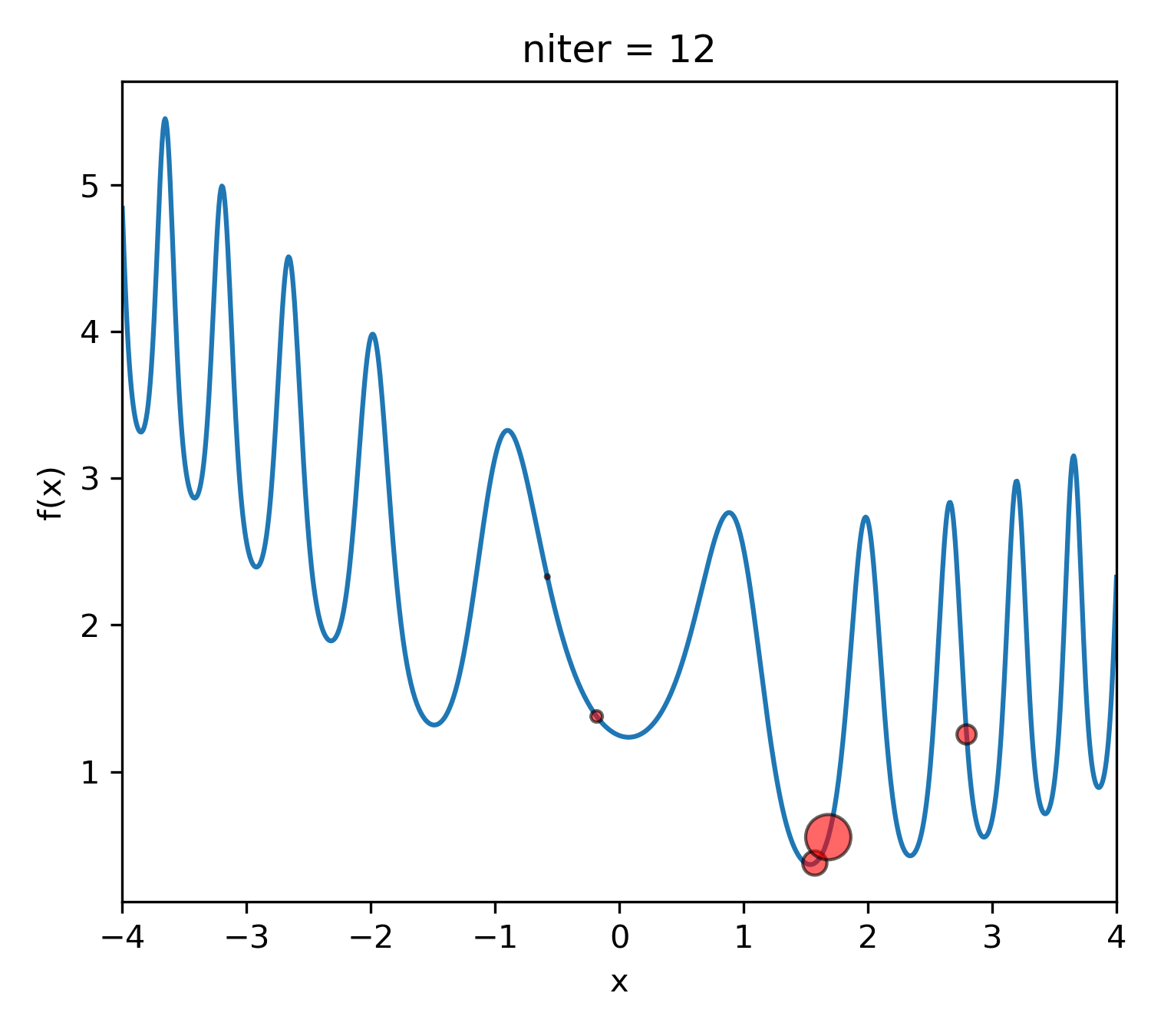}
		\includegraphics[width=0.15\textwidth]{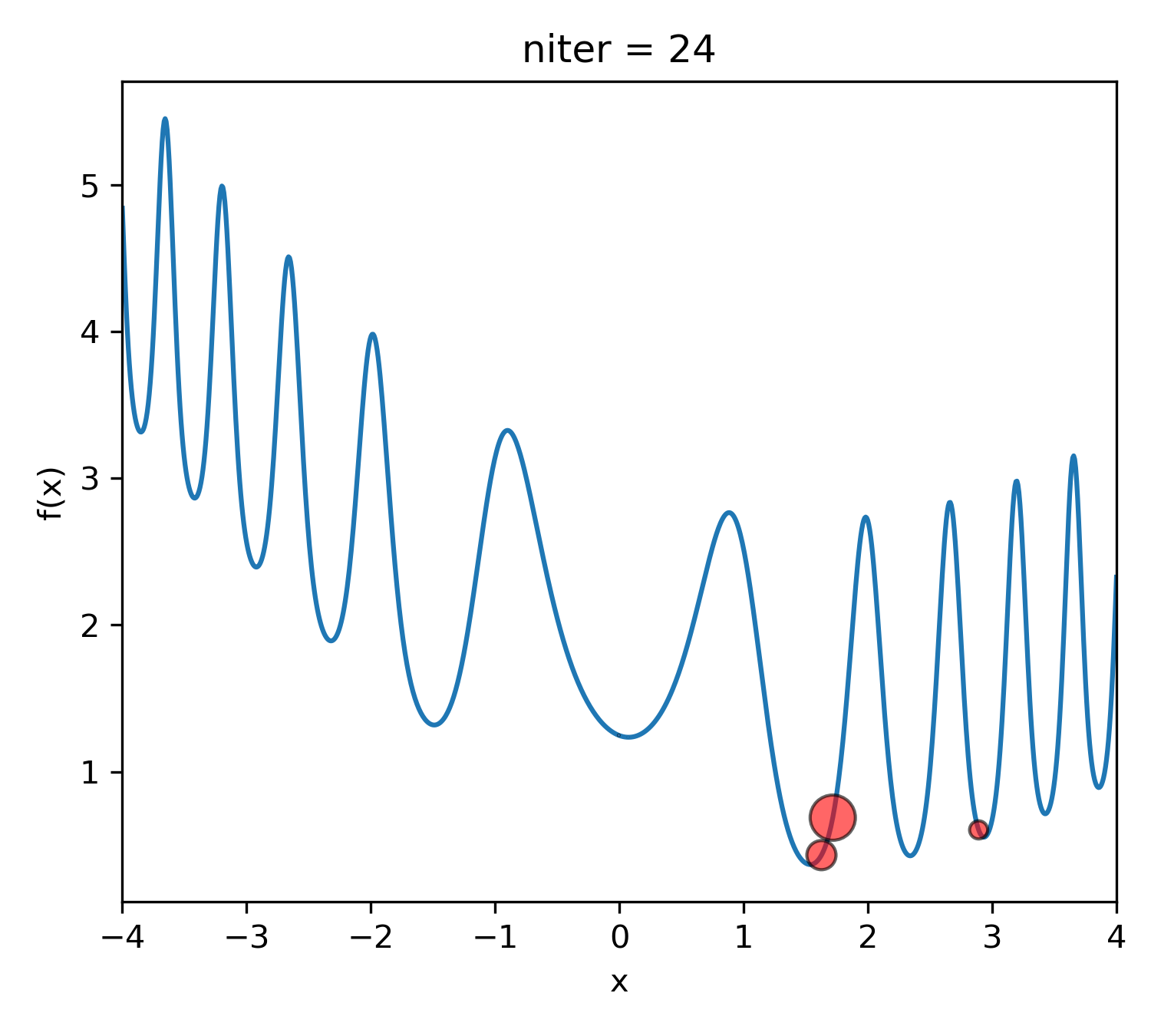}
		\includegraphics[width=0.15\textwidth]{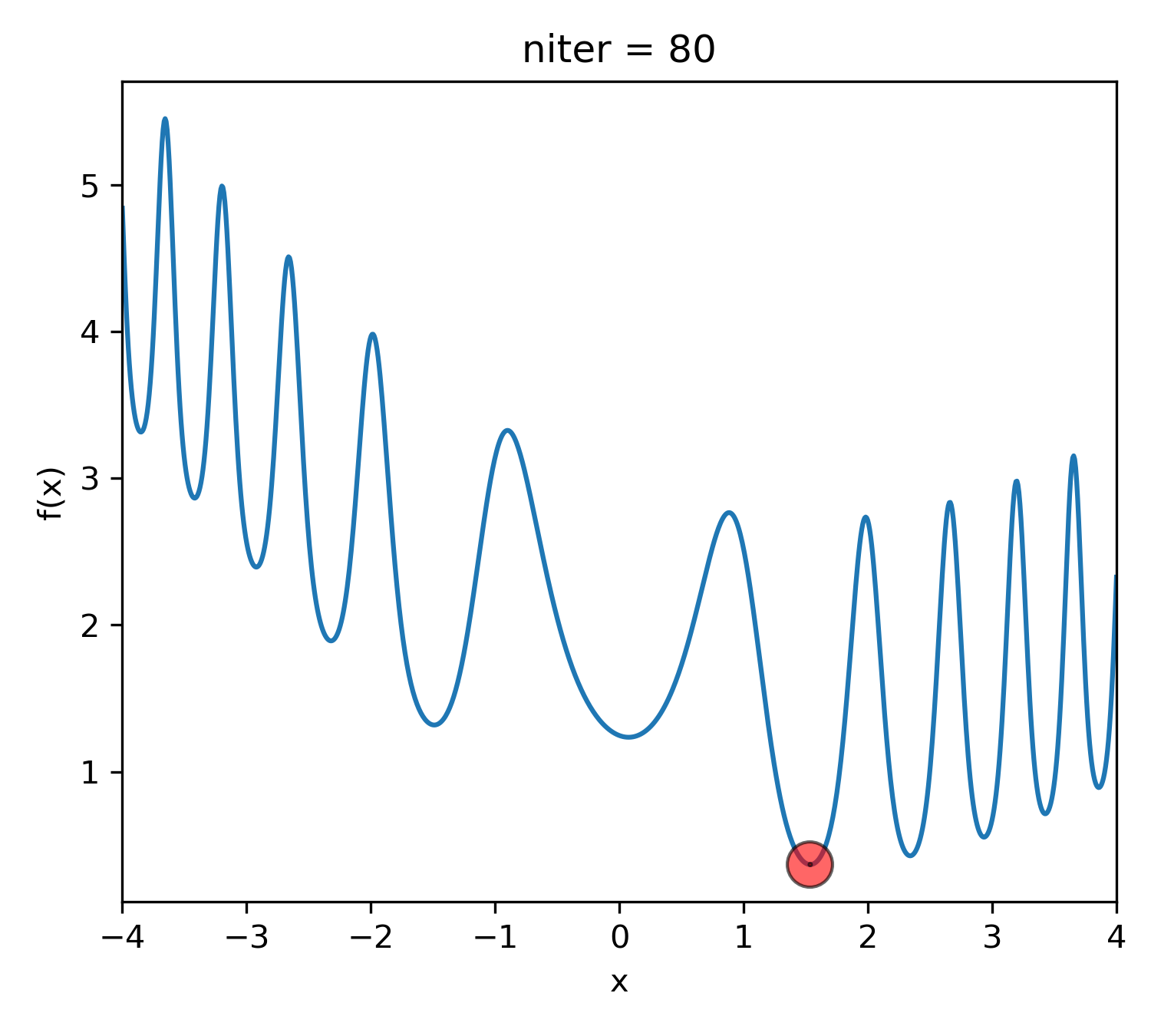}
	\end{center}
	\caption{Movement of agents with initial position $x = {\rm random}(-4, -2)$, velocity $v = {\rm random} (1, 5)$, $w_i = 10^{-4}$, $R = 1$, where the merging and removal strategy are implemented.}\label{fig:goodv_goodw}
\end{figure}

\begin{figure}[H]
	\begin{center}
		\includegraphics[width=0.15\textwidth]{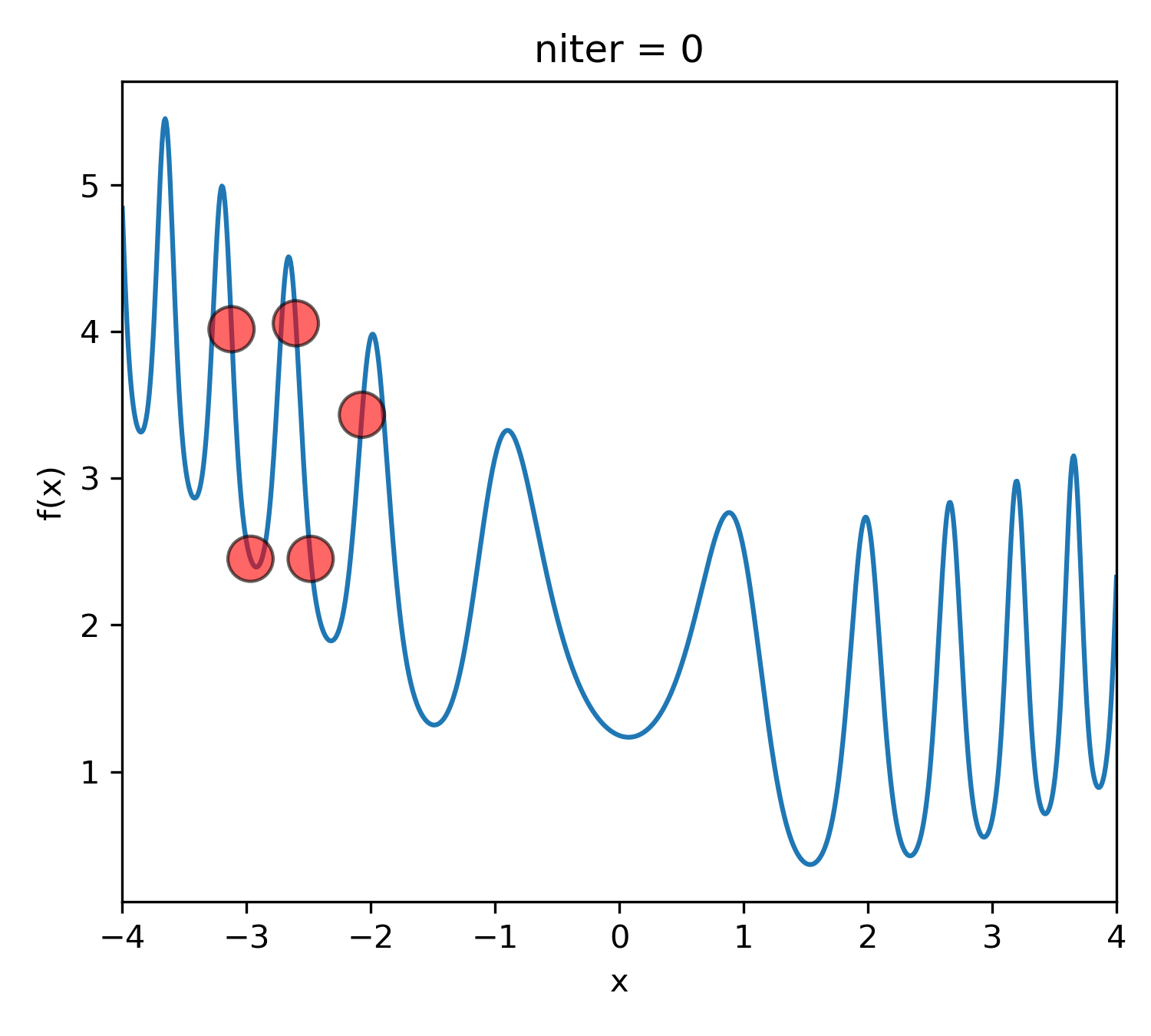}
		\includegraphics[width=0.15\textwidth]{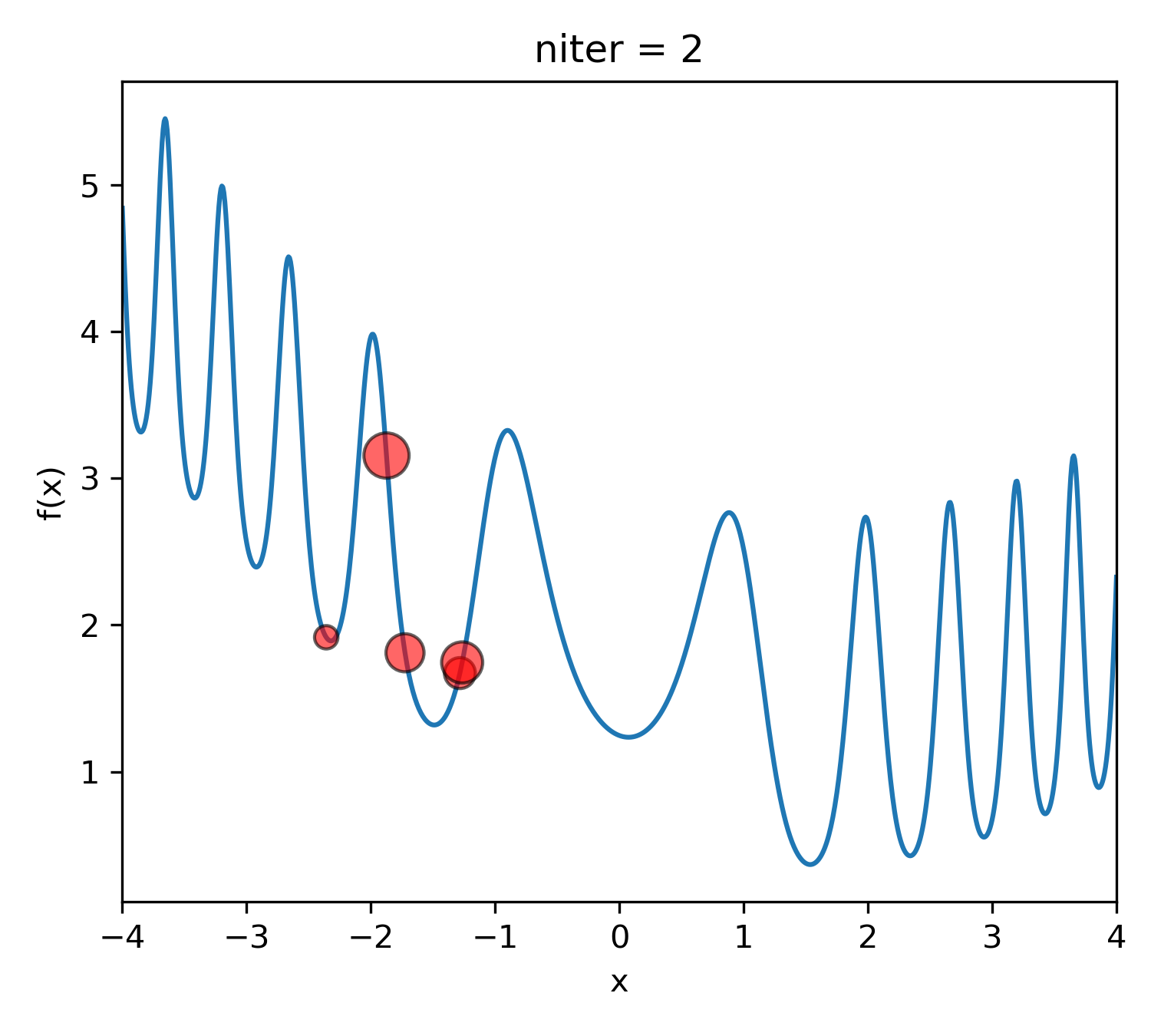}
		\includegraphics[width=0.15\textwidth]{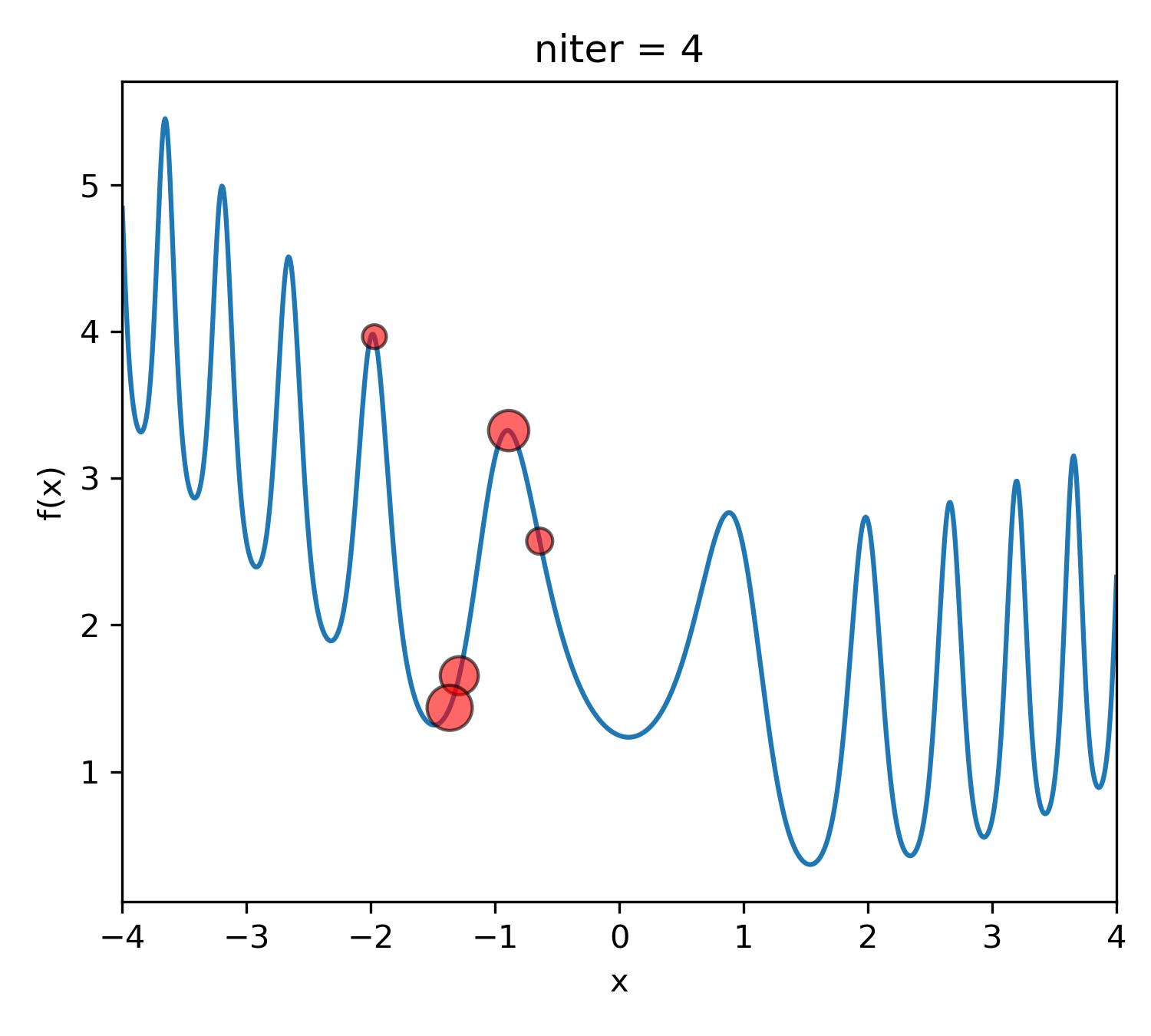}
		\includegraphics[width=0.15\textwidth]{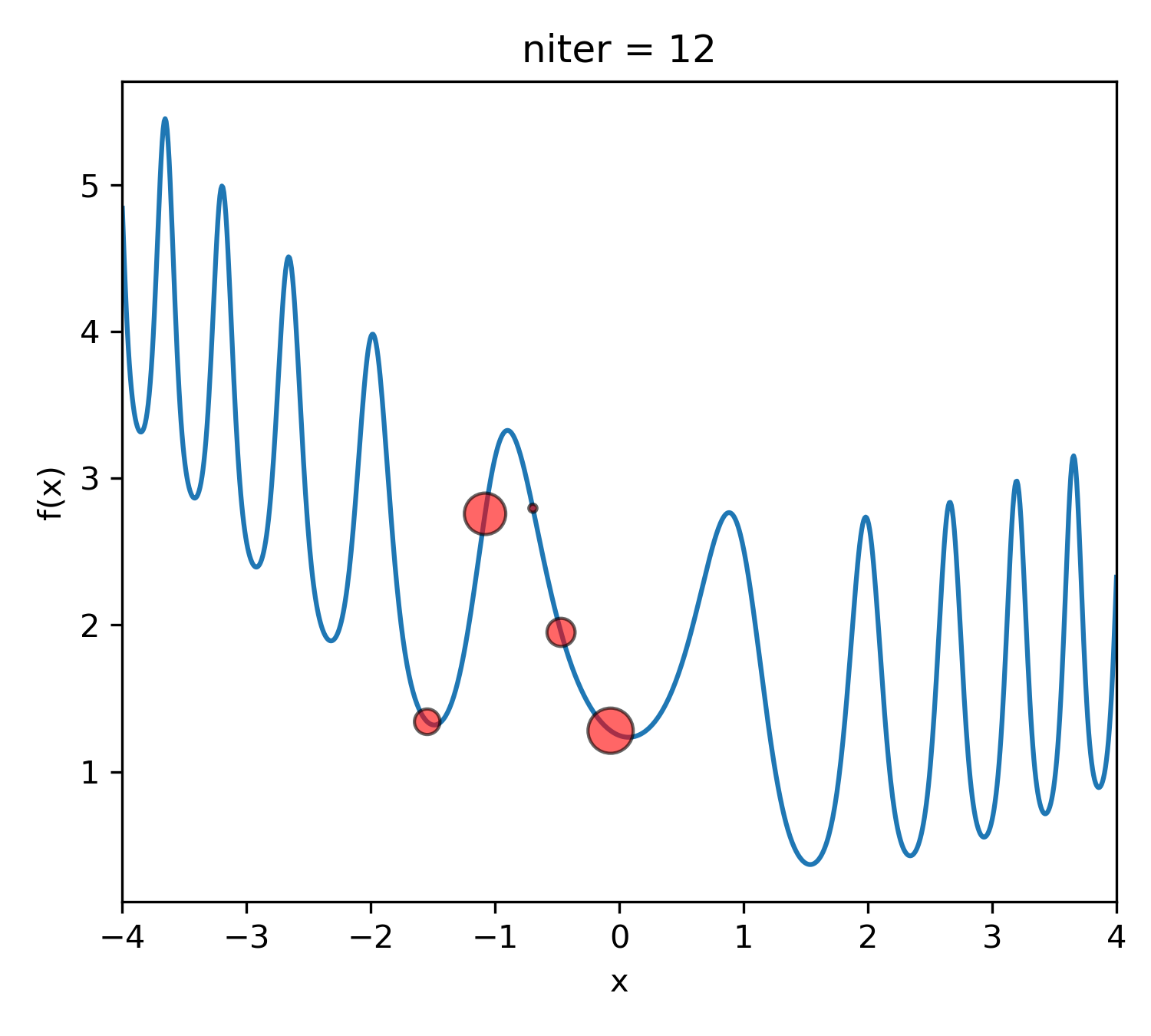}
		\includegraphics[width=0.15\textwidth]{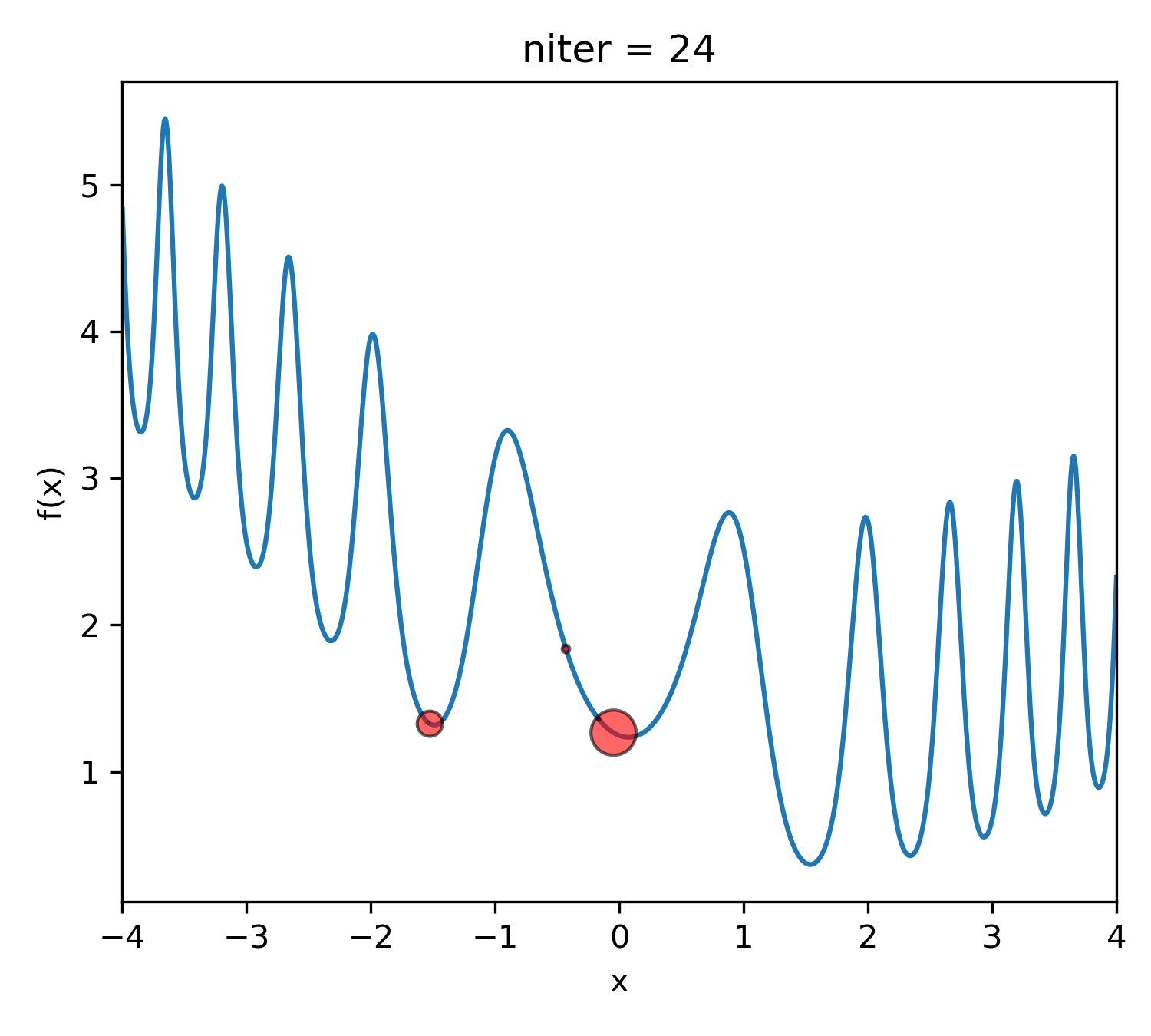}
		\includegraphics[width=0.15\textwidth]{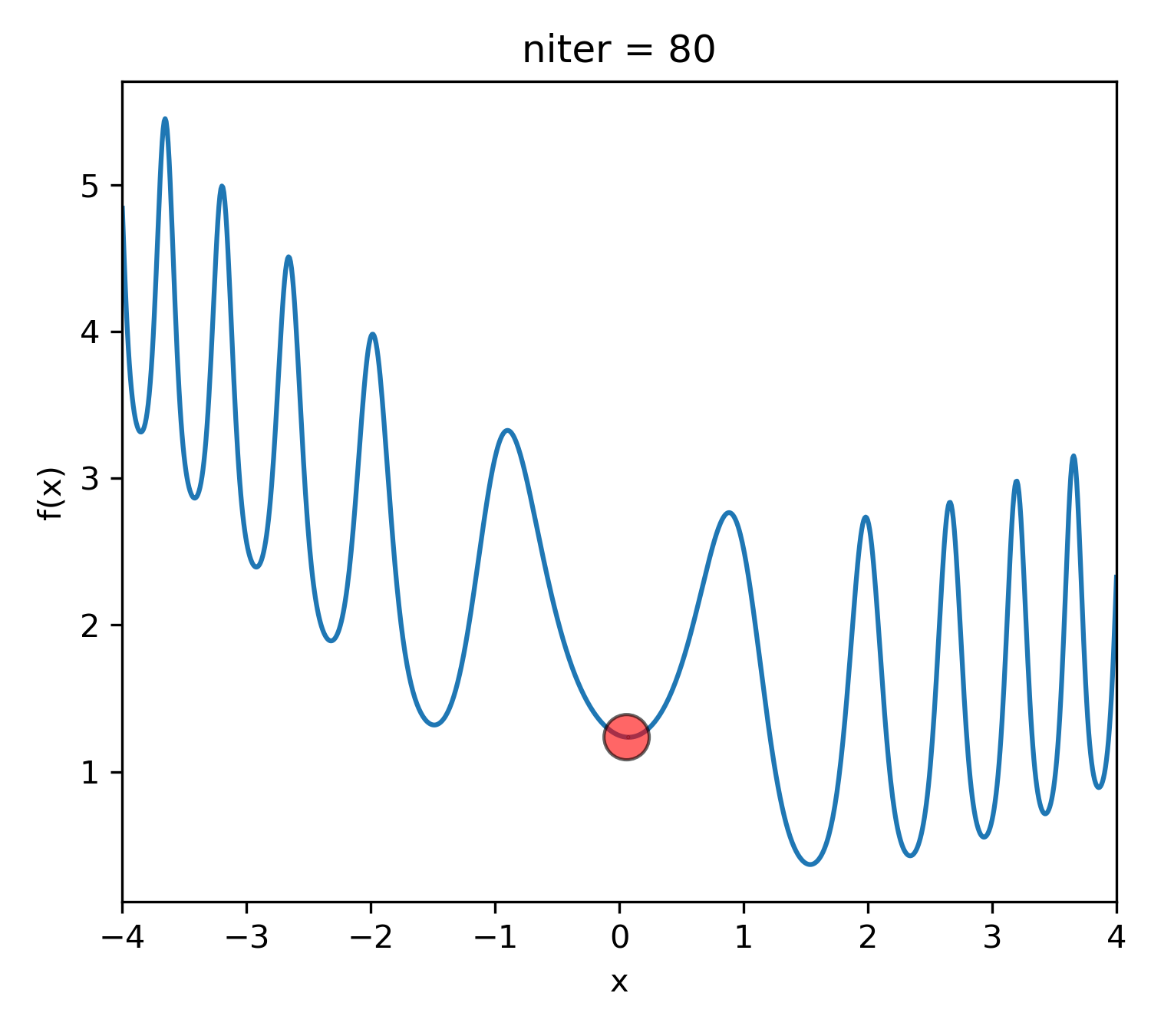}
	\end{center}
	\caption{Movement of agents with initial position $x = {\rm random}(-4, -2)$, velocity $v = {\rm random} (1, 2)$, $w_i = 10^{-4}$, $R = 1$, where the merging and removal strategy are implemented.}\label{fig:smallv_goodw}
\end{figure}

\begin{figure}[H]
	\begin{center}
		\includegraphics[width=0.15\textwidth]{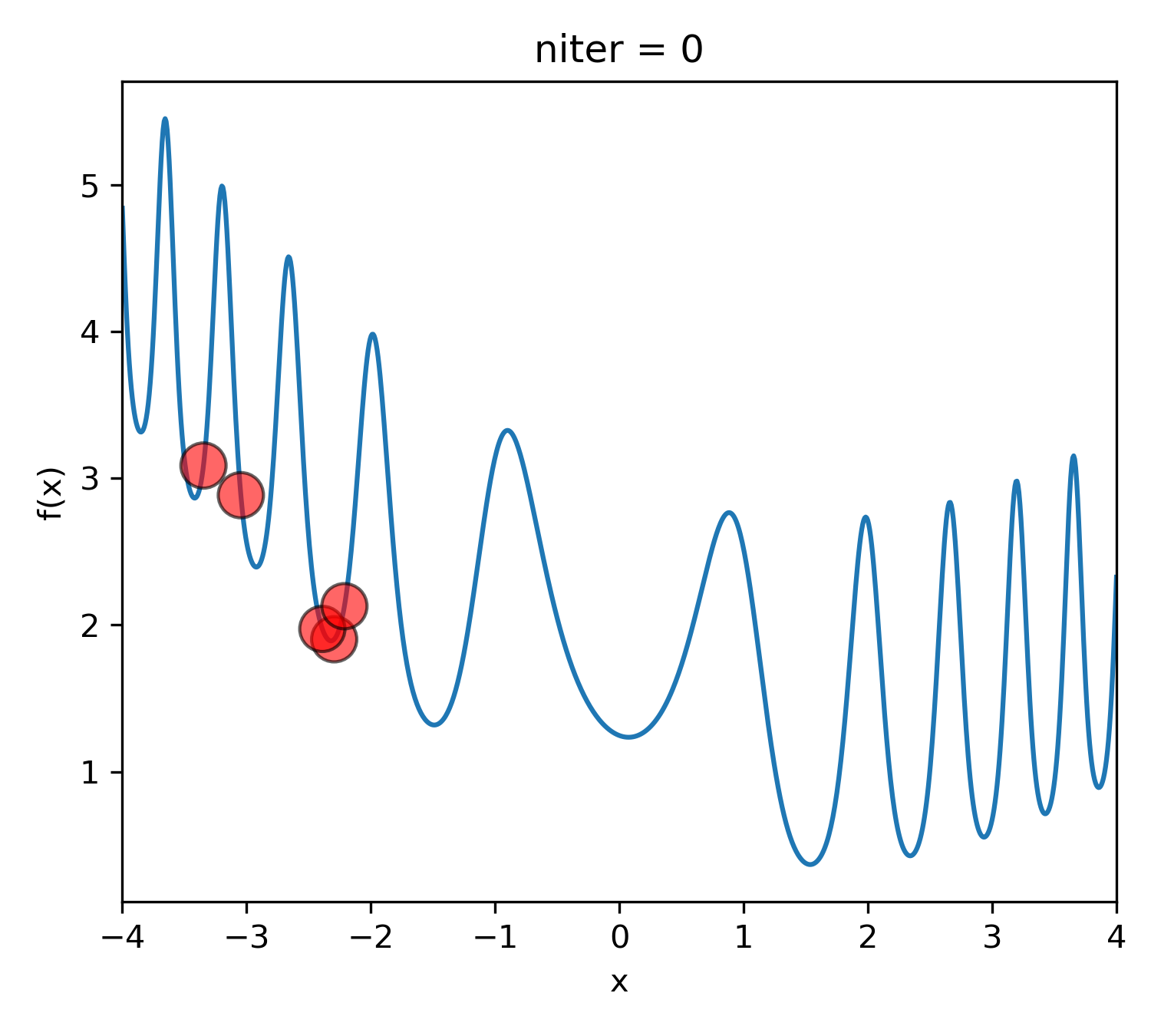}
		\includegraphics[width=0.15\textwidth]{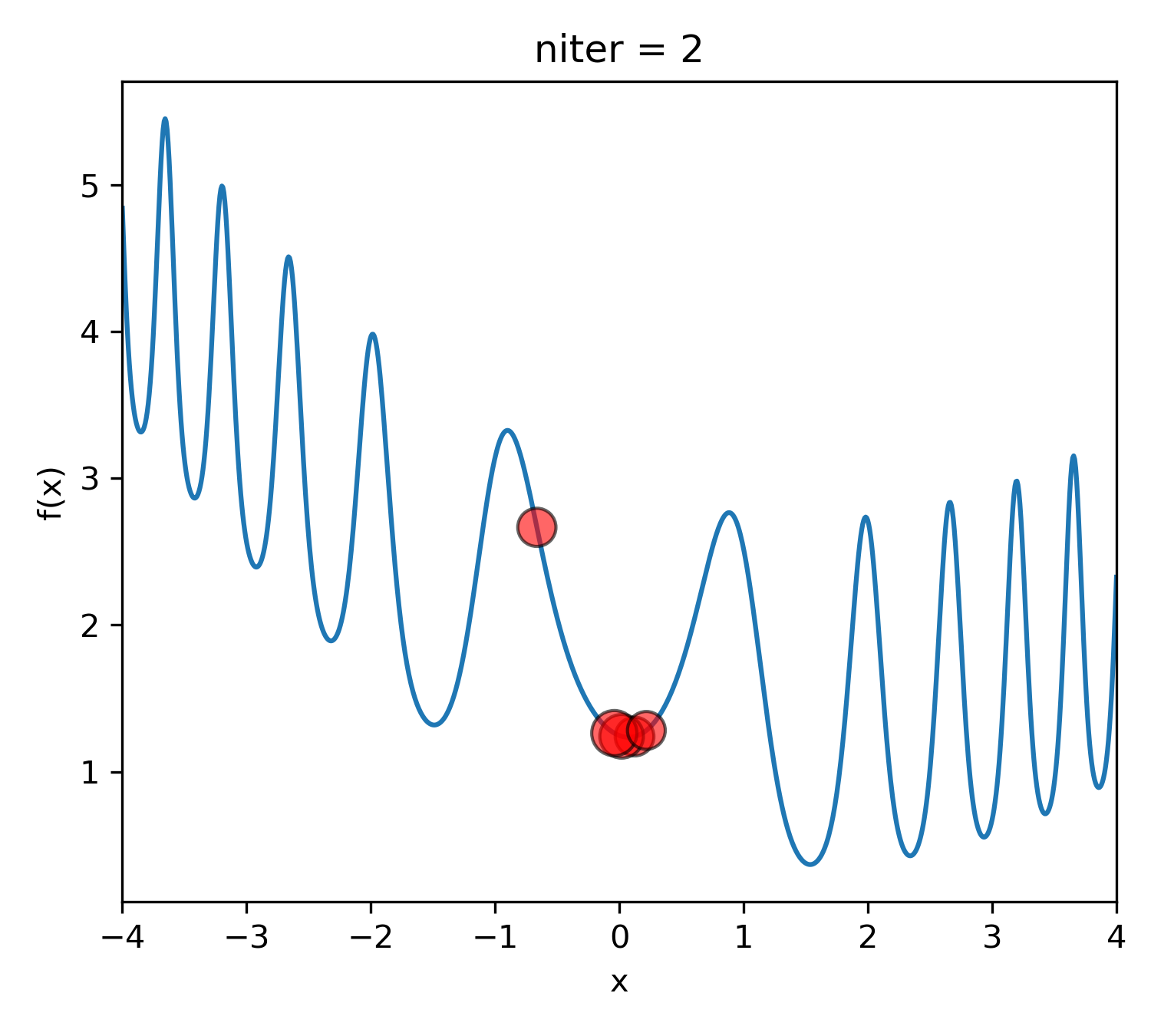}
		\includegraphics[width=0.15\textwidth]{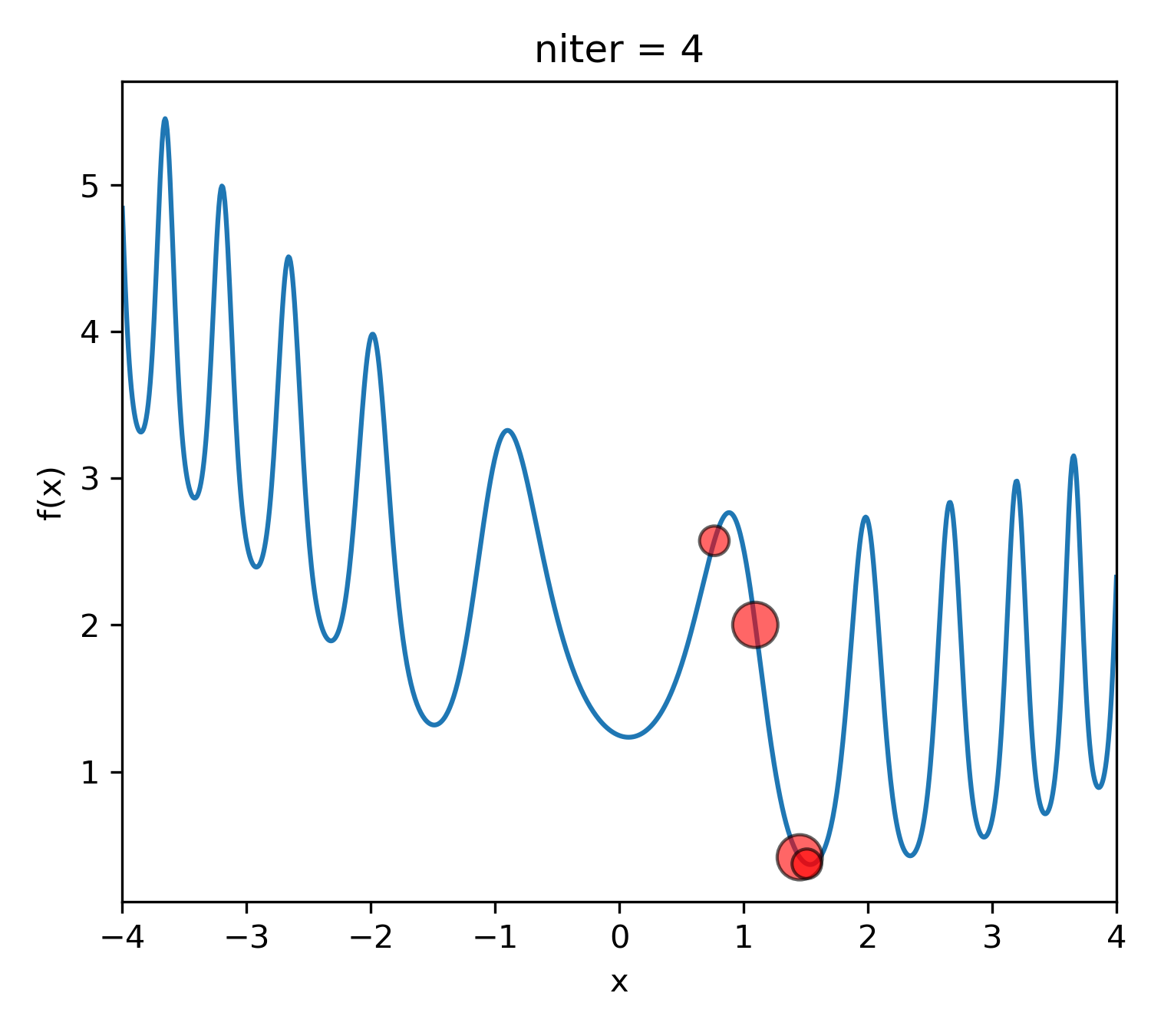}
		\includegraphics[width=0.15\textwidth]{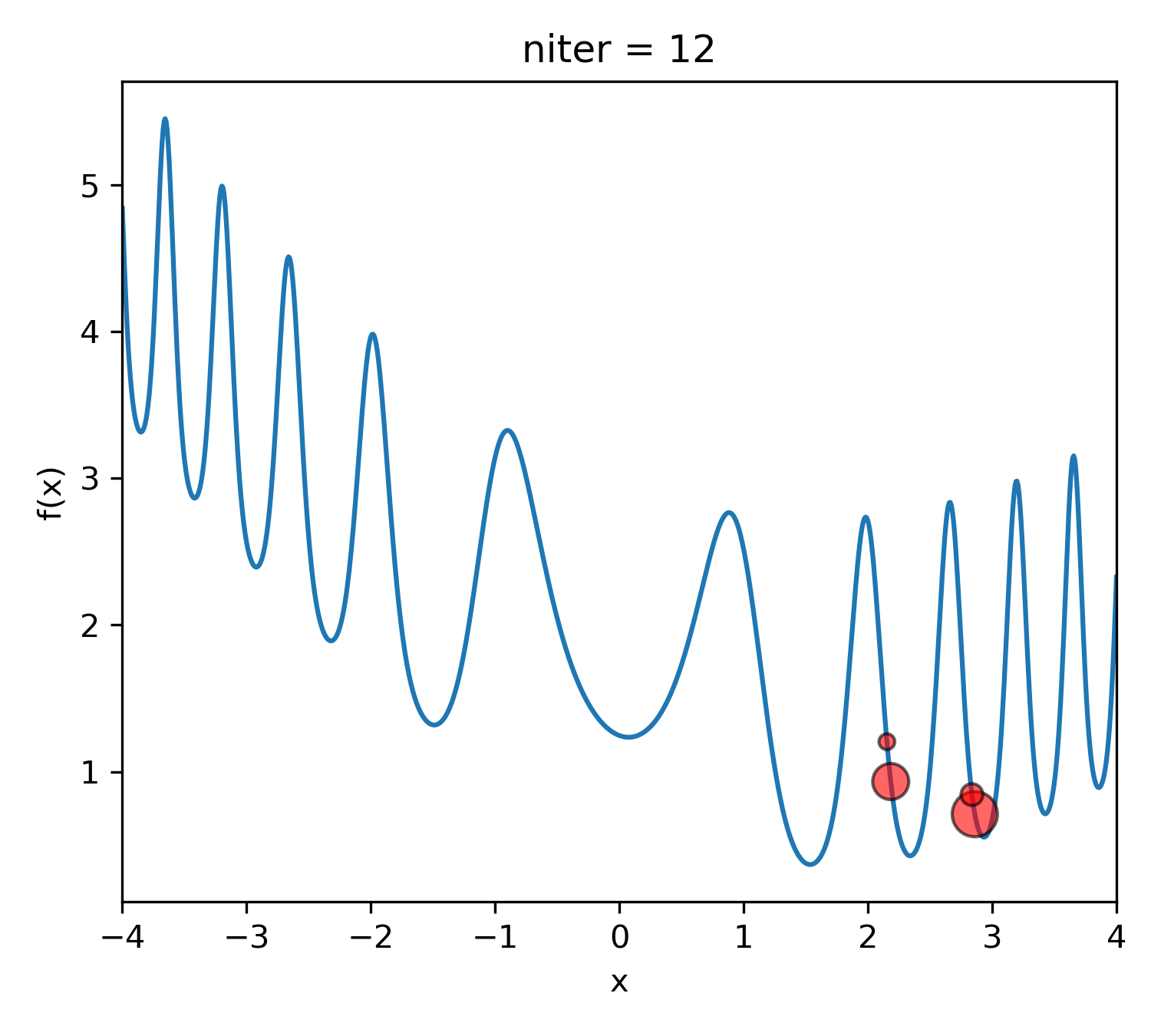}
		\includegraphics[width=0.15\textwidth]{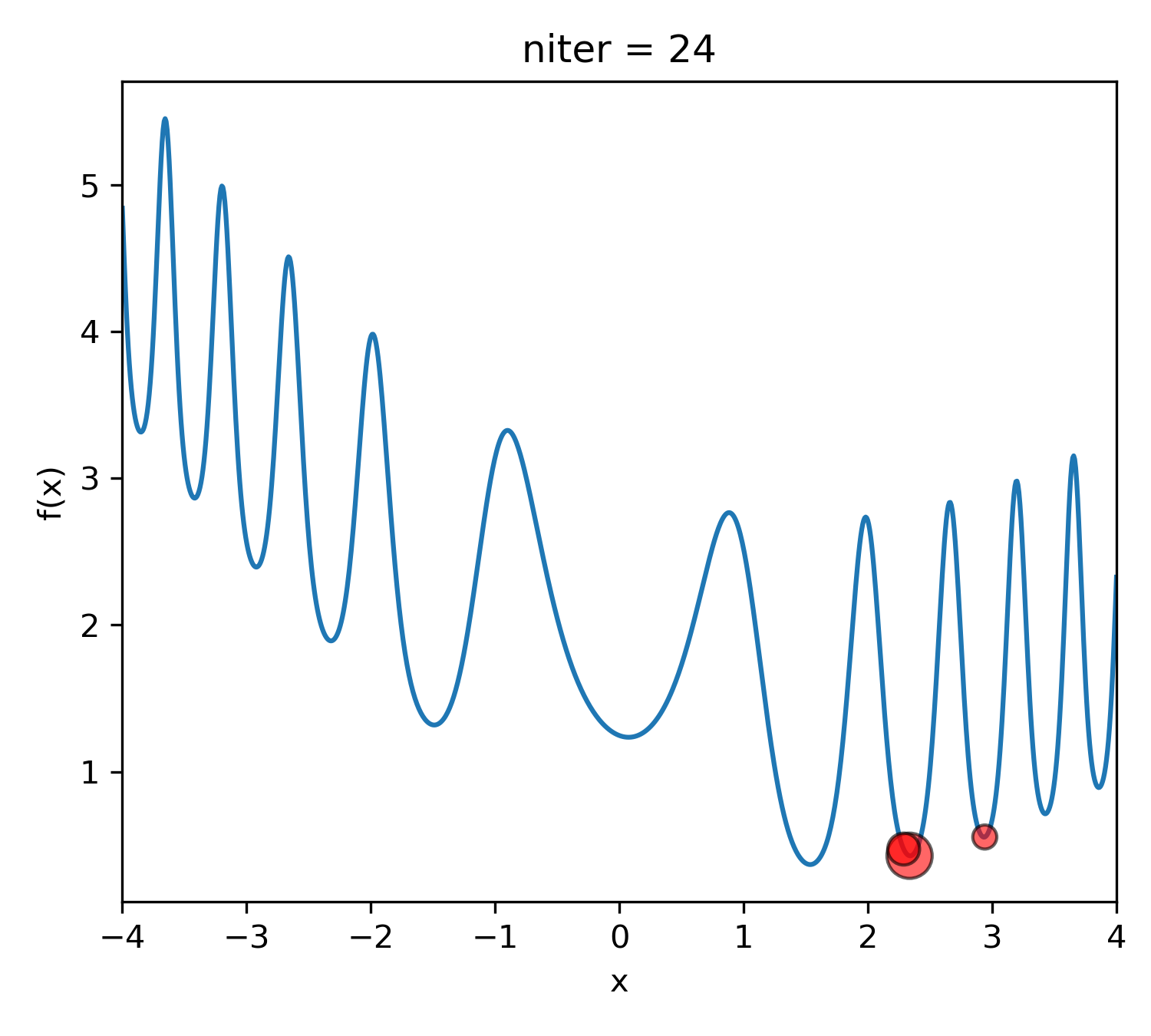}
		\includegraphics[width=0.15\textwidth]{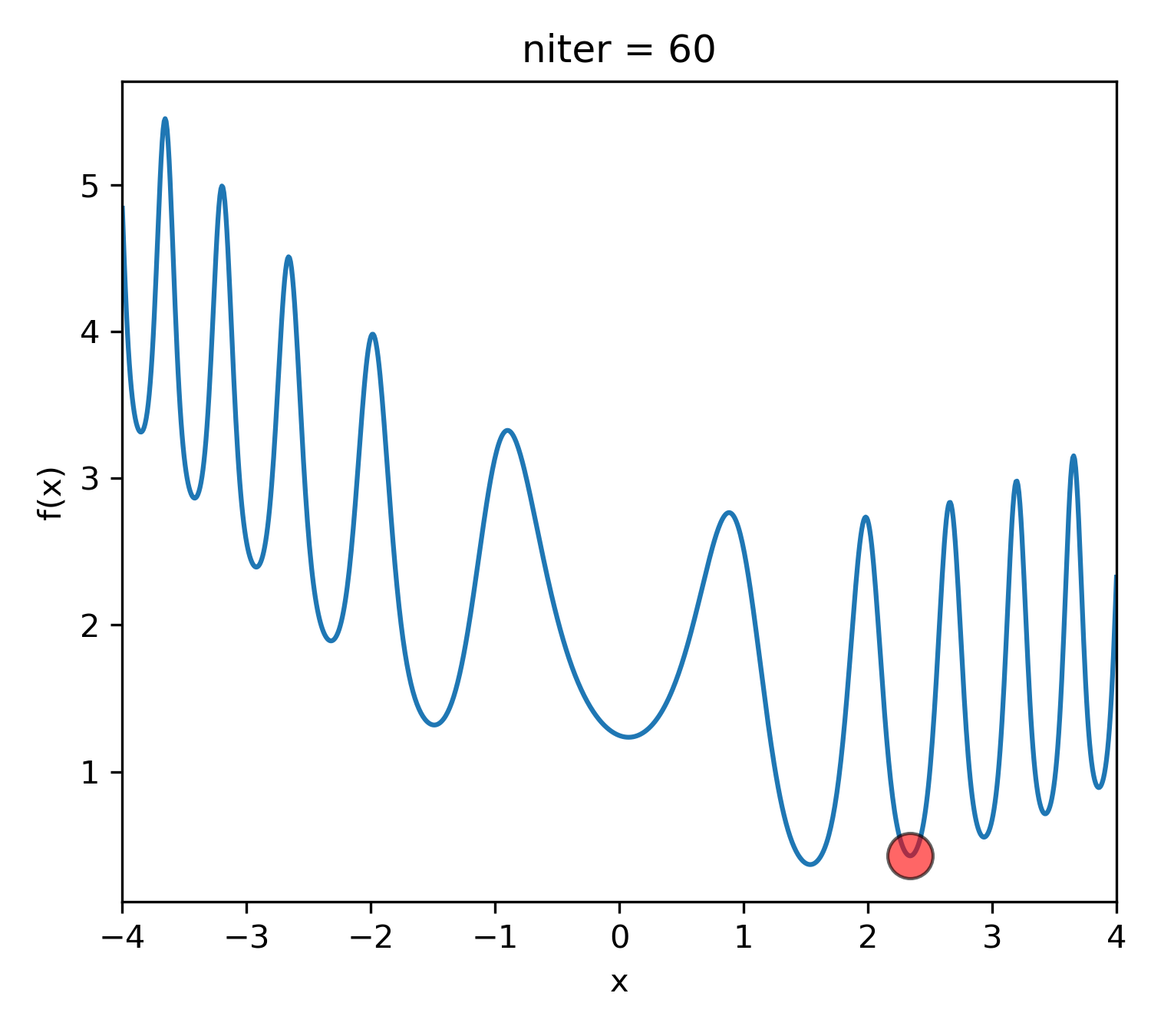}
	\end{center}
	\caption{Movement of agents with initial position $x = {\rm random}(-4, -2)$, velocity $v = {\rm random} (4, 5)$, $w_i = 10^{-4}$, $R = 1$, where the merging and removal strategy are implemented..}\label{fig:largev_goodw}
\end{figure}

\begin{figure}[H]
	\begin{center}
		\includegraphics[width=0.15\textwidth]{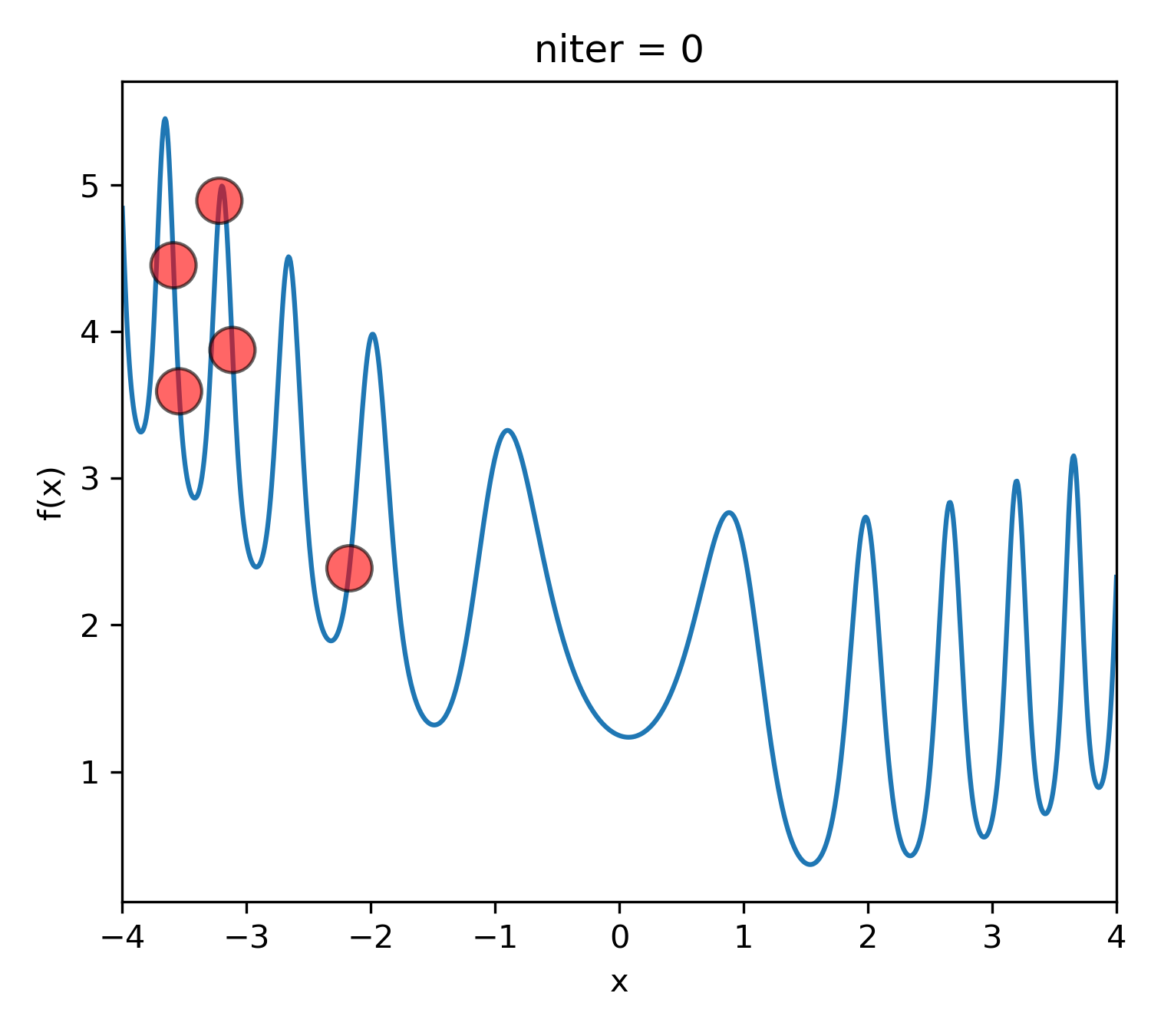}
		\includegraphics[width=0.15\textwidth]{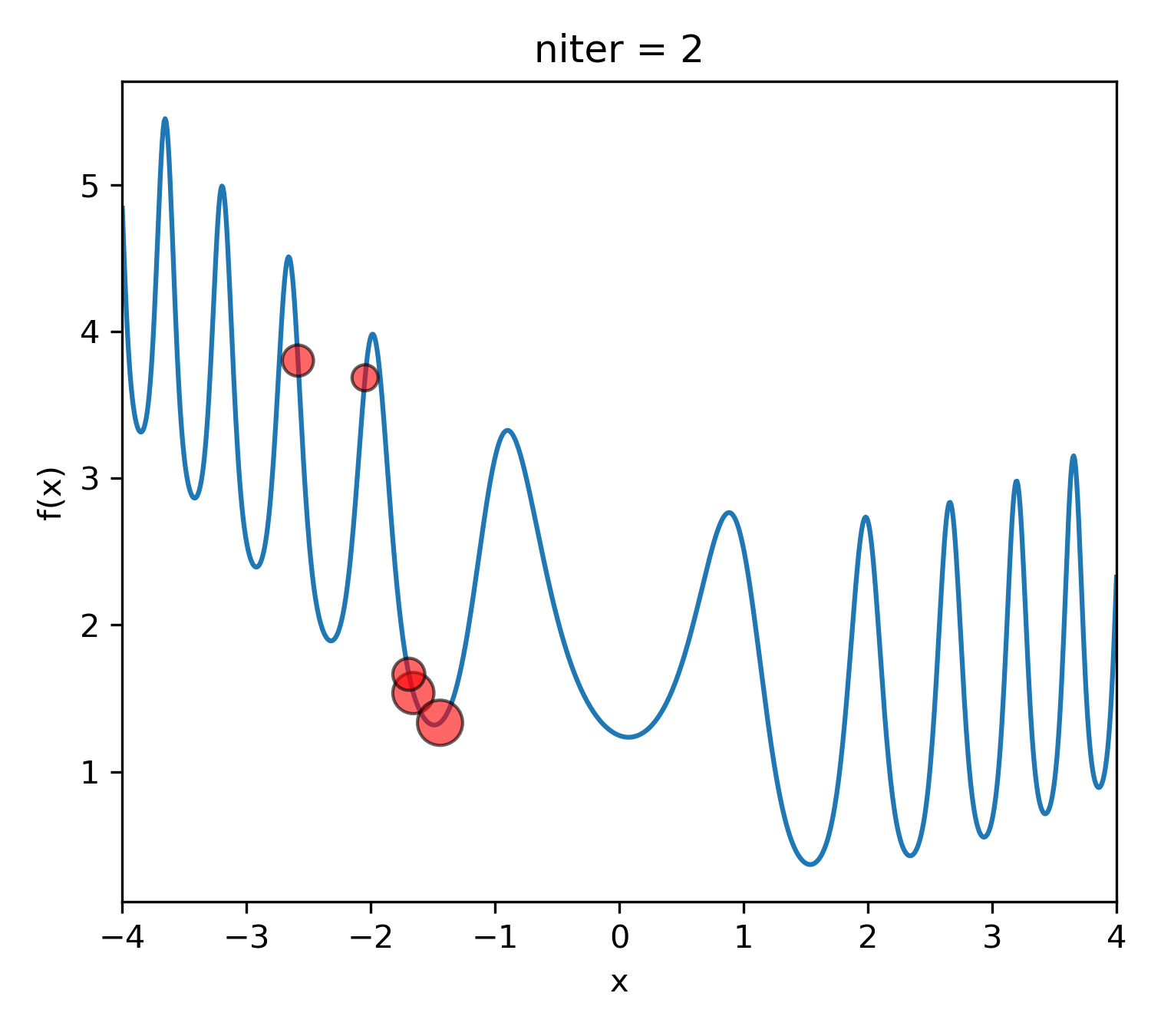}
		\includegraphics[width=0.15\textwidth]{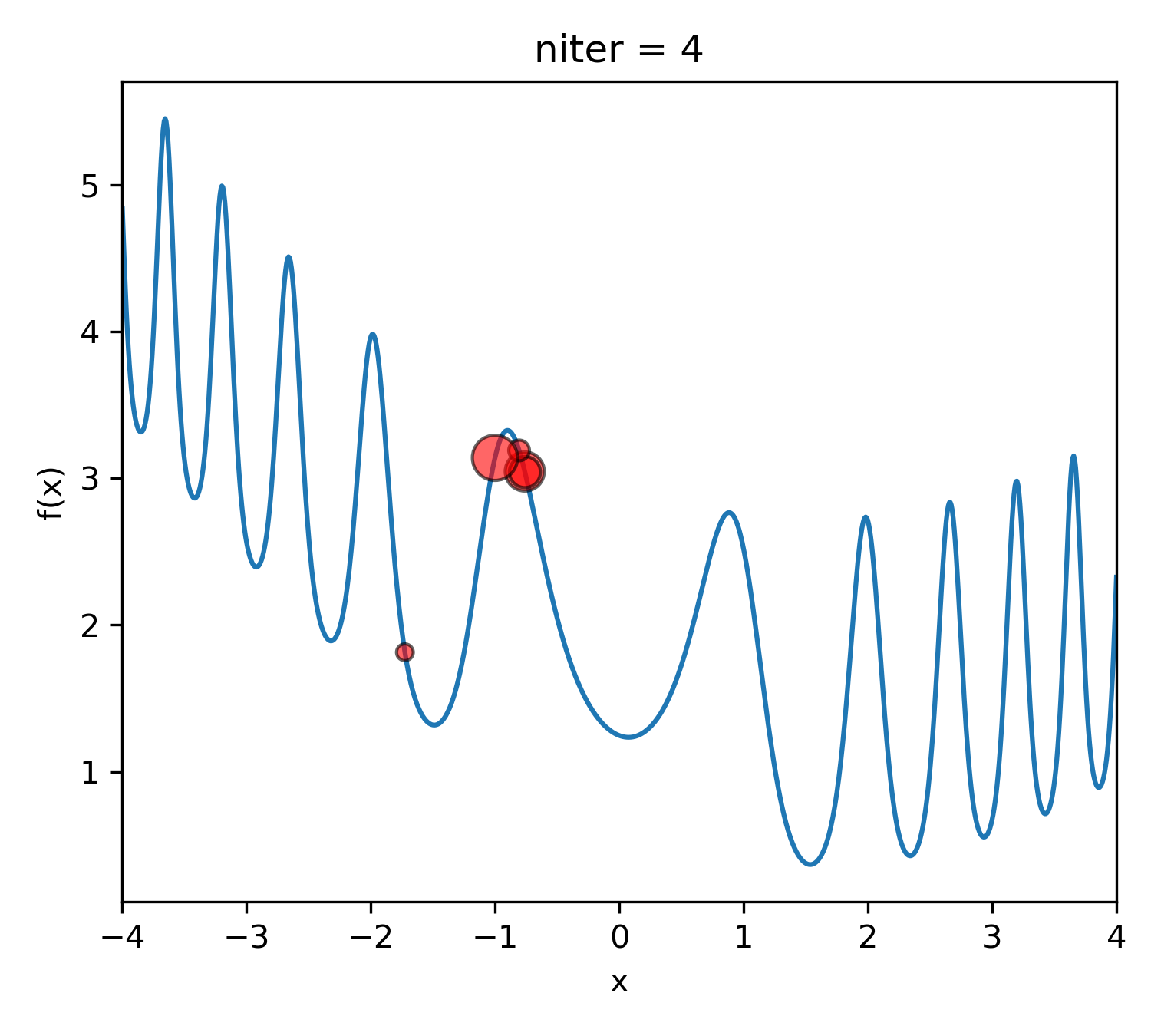}
		\includegraphics[width=0.15\textwidth]{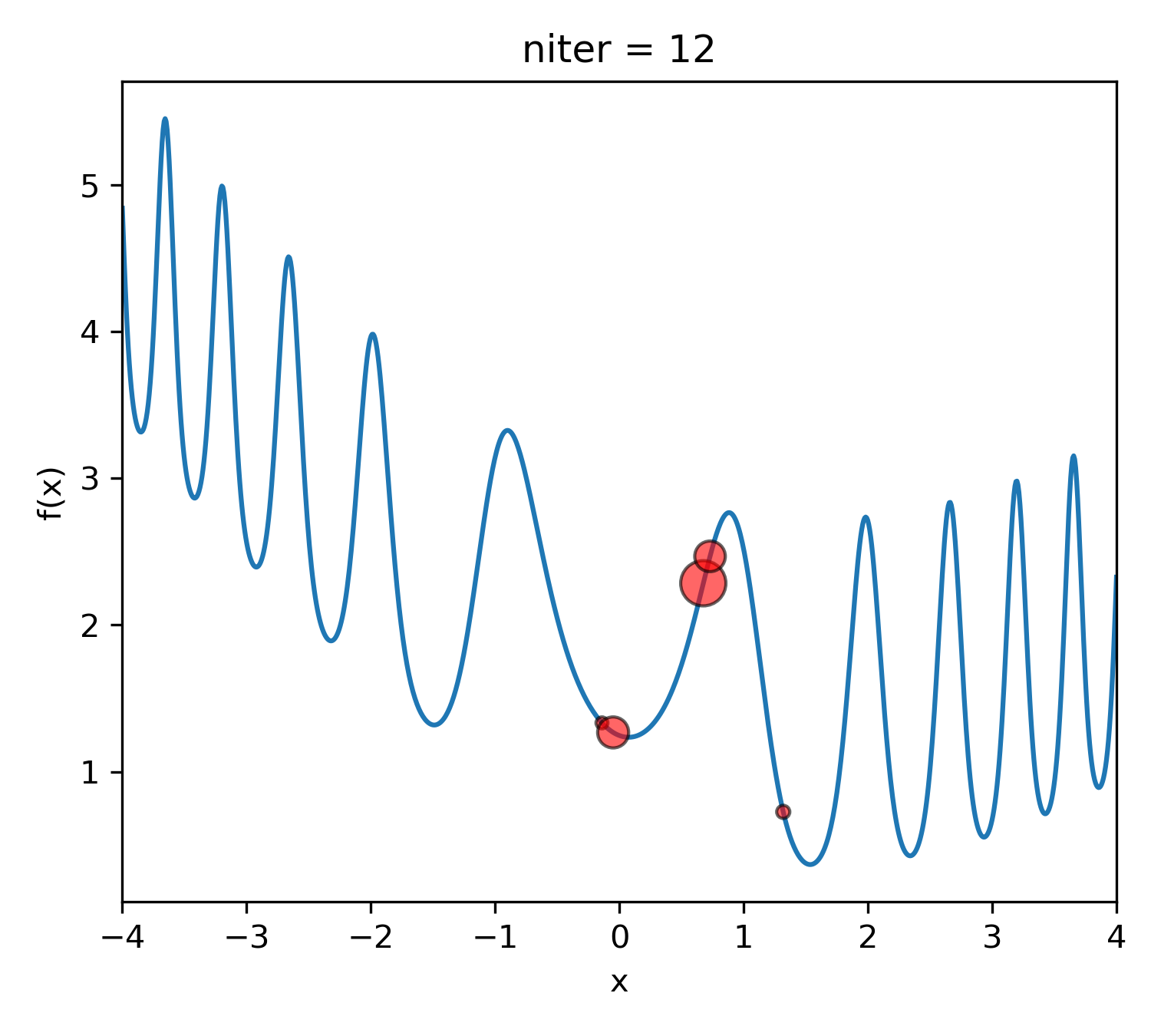}
		\includegraphics[width=0.15\textwidth]{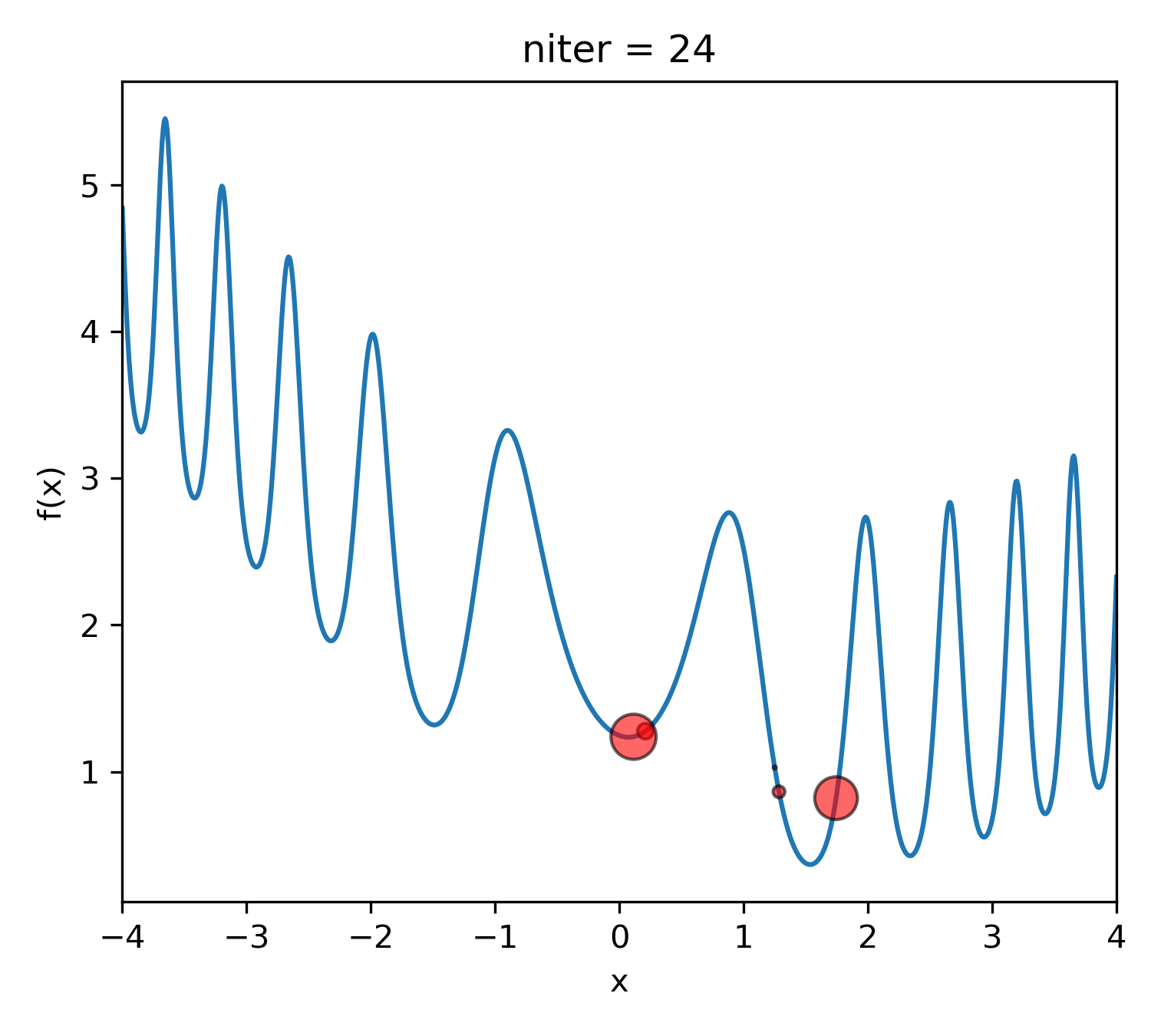}
		\includegraphics[width=0.15\textwidth]{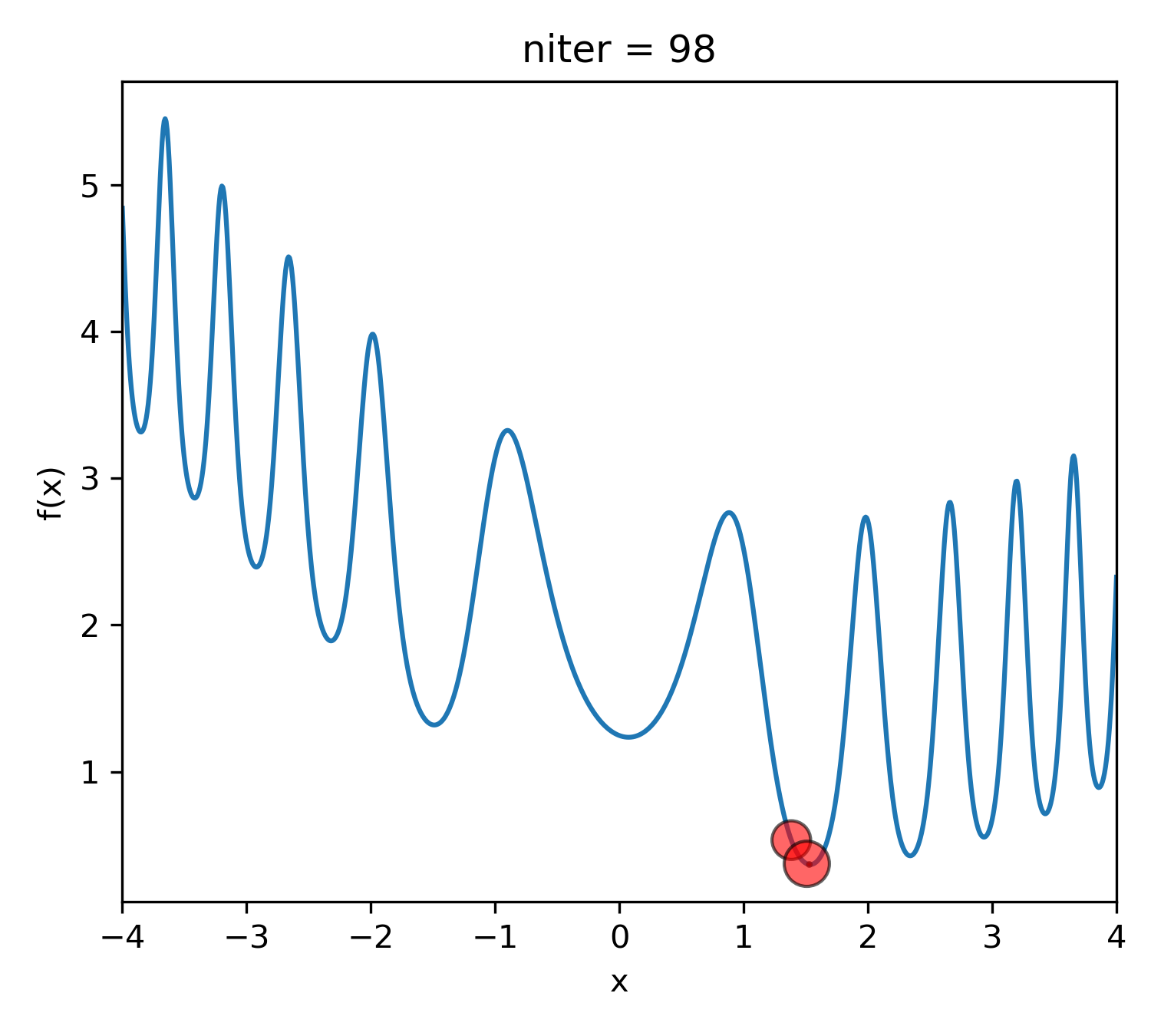}
	\end{center}
	\caption{Movement of agents with initial position $x = {\rm random}(-4, -2)$, velocity $v = {\rm random} (1, 2)$, $w_i = 10^{-5}$, $R = 0.6$, where the merging and removal strategy are implemented..}\label{fig:smallv_smallw}
\end{figure}

\begin{figure}[H]
	\begin{center}
		\includegraphics[width=0.15\textwidth]{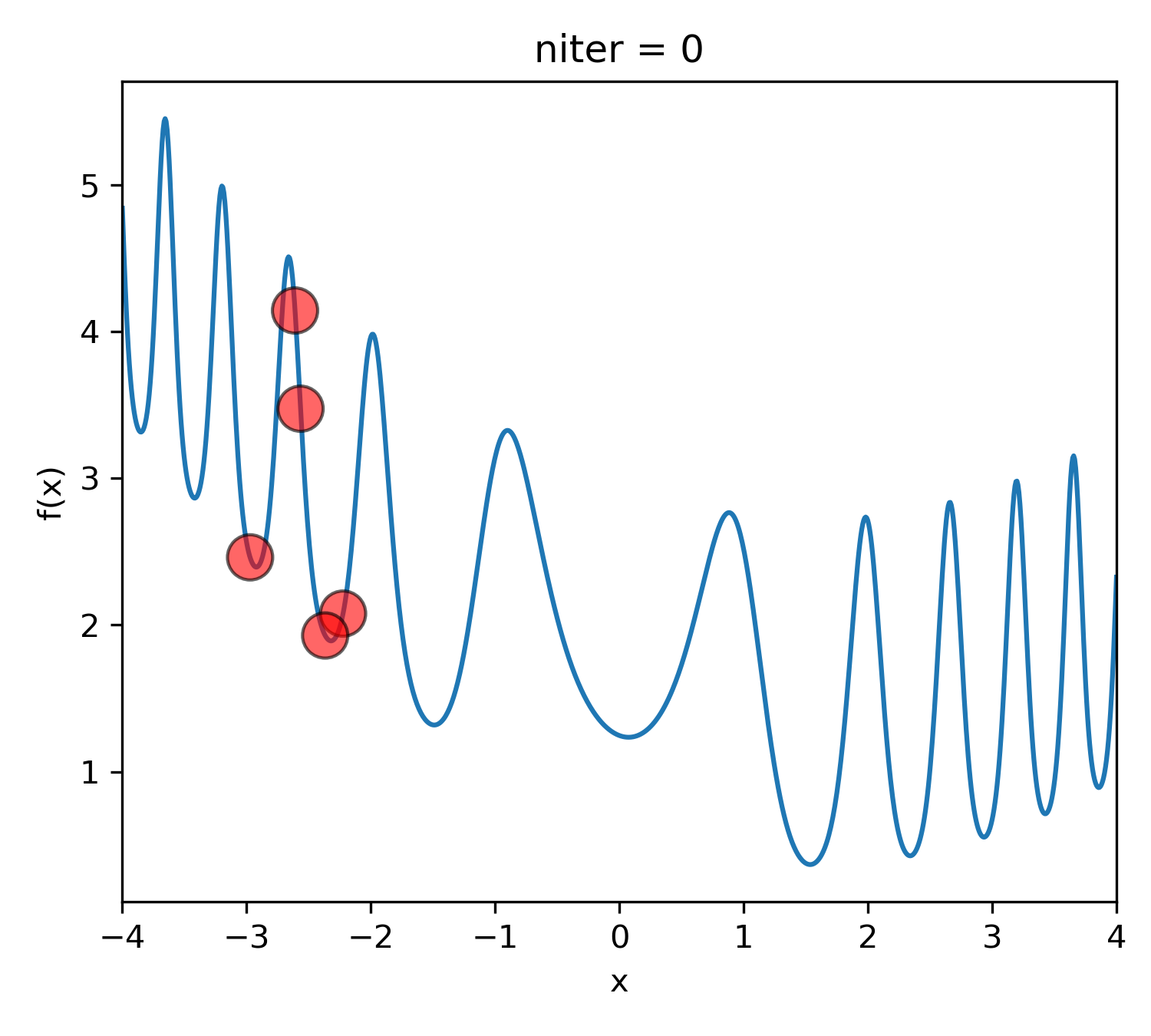}
		\includegraphics[width=0.15\textwidth]{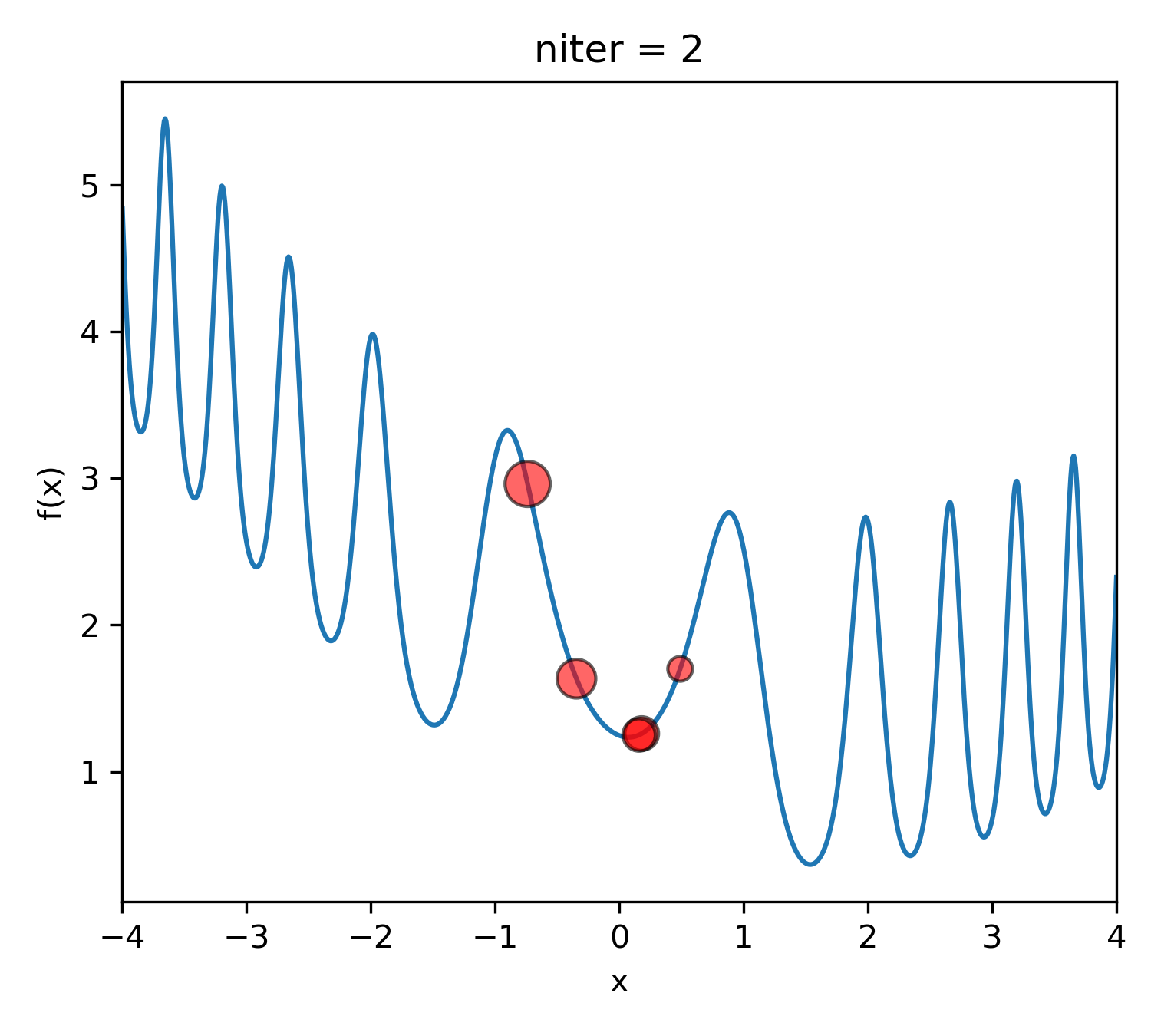}
		\includegraphics[width=0.15\textwidth]{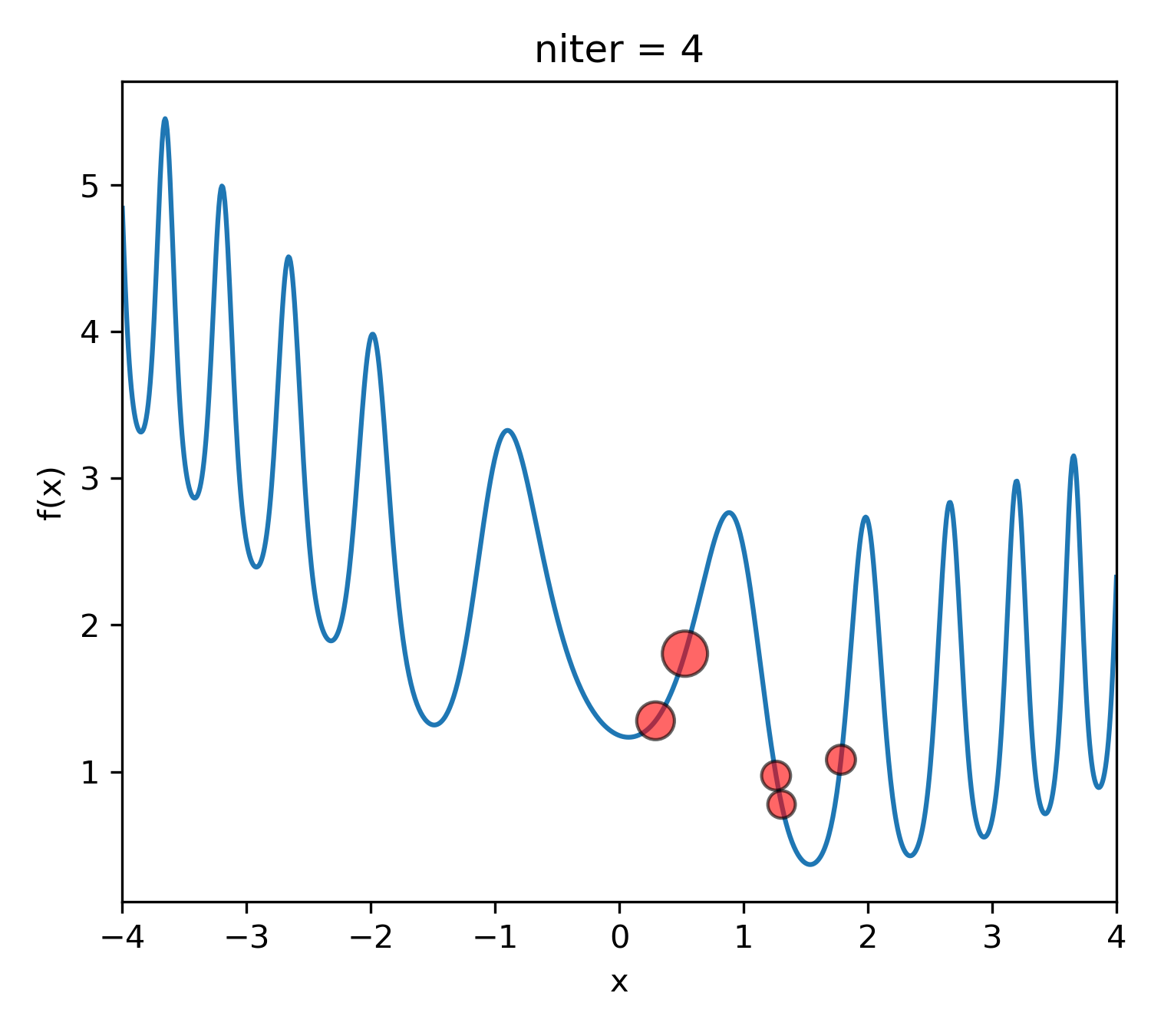}
		\includegraphics[width=0.15\textwidth]{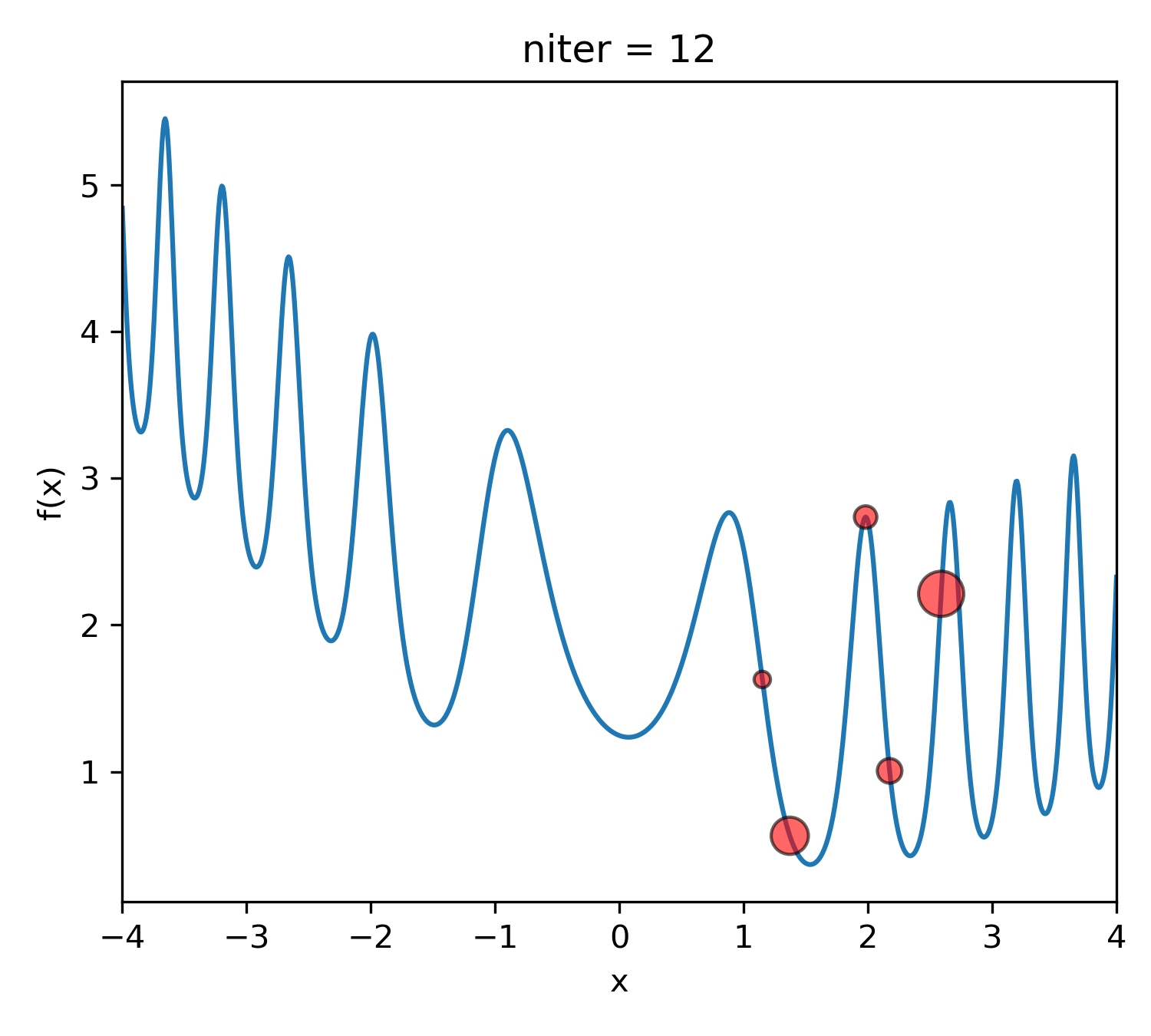}
		\includegraphics[width=0.15\textwidth]{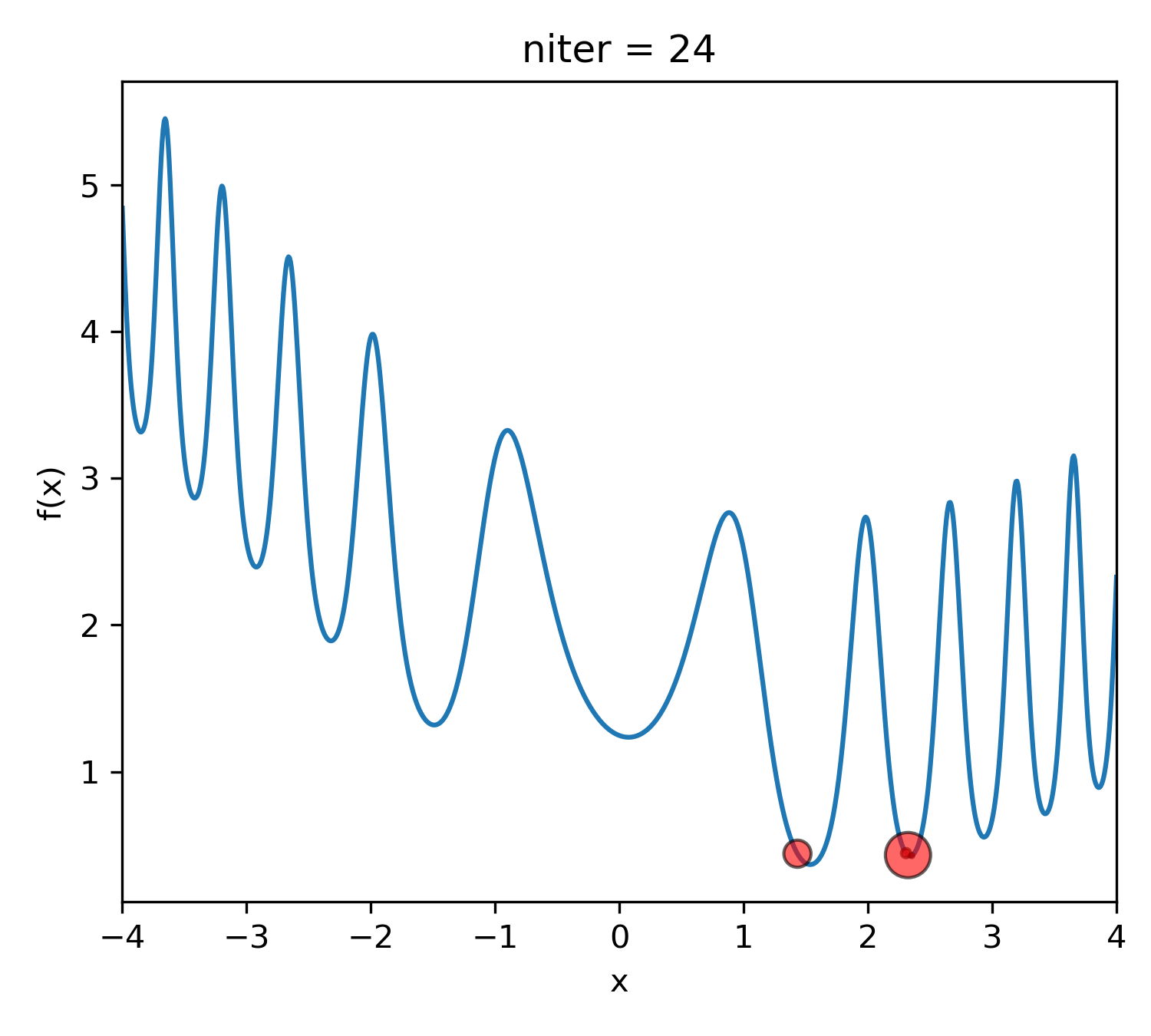}
		\includegraphics[width=0.15\textwidth]{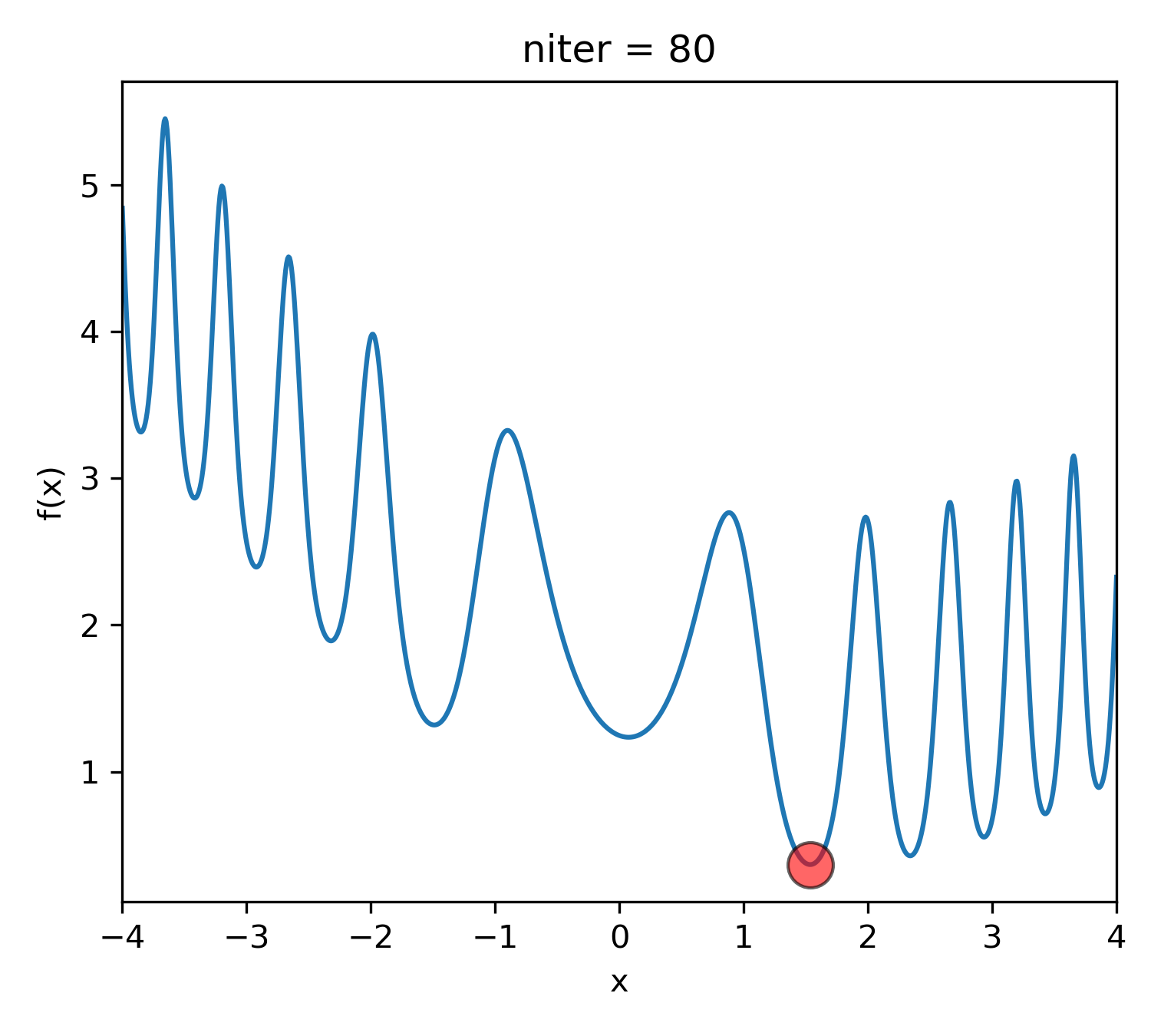}
	\end{center}
	\caption{Movement of agents with initial position $x = {\rm random}(-4, -2)$, velocity $v = {\rm random} (4, 5)$, $w_i = 10^{-3}$, $R = 1.2$, where the merging and removal strategy are implemented..}\label{fig:largev_largew}
\end{figure}

Figures \ref{fig:goodv_goodw}, \ref{fig:smallv_goodw}, and \ref{fig:largev_goodw} illustrate the trajectories of agents under various initial velocity conditions. In Figure \ref{fig:goodv_goodw}, the global minimum is successfully reached. Conversely, Figure \ref{fig:smallv_goodw} demonstrates that when the initial speed is too low, the agents exhaust their momentum before reaching the global minimum, leading to a failure in searching for the optimum. For this scenario, reducing $w_i$ and the friction coefficient $R$ allows the agents to preserve sufficient momentum to eventually find the global minimum, shown in \ref{fig:smallv_smallw}. In contrast, Figure \ref{fig:largev_goodw} reveals that an excessively high initial speed causes the agents to overshoot the target region. In such cases, increasing $w_i$ and the friction coefficient facilitates a more rapid dissipation of the total energy, thereby enhancing the chance of convergence toward the global minimum shown in \ref{fig:largev_largew}. As a side note, these parameters can be made time-dependent in the model without affecting the theoretical results alluded to earlier. As a result, one can devise a strategy to adjust their sizes during iterations.

\subsection{Optimization of a highly oscillatory objective function}

In this section, we apply the swarming algorithms for global optimization to the following highly oscillatory objective function:
\ben
F(x) = x \sin{(x)} \cos{(2x)} - 2 x \sin{(3x)} + 3x \sin{(4x)} + 0.1 x^2. \label{eq:objective-02}
\een
This function has a global minimum at $x^\star \approx 21.5627$, as illustrated in Figure \ref{fig:ex2_solution}.
\begin{figure}[H]
	\begin{center}
		\includegraphics[width=0.3\textwidth]{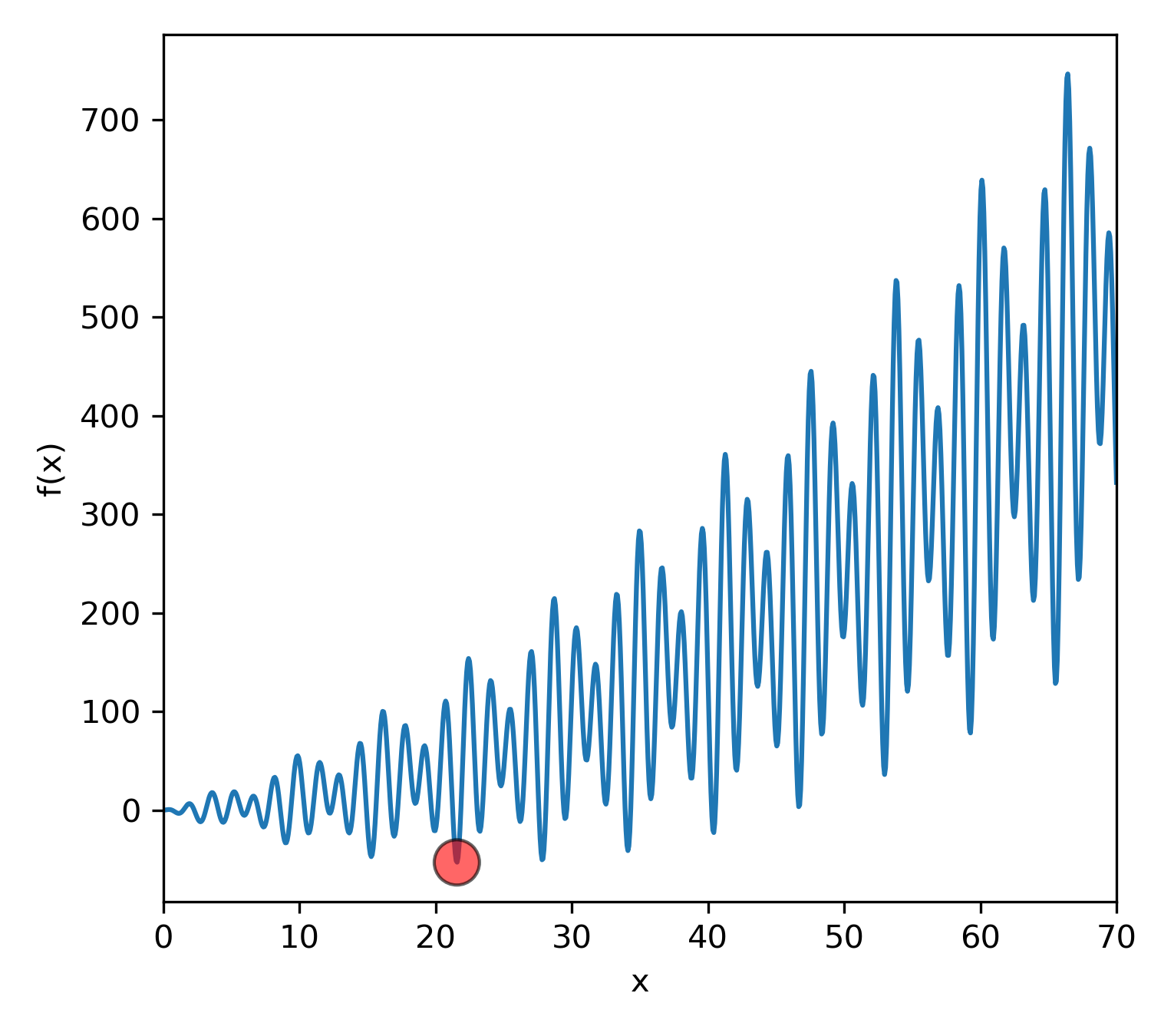}
	\end{center}
	\caption{Plot of objective function \eqref{eq:objective-02}}\label{fig:ex2_solution}
\end{figure}

We initialize the SBI-SIMEX scheme by deploying $20$ agents with initial positions $x_0 = {\rm random} (0, 5)$ and initial velocities $v_0 = {\rm random}(0, 40)$. The parameters used are $w_i = 10^{-4}$, $R=1$, $h = 0.5$.
Figure \ref{fig:ex2_res} shows the dynamics of the agents during the optimization process. Because of the inertia, some agents are allowed to wander over the landscape in a wider range so that they eventually converge to the global minimum. This demonstrates the advantage or even necessity of including inertia into the global optimization process in the swarming framework.
\begin{figure}[H]
	\begin{center}
		\includegraphics[width=0.25\textwidth]{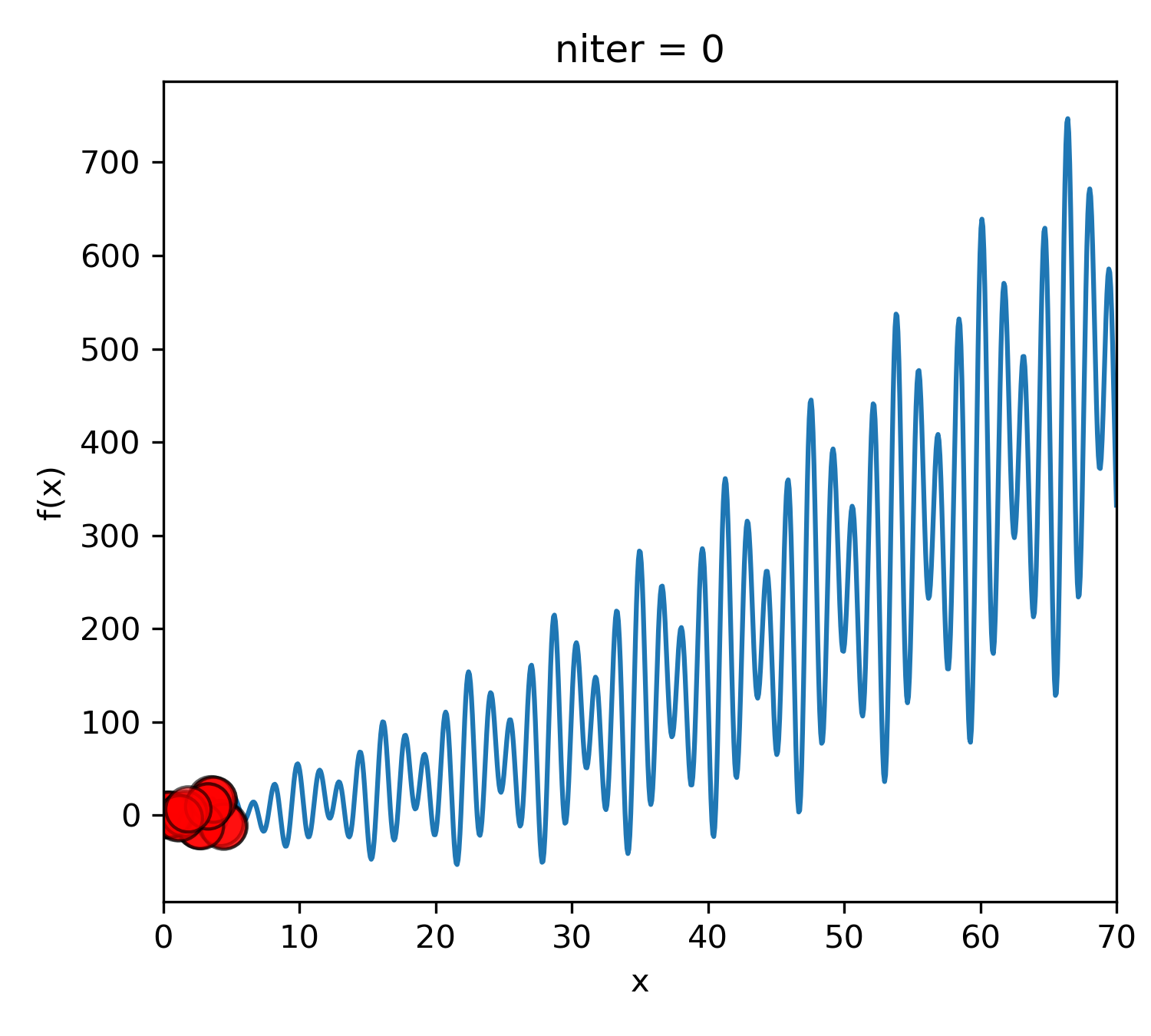}
		\includegraphics[width=0.25\textwidth]{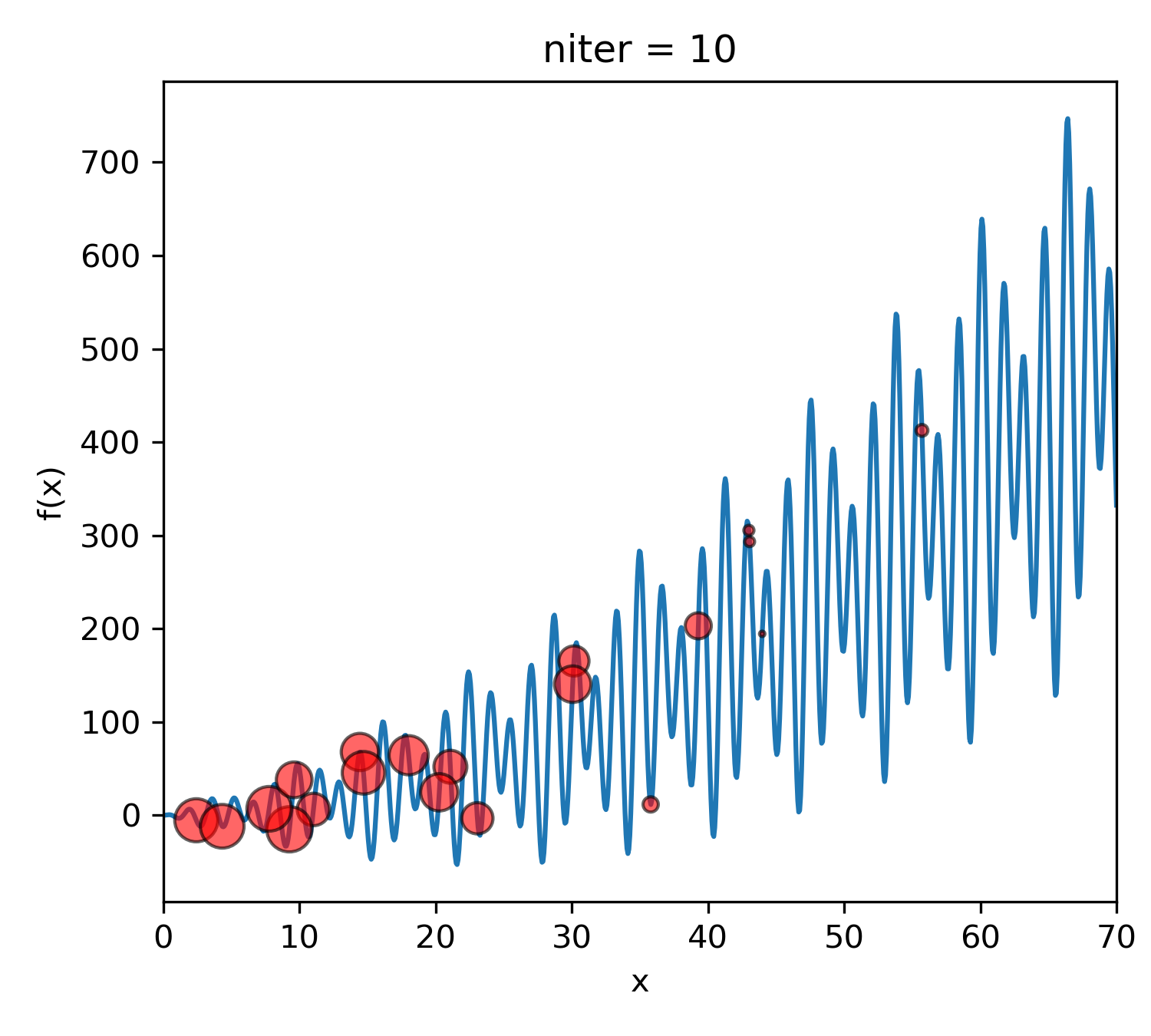}
		\includegraphics[width=0.25\textwidth]{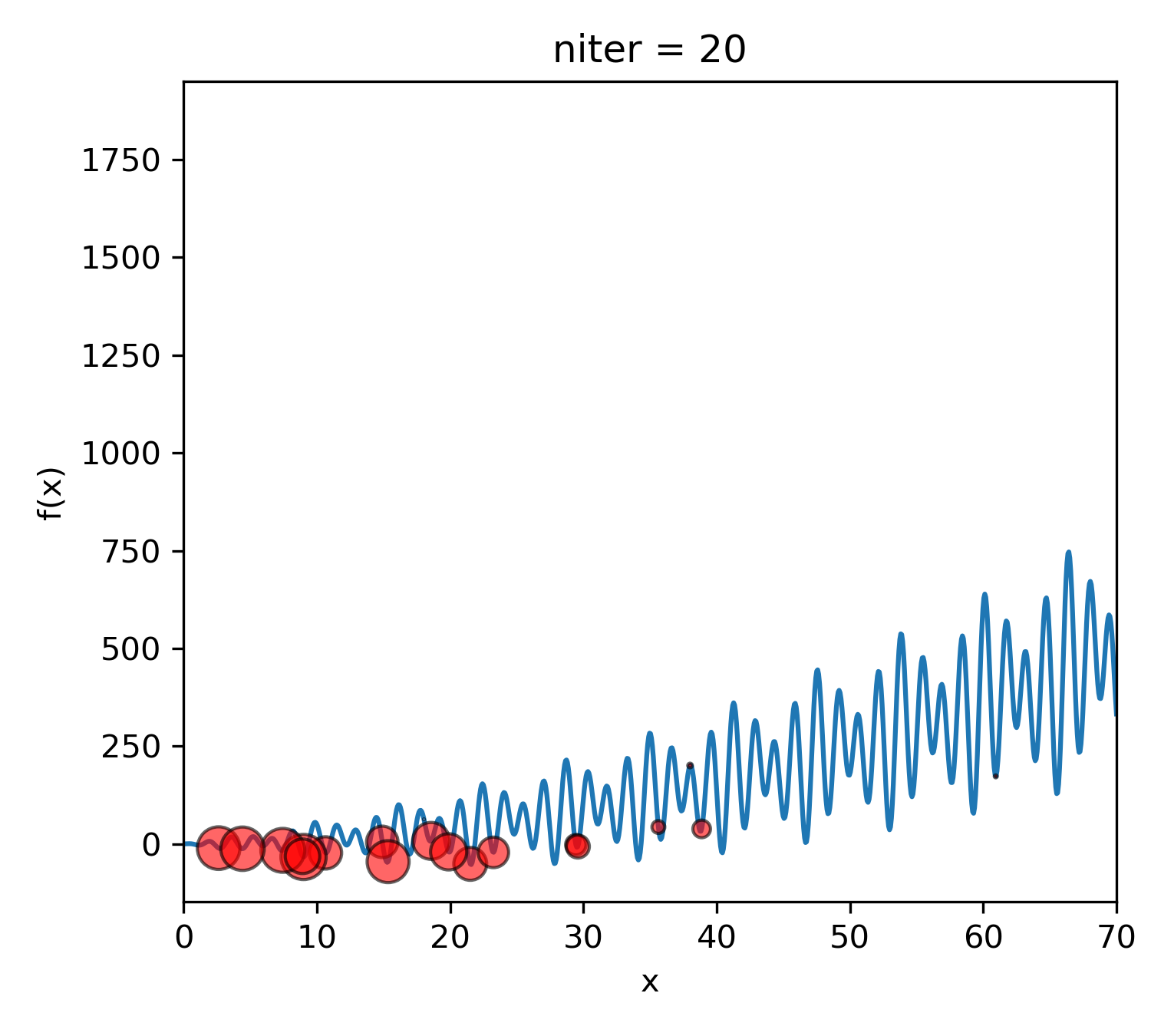}
		\includegraphics[width=0.25\textwidth]{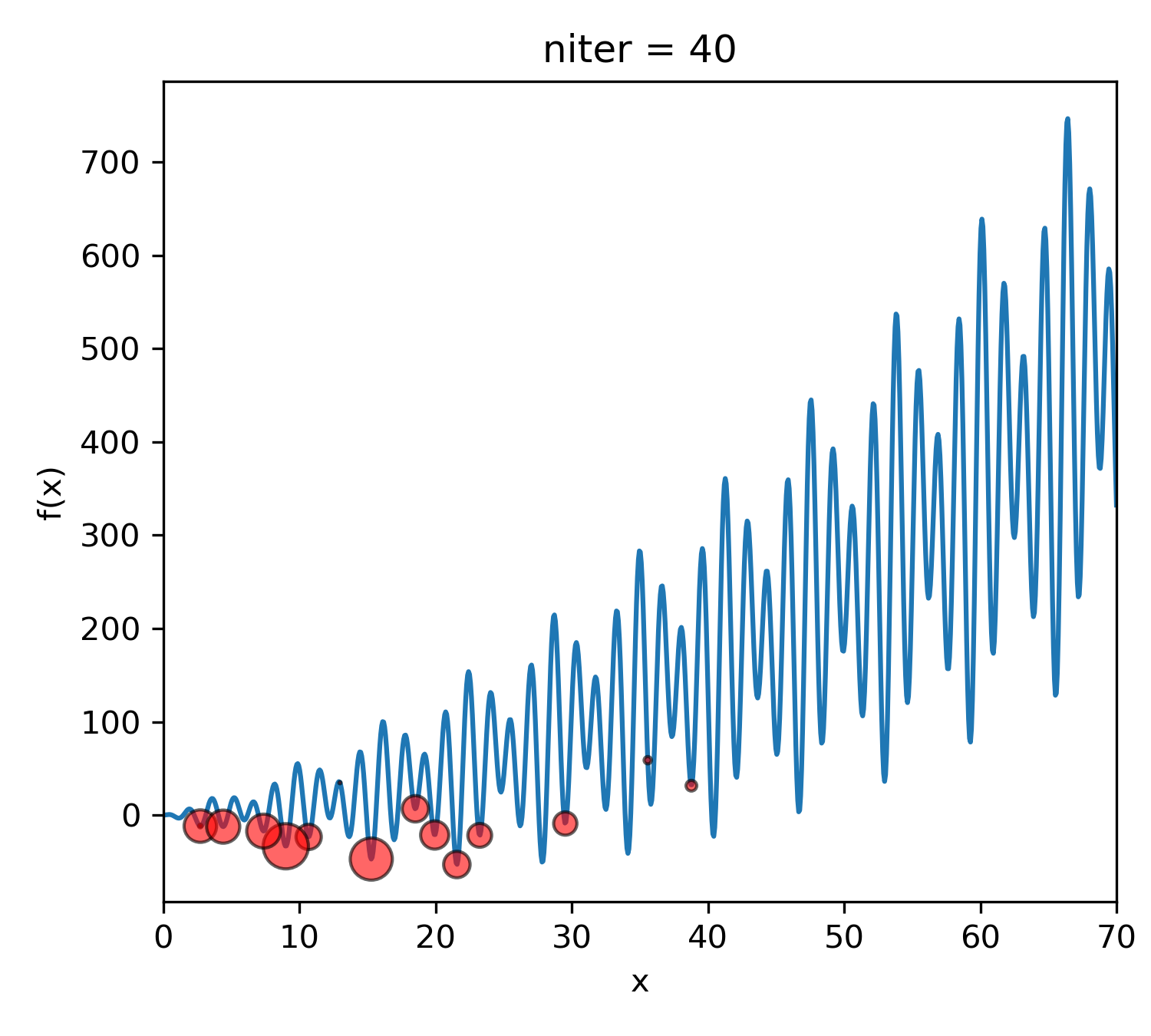}
		\includegraphics[width=0.25\textwidth]{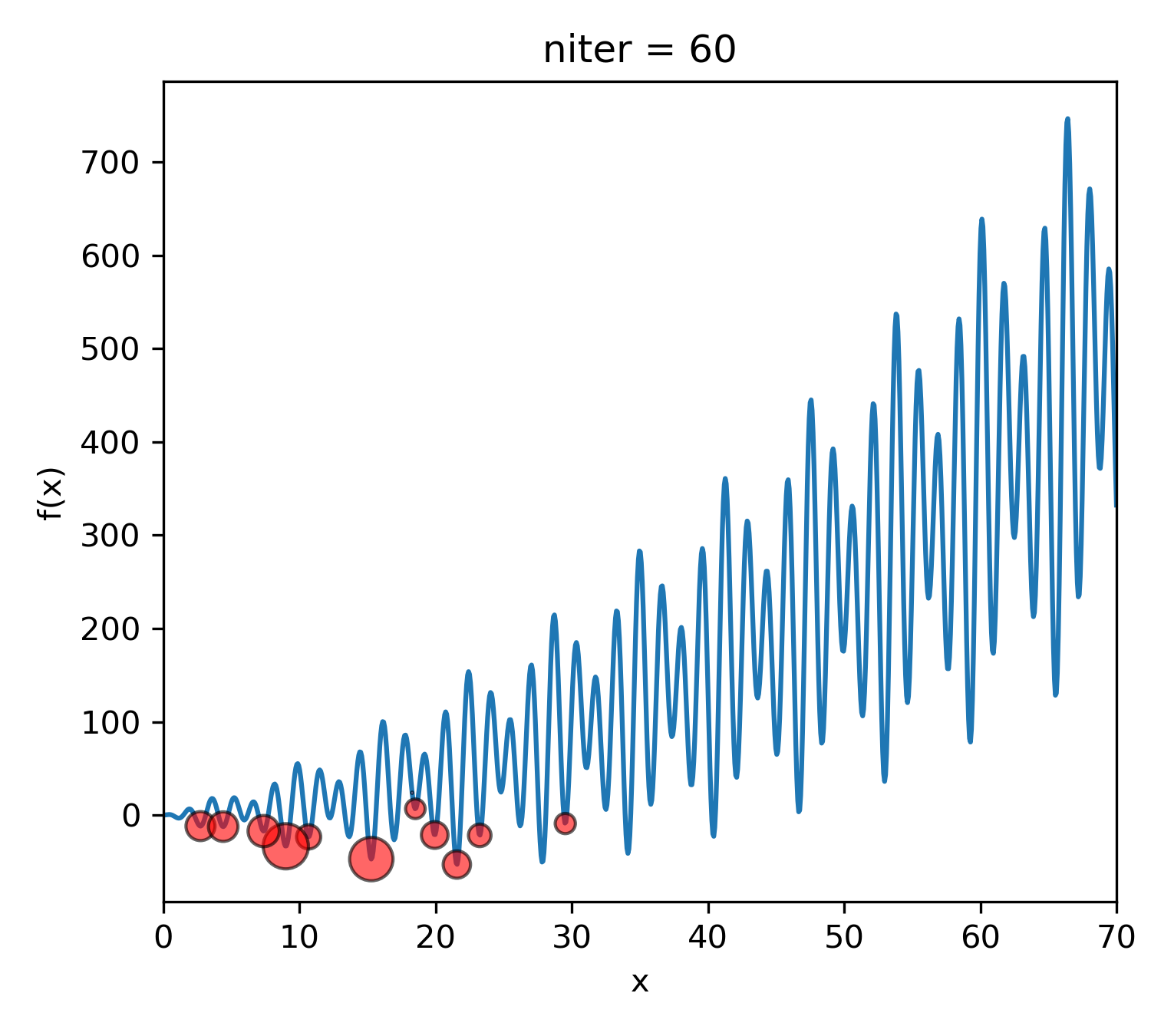}
		\includegraphics[width=0.25\textwidth]{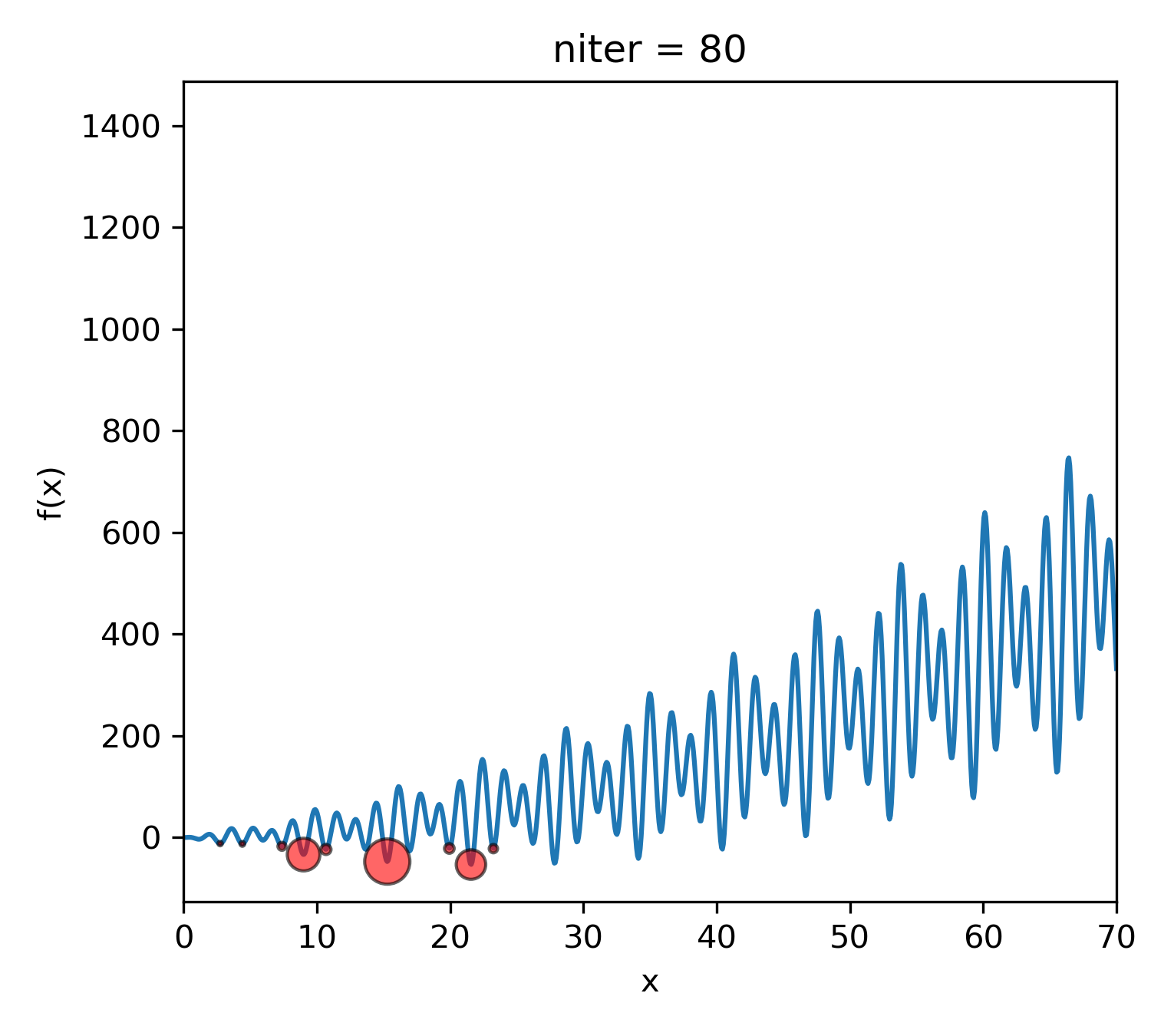}
		\includegraphics[width=0.25\textwidth]{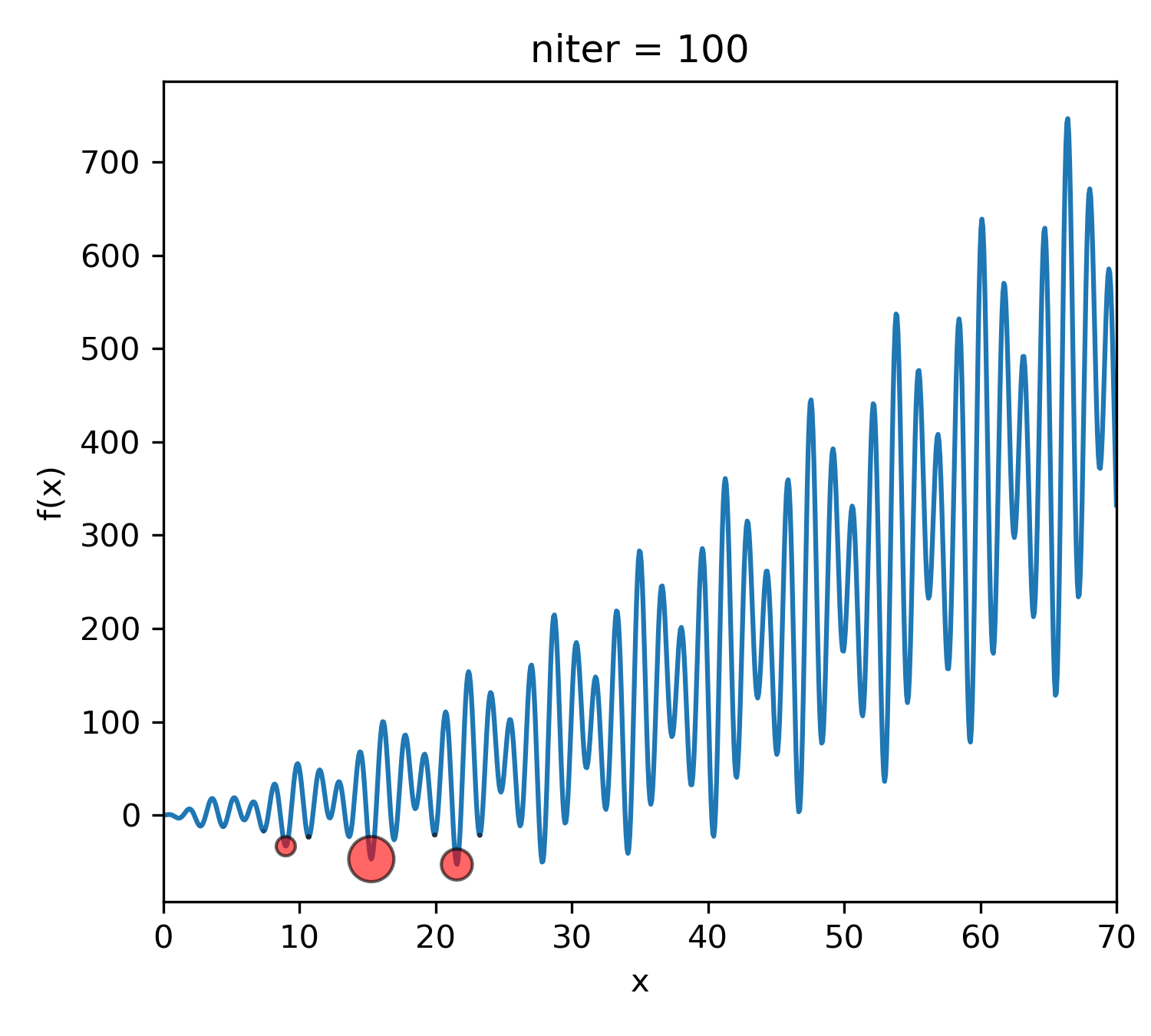}
		\includegraphics[width=0.25\textwidth]{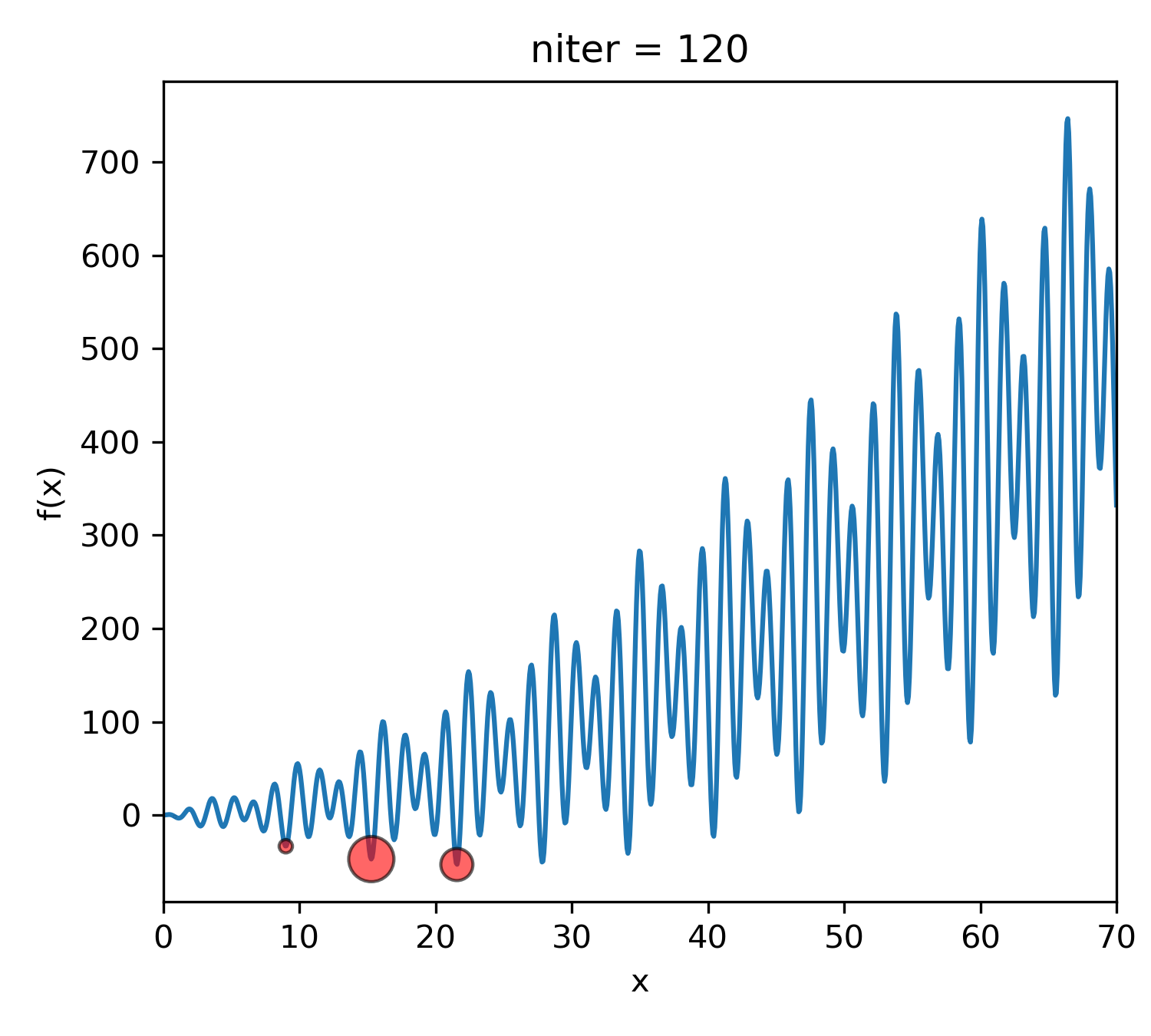}
		\includegraphics[width=0.25\textwidth]{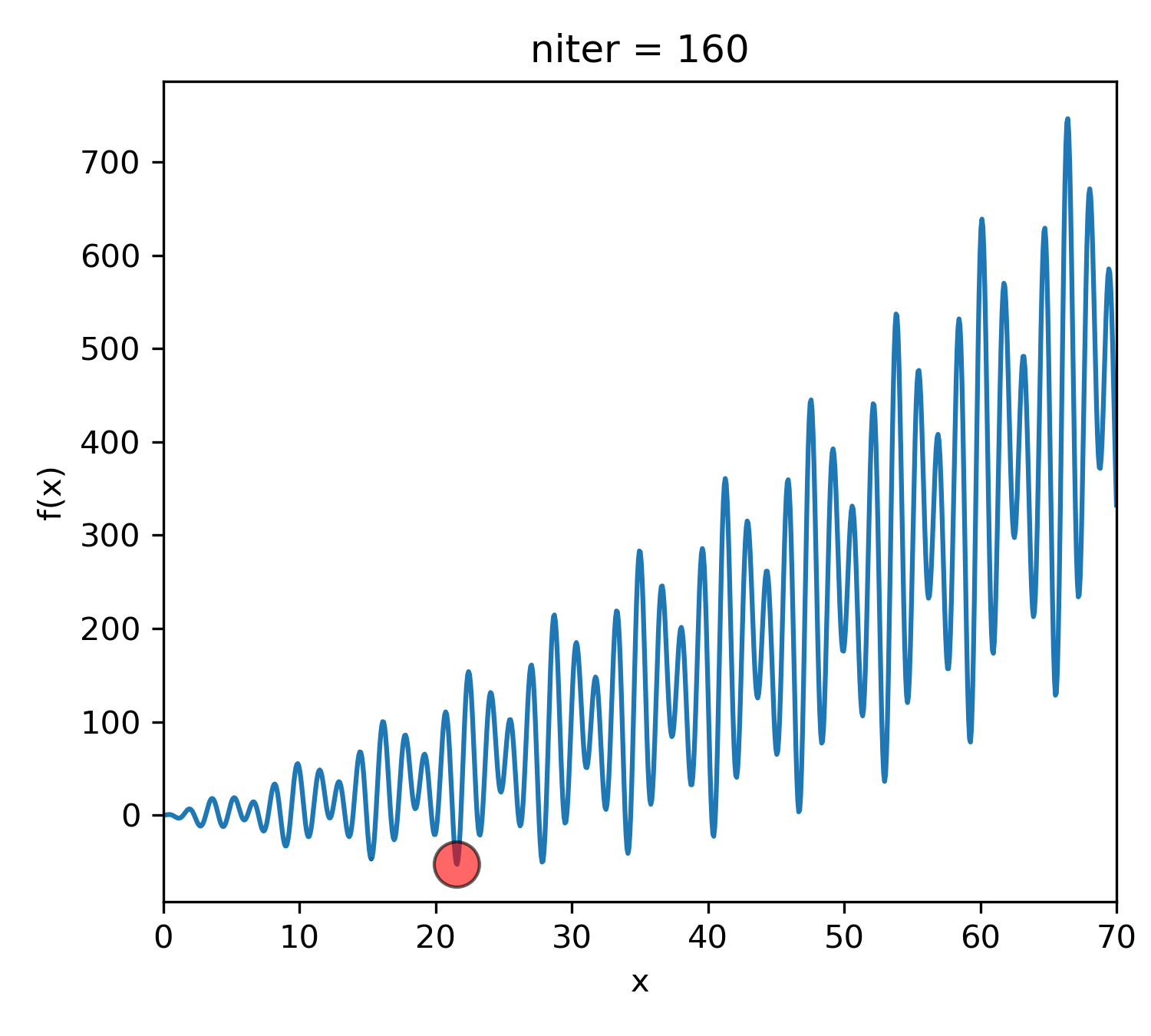}
	\end{center}
	\caption{Dynamics of the agents in the optimization process of function \eqref{eq:objective-02}.}\label{fig:ex2_res}
\end{figure}

\subsection{Optimization of high dimensional nonconvex functions}

We test the performance of the proposed algorithm on several multivariate functions in $d$-dimensions. We consider three benchmark cases using the Rastrigin, Rosenbrock, Styblinski-Tang objective functions in $d$-dimensions. These functions are defined as follows.

\begin{equation*}
	F_{\text{Rastrigin}}(\mathbf{x})
	= 10d + \sum\limits_{i=1}^{d}
	\left(x_i^2 - 10 \cos(2\pi x_i)\right),
\end{equation*}
\begin{equation*}
	F_{\text{Rosenbrock}}(\mathbf{x})
	= \sum_{i=1}^{d-1}
	\left(100\left(x_{i+1} - x_i^2\right)^2 + \left(1 - x_i\right)^2\right),
\end{equation*}
and
\begin{equation*}
	F_{\text{ST}}(\mathbf{x})
	= \frac{1}{2} \sum_{i=1}^{d}
	\left(x_i^4 - 16 x_i^2 + 5 x_i\right).
\end{equation*}

We note that the global minimum of the Rastrigin functions is $(0, \cdots, 0)^\top$, the global minimum of the Rosenbrock function is $(1, \cdots, 1)^\top$ and the global minimum of the Styblinski-Tang function is $(-2.903534, \cdots, -2.903534)^\top$. Figure \ref{fig:landscape2d} depicts the landscapes of these functions and their global minimum in case $d = 2$, respectively.

In our numerical experiments, we primarily compare the SBI-SIMEX scheme with mass conservation against the SBGD method, as the performance of the remaining approaches is essentially similar with that of SBI-SIMEX. Additionally, we evaluate a strategy that incorporates stochasticity. In our implementation, each agent first checks whether the function value at its updated position is smaller than that of the previous iteration; if it is, the update is accepted. Otherwise, a probabilistic acceptance criterion is applied based on the agent’s quality level. Specifically, high-quality agents are considerably less likely to accept inferior updates, while low-quality agents are more inclined to do so, thereby broadening the search region. The acceptance probability is determined by the function $P(m) = \tfrac{1}{2} - \tfrac{1}{2}{\rm tanh}(1000(m- \beta))$; a random number uniformly distributed in the interval \([0, 1]\) is generated and compared with the predetermined acceptance probability to decide whether the update should be adopted. The SBI-SIMEX method with this stochastic strategty iss hereby abbreviated as RSBI-SIMEX.

\begin{figure}[H]
	\begin{center}
		\includegraphics[width=0.3\textwidth]{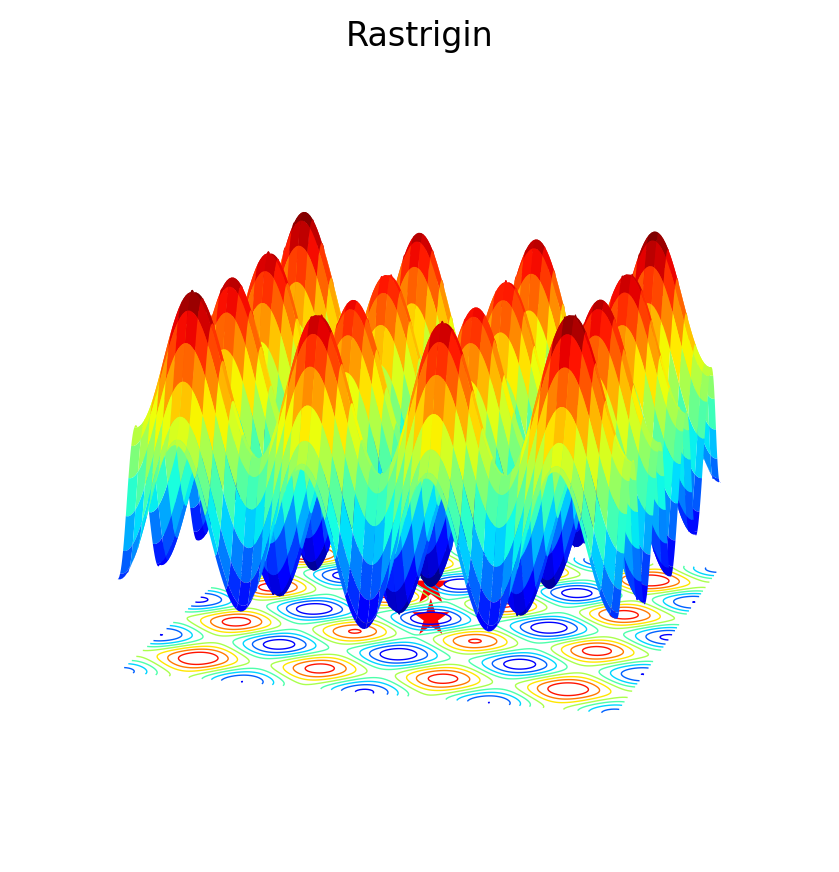}
		\includegraphics[width=0.3\textwidth]{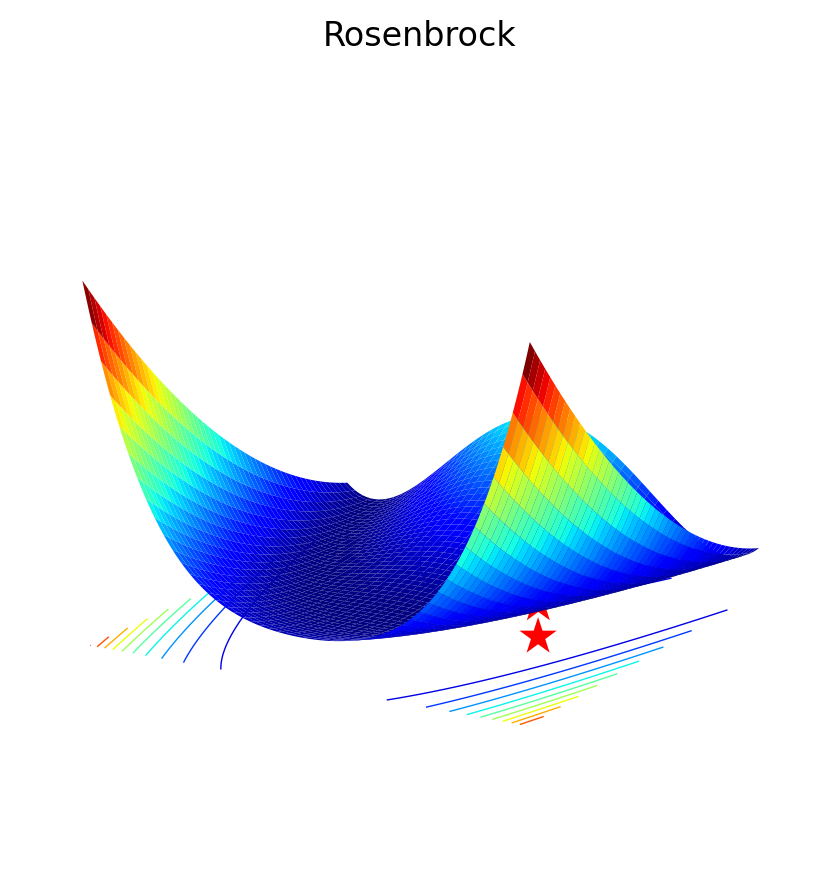}
		\includegraphics[width=0.3\textwidth]{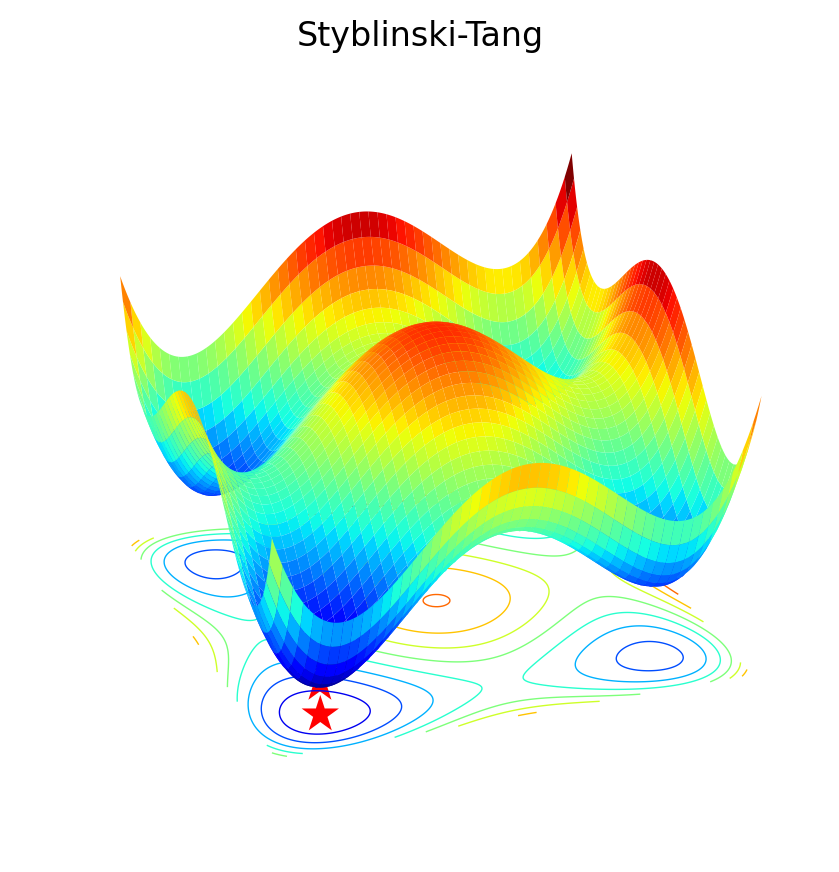}
	\end{center}
	\caption{Landscapes and the global minimum of the Rastrigin, Rosebrock and Styblinski-Tang functions in 2D, respectively.}\label{fig:landscape2d}
\end{figure}

Tables \ref{tab:ackley_comparison}--\ref{tab:st_comparison} summarize the comparison between the two methods. In all three cases, SBI-SIMEX delivers superior performance in searching for the global minimum. As the dimension increases, the relative success rate (success-rate-of-SBI-SIMEX vs success-rate-of-SBGD) improves significantly, especially in the optimization of the Rosenbrock function, attesting the superior power of the SBI approach. Moreover, in the case of the Rosenbrock function, the RSBI-SIMEX method demonstrates an enhanced success rate. In contrast, for the two remaining functions, the success rate achieved by the RSBI-SIMEX method is comparable to that of the SBI-IEMX method.
\begin{table}[H]
	\caption{Success rates of SBI-SIMEX and SBGD methods for global optimization of the Rastrigin function in various dimensions based on 1000 runs with uniformly generated initial position within $[-3, -1]^d$ and initial velocity within $[0, 4]^d$. The results of SBGD are from \cite{SwarmGDRandom}.}
	\label{tab:ackley_comparison}
	\begin{center}
		\resizebox{\textwidth}{!}{
			\begin{tabular}{c ccc ccc ccc ccc}
				\hline
				\multirow{2}{*}{$d$} & \multicolumn{3}{c}{N = 10} & \multicolumn{3}{c}{N = 25} & \multicolumn{3}{c}{N = 50} & \multicolumn{3}{c}{N = 100}                                                                                            \\
				\cmidrule(lr){2-4} \cmidrule(lr){5-7} \cmidrule(lr){8-10} \cmidrule(lr){11-13}
				                     & SBI-SIMEX                  & RSBI-SIMEX                 & SBGD                       & SBI-SIMEX                   & RSBI-SIMEX & SBGD   & SBI-SIMEX & RSBI-SIMEX & SBGD   & SBI-SIMEX & RSBI-SIMEX & SBGD    \\
				\hline
				2                    & 46.5\%                     & 41.5\%                     & 28.0\%                     & 81.8\%                      & 84.7\%     & 67.8\% & 95.9\%    & 97.0\%     & 95.3\% & 99.7\%    & 99.9\%     & 100.0\% \\
				3                    & 19.4\%                     & 13.6\%                     & 5.6\%                      & 36.2\%                      & 37.3\%     & 13.6\% & 58.0\%    & 62.8\%     & 28.6\% & 76.3\%    & 84.5\%     & 52.0\%  \\
				4                    & 4.1\%                      & 5.4\%                      & 1.0\%                      & 11.6\%                      & 10.7\%     & 3.9\%  & 19.6\%    & 24.5\%     & 5.7\%  & 31.0\%    & 38.6\%     & 11.4\%  \\
				5                    & 0.8\%                      & 1.2\%                      & 0.0\%                      & 4.0\%                       & 2.9\%      & 0.4\%  & 3.7\%     & 7.6\%      & 0.4\%  & 8.1\%     & 12.1\%     & 1.2\%   \\
				6                    & 0.2\%                      & 0.3\%                      & 0.0\%                      & 0.8\%                       & 0.9\%      & 0.0\%  & 1.6\%     & 1.9\%      & 0.1\%  & 2.3\%     & 6.7\%      & 0.4\%   \\
				\hline
			\end{tabular}
		}
	\end{center}
\end{table}

\begin{table}[H]
	\caption{Success rates of SBI-SIMEX and SBGD methods for global optimization of the Rosenbrock function of in various dimensions based on 1000 runs with uniformly generated initial position within $[-2.048, 2.048]^d$ and initial velocity within $[-1, 1]^d$. The results of SBGD are from \cite{SwarmGDRandom}.}
	\label{tab:rosen_comparison}
	\begin{center}
		\resizebox{\textwidth}{!}{
			\begin{tabular}{c ccc ccc ccc ccc}
				\hline
				\multirow{2}{*}{$d$} & \multicolumn{3}{c}{N = 10} & \multicolumn{3}{c}{N = 25} & \multicolumn{3}{c}{N = 50} & \multicolumn{3}{c}{N = 100}                                                                                           \\
				\cmidrule(lr){2-4} \cmidrule(lr){5-7} \cmidrule(lr){8-10} \cmidrule(lr){11-13}
				                     & SBI-SIMEX                  & RSBI-SIMEX                 & SBGD                       & SBI-SIMEX                   & RSBI-SIMEX & SBGD   & SBI-SIMEX & RSBI-SIMEX & SBGD   & SBI-SIMEX & RSBI-SIMEX & SBGD   \\
				\hline
				2                    & 99.9\%                     & 99.8\%                     & 10.3\%                     & 100.0\%                     & 100.0\%    & 18.7\% & 100.0\%   & 100.0\%    & 39.4\% & 100.0\%   & 100.0\%    & 56.7\% \\
				3                    & 99.5\%                     & 99.9\%                     & 2.2\%                      & 100.0\%                     & 99.3\%     & 9.6\%  & 100.0\%   & 100.0\%    & 33.9\% & 100.0\%   & 100.0\%    & 71.0\% \\
				4                    & 98.4\%                     & 93.8\%                     & 2.1\%                      & 99.5\%                      & 98.9\%     & 3.0\%  & 100.0\%   & 99.8\%     & 3.9\%  & 99.8\%    & 100.0\%    & 6.5\%  \\
				5                    & 96.4\%                     & 92.9\%                     & 0.8\%                      & 99.3\%                      & 98.6\%     & 1.6\%  & 99.1\%    & 100.0\%    & 3.2\%  & 99.6\%    & 100.0\%    & 6.1\%  \\
				6                    & 98.0\%                     & 93.1\%                     & 0.6\%                      & 99.0\%                      & 98.1\%     & 1.2\%  & 99.3\%    & 100.0\%    & 1.7\%  & 99.7\%    & 100.0\%    & 2.6\%  \\
				20                   & 92.0\%                     & 85.2\%                     & -                          & 88.5\%                      & 86.5\%     & -      & 92.2\%    & 82.9\%     & -      & 95.8\%    & 78.7\%     & -      \\
				\hline
			\end{tabular}
		}
	\end{center}
\end{table}

\begin{table}[H]
	\caption{Success rates of SBI-SIMEX and SBGD methods for global optimization of the Styblinski-Tang function in various dimensions based on 1000 runs with uniformly generated initial position within $[-3, 3]^d$ and initial velocity within $[-1, 1]^d$. The results of SBGD are from \cite{SwarmGDRandom}.}
	\label{tab:st_comparison}
	\begin{center}
		\resizebox{\textwidth}{!}{
			\begin{tabular}{c ccc ccc ccc ccc}
				\hline
				\multirow{2}{*}{$d$} & \multicolumn{3}{c}{N = 10} & \multicolumn{3}{c}{N = 25} & \multicolumn{3}{c}{N = 50} & \multicolumn{3}{c}{N = 100}                                                                                             \\
				\cmidrule(lr){2-4} \cmidrule(lr){5-7} \cmidrule(lr){8-10} \cmidrule(lr){11-13}
				                     & SBN-SIMEX                  & RSBI-SIMEX                 & SBGD                       & SBN-SIMEX                   & RSBI-SIMEX & SBGD   & SBN-SIMEX & RSBI-SIMEX & SBGD    & SBN-SIMEX & RSBI-SIMEX & SBGD    \\
				\hline
				2                    & 95.5\%                     & 96.2\%                     & 92.8\%                     & 99.9\%                      & 100.0\%    & 99.9\% & 100.0 \%  & 100.0\%    & 100.0\% & 100.0\%   & 100.0\%    & 100.0\% \\
				4                    & 56.8\%                     & 54.1\%                     & 35.3\%                     & 85.2\%                      & 86.9\%     & 79.0\% & 98.4\%    & 99.2\%     & 97.4\%  & 100.0\%   & 100.0\%    & 99.9\%  \\
				6                    & 18.5\%                     & 17.80\%                    & 10.4\%                     & 39.1\%                      & 42.3\%     & 32.5\% & 66.7\%    & 64.2\%     & 55.4\%  & 88.4\%    & 88.5\%     & 83.2\%  \\
				8                    & 5.8\%                      & 5.5\%                      & 2.5\%                      & 13.1\%                      & 11.5\%     & 9.7\%  & 23.0\%    & 24.7\%     & 18.7\%  & 47.5\%    & 45.6\%     & 35.4\%  \\
				10                   & 1.7\%                      & 0.9\%                      & 0.6\%                      & 3.0\%                       & 2.3\%      & 3.2\%  & 7.8\%     & 6.7\%      & 6.0\%   & 14.7\%    & 15.4\%     & 12.5\%  \\
				12                   & 0.6\%                      & 0.3\%                      & 0.2\%                      & 1.1\%                       & 1.3\%      & 0.8\%  & 1.6\%     & 1.8\%      & 2.2\%   & 4.0\%     & 2.9\%      & 3.8\%   \\
				\hline
			\end{tabular}
		}
	\end{center}
\end{table}

\section{Conclusion}

Based on nonequilibrium thermodynamics, we formulate the swam-based optimization problem as a minimization problem for the total mechanical energy of an initial system by coupling inertia of each agent with its potential energy given by the objective function in the optimization problem. The initial velocity of the agent and the energy dissipation rate for the mechanical energy of the agent serve as adjustable model parameters that can be adjusted to improve the search for the global optimum. The energy stable numerical approximation to the energy-dissipative system devised provides a new venue to devise efficient swarm-based algorithms.
The swarm-based inertial algorithms demonstrate strong  search capability for global optimization especially in the case when a relatively small number of agents are employed. This provides an efficient computational framework for global optimization problems in high dimension.

\section*{Acknowledgements}
Xuelong Gu's research is  supported  by
NSF award  OIA-2242812.  and Qi Wang's research is  partially supported  by
NSF awards  OIA-2242812 and DMS-2038080, DOE award DE-SC0025229, and an SC GEAR award.


\end{document}